\documentclass[a4paper, reqno]{amsart}
\usepackage{amssymb, amsmath, amsthm, bm,  enumerate, marginnote,  tikz, tikz-3dplot}
\usepackage{standalone}
\usetikzlibrary{arrows.meta, cd}

\usepackage{float}

\usepackage{makecell}
\setcellgapes{2pt}

\tdplotsetmaincoords{75}{15}

\usepackage[thinlines]{easytable}
\usetikzlibrary{patterns}

\numberwithin{equation}{section}

\newtheorem{theorem}{Theorem}[section]

\newtheorem{lemma}[theorem]{Lemma}
\newtheorem{proposition}[theorem]{Proposition}

\newtheorem{conjecture}[theorem]{Conjecture}

\theoremstyle{definition}
\newtheorem{definition}[theorem]{Definition}

\newtheorem*{preliminary definition}{Preliminary definition}

\newtheorem*{non degeneracy hypothesis}{Non-degeneracy hypothesis}
\newtheorem*{mixed hessian condition}{Mixed Hessian condition}
\newtheorem*{curvature condition}{Curvature condition}
\newtheorem*{projection condition}{Projection condition}
\newtheorem*{cone condition}{Cone condition}
\newtheorem*{rank condition}{Rank condition}

\newtheorem{example}[theorem]{Example}
\newtheorem{remark}[theorem]{Remark}
\newtheorem*{induction hypothesis}{Radial induction hypothesis}

\newcommand{\R}{\mathbb{R}}
\newcommand{\Z}{\mathbb{Z}}
\newcommand{\N}{\mathbb{N}}

\newcommand{\ud}{\mathrm{d}}

\newcommand\supp{\operatorname{supp}}
\newcommand\rank{\operatorname{rank}}

\newcommand{\e}{\varepsilon}
\newcommand{\Rn}{{\mathbb R}^n}
\newcommand{\la}{\lambda}

\usepackage{multicol}

\begin{document}

\title[Local smoothing for Fourier integral operators]{Sharp local smoothing estimates for Fourier integral operators}

\author[D. Beltran, J. Hickman and C. D. Sogge]{David Beltran, Jonathan Hickman and Christopher D. Sogge}
\address{David Beltran: BCAM - Basque Center for Applied Mathematics, Alameda de Mazarredo 14, 48009 Bilbao, Spain.}
\email{dbeltran@bcamath.org}

\address{Jonathan Hickman: Mathematical Institute, University of St Andrews, North Haugh, St Andrews, Fife, KY16 9SS, UK.}
\email{jeh25@st-andrews.ac.uk}

\address{Christopher D. Sogge: Department of Mathematics, Johns Hopkins University, Baltimore, MD 21218, USA.}
\email{sogge@jhu.edu}

\subjclass[2010]{Primary: 35S30, Secondary: 35L05}
\keywords{Local smoothing, variable coefficient, Fourier integral operators, decoupling inequalities.}


\thanks{ D.B. was supported by: the ERCEA Advanced Grant 2014 669689 - HADE, the MINECO project MTM2014-53850-P, the Basque Government project IT-641-13, the Basque Government through the BERC 2018-2021 program and by the Spanish Ministry of Science, Innovation and Universities: BCAM Severo Ochoa accreditation SEV-2017-0718. J.H. was supported by NSF Grant DMS-1440140 and  EPSRC standard grant EP/R015104/1. C.D.S. was supported by NSF Grant DMS-1665373 and a Simons Fellowship.}

\begin{abstract} The theory of Fourier integral operators is surveyed, with an emphasis on local smoothing estimates and their applications. After reviewing the classical background, we describe some recent work of the authors which established sharp local smoothing estimates for a natural class of Fourier integral operators. We also show how local smoothing estimates imply oscillatory integral estimates and obtain a maximal variant of an oscillatory integral estimate of Stein. Together with an oscillatory integral counterexample of Bourgain, this shows that our local smoothing estimates are sharp in odd spatial dimensions. Motivated by related counterexamples, we formulate local smoothing conjectures which take into account natural geometric assumptions arising from the structure of the Fourier integrals.
\end{abstract}

\maketitle

\tableofcontents




\section{Basic definitions and examples of Fourier integral operators}\label{section: introduction}




\subsection{Motivating examples}\label{motivating examples section} This article explores aspects of the theory of \textit{Fourier integral operators} (FIOs), a rich class of objects which substantially generalises the class of pseudo-differential operators. The genesis of the theory can be found in various early works on hyperbolic equations \cite{Eskin1967, Egorov1969, Hormander1968, Lax1957, Maslov} but for the purposes of this article the study of FIOs began in earnest in the groundbreaking treaties of H\"ormander \cite{Hormander1971} and Duistermaat--H\"ormander \cite{Duistermaat1972}. 

For the majority of this discussion it will suffice to work with the following definition of a FIO, although below a more general and robust framework is recalled.

\begin{preliminary definition} A \textit{Fourier integral operator} (or FIO) $\mathcal{F}$ \textit{of order} $\mu \in \R$ is an operator, defined initially on the space of Schwartz functions $\mathcal{S}(\R^n)$, of the form 
\begin{equation}\label{preliminary FIO}
    \mathcal{F}f(x) := \frac{1}{(2\pi)^n}\int_{\hat{\R}^n}e^{i\phi(x;\xi)} a(x;\xi) \hat{f}(\xi)\,\ud \xi
\end{equation}
where
\begin{itemize}
    \item The \emph{phase} $\phi \colon \R^n \times \R^n \to \R$ is homogeneous of degree 1 in $\xi$ and smooth away from $\xi = 0$ on the support of $a$.
    \item The \emph{amplitude} $a \colon \R^n \times \R^n \to \R$ belongs to the symbol class $S^{\mu}$; that is, $a$ is smooth away from $\xi = 0$ and satisfies
    \begin{equation*}
    |\partial_x^{\beta} \partial_{\xi}^{\alpha} a(x;\xi)| \lesssim_{\alpha, \beta} (1+|\xi|)^{\mu - |\alpha|} \qquad \textrm{for all $(\alpha,\beta) \in \N_0^n \times \N_0^n$.}
    \end{equation*}
\end{itemize}
\end{preliminary definition}



Taking $\phi(x;\xi) := \langle x, \xi \rangle$, one immediately recovers the class of pseudo-differential operators associated to standard symbols (that is, symbols belonging to some class $S^{\mu}$). For the purposes of this article this is a somewhat trivial case, however, and it is constructive to consider some more representative examples of FIOs. 

\begin{example}\label{wave propagator example} Prototypical FIOs arise from the (euclidean) half-wave propagator, defined by 
\begin{equation}\label{wave propagator}
    e^{it\sqrt{-\Delta}}f(x) := \frac{1}{(2 \pi)^n}\int_{\hat{\R}^n} e^{i(\langle x, \xi \rangle + t|\xi|)} \hat{f}(\xi)\,\ud \xi.
\end{equation}
Under suitable regularity hypotheses on $f_0$ and $f_1$, if $f_+:=\frac{1}{2}(f_0 - i( \sqrt{-\Delta})^{-1}f_1)$ and $f_-:=\frac{1}{2}(f_0 + i (\sqrt{-\Delta})^{-1}f_1)$, \footnote{In general, $m(i^{-1}\partial_x)$ denotes the Fourier multiplier operator (defined for $f$ belonging to a suitable \emph{a priori} class)
\begin{equation*}
m(i^{-1}\partial_x)f(x) := \frac{1}{(2\pi)^n}\int_{\hat{\R}^n} e^{i \langle x, \xi \rangle} m(\xi)  \hat{f}(\xi)\,\ud \xi
\end{equation*}
for any $m \in L^{\infty}(\hat{\R}^n)$. The operator $m(\sqrt{-\Delta_x})$ is then defined in the natural manner via the identity $-\Delta_x = i^{-1}\partial_x \cdot i^{-1}\partial_x$.} then the function 
\begin{equation*}
    u(x,t):= e^{i t \sqrt{-\Delta}} f_+(x) + e^{-it \sqrt{-\Delta}}f_-(x)
\end{equation*}
solves the Cauchy problem
\begin{equation}\label{euclidean wave equation}
    \left\{ \begin{array}{l}
        (\partial_t^2 - \Delta)u(x,t) = 0  \\[3pt]
        u(x,0)  := f_0(x), \quad \partial_t u(x,0)  := f_1(x) 
    \end{array}
    \right. .
\end{equation}
Up to a constant multiple, each term in the expression for $u(x,t)$ is of the form
\begin{equation*}
    \mathcal{F}_tf(x) := \frac{1}{(2\pi)^n} \int_{\hat{\R}^n} e^{i(\langle x, \xi \rangle \pm t|\xi|)} |\xi|^{-j}\hat{f}(\xi) \,\ud \xi
\end{equation*}
 for either $j=0$ or $j=1$. These operators provide important examples of FIOs of order $-j$. Indeed, much of the motivation for the development of the theory of FIOs was to provide an effective counterpart to the theory of pseudo-differential operators to study hyperbolic, rather than elliptic, PDE, a fundamental example being the wave equation \eqref{euclidean wave equation}. The reader is referred to the original papers \cite{Hormander1971, Duistermaat1972} and the classical texts \cite{HormanderIV} and \cite{Duistermaat2011} for further discussion in this direction.
\end{example} 

\begin{example}\label{wave equation on a manifold example} One may also consider wave propagators on other Riemannian manifolds $(M,g)$, defined with respect to the Laplace--Beltrami operator $\Delta_g$. In particular, suppose $(M,g)$ is a compact $n$-dimensional Riemannian manifold, in which case $-\Delta_g$ has a discrete, positive spectrum which may be ordered $0 = \lambda_0^2 < \lambda_1^2 \leq \lambda_2^2 \leq \dots$ (here the eigenvalues are enumerated with multiplicity). Thus, one may write $-\Delta_g = \sum_{j=0}^{\infty} \lambda_j^2 E_j$ where each $E_j$ is the orthogonal projection in $L^2(M)$ onto a 1-dimensional eigenspace associated to the eigenvalue $\lambda_j^2$. For proofs of these facts see, for instance, \cite{Sogge2014, Zelditch}.

Now consider the half-wave propagator 
\begin{equation}\label{manifold propagator}
e^{it\sqrt{-\Delta_g}}f(x) := \sum_{j=0}^{\infty} e^{it\lambda_j} E_jf(x).
\end{equation}
If $u$ is defined as in the previous example (but now the initial data $f_0$, $f_1$ is defined on $M$ and the multipliers are interpreted in terms of the spectral decomposition), then this function solves the Cauchy problem
\begin{equation}\label{wave equation on a manifold}
    \left\{ \begin{array}{l}
        (\partial_t^2 - \Delta_g)u(x,t) = 0  \\[3pt]
        u(x,0)  := f_0(x), \quad \partial_t u(x,0)  := f_1(x) 
    \end{array}
    \right. .
\end{equation}

In local coordinates, one may construct a parametrix for the propagator \eqref{manifold propagator} which is of the form of a Fourier integral operator. In particular, for some $t_0 > 0$ one may write
\begin{equation*}
    e^{it\sqrt{-\Delta_g}}f(x) = \int_{\hat{\R}^n} e^{i \phi(x;t;\xi)} a(x;t;\xi) \hat{f}(\xi) \,\ud \xi + R_tf(x) \qquad \textrm{for all $0 < t < t_0$}
\end{equation*}
for some suitable choice of phase $\phi$ and 0-order symbol $a$, where $R_t$ is a smoothing operator (that is, a pseudo-differential operator with rapidly decaying symbol). Here the Fourier transform of $f$ is taken in the euclidean sense, in the chosen coordinate domain. This construction is a special case of a general result concerning \emph{strictly hyperbolic} equations (of arbitrary order) which dates back to Lax \cite{Lax1957}; further discussion can be found in \cite[Chapter 5]{Duistermaat2011} or \cite[Chapter 4]{Sogge2017}.
\end{example}

\begin{example}\label{averaging operator example} Closely related to the wave propagator \eqref{wave propagator} are the convolution operators
\begin{equation*}
    A_t f(x) := f \ast \sigma_t(x), \qquad t > 0
\end{equation*}
where $\sigma = \sigma_1$ is the surface measure on the unit sphere $\mathbb{S}^{n-1}$ and $\sigma_t$ is defined by
\begin{equation*}
    \int_{\R^n} f(x) \,\ud \sigma_t(x) := \int_{\R^n} f(tx) \,\ud \sigma(x).
\end{equation*}
When $n = 3$ the solution to \eqref{euclidean wave equation} at time $t$ is related to $A_t$ via the classical Kirchhoff formula (see, for instance, \cite[Chapter 1]{Sogge2008}). These averaging operators are also of significant interest in harmonic analysis and, in particular, the spherical maximal function of Stein \cite{Stein1976} and Bourgain \cite{Bourgain1986} is defined by $Mf(x) := \sup_{t > 0} A_t|f|(x)$. 

To see how such averages fall into the Fourier integral framework, recall that the method of stationary phase (see, for instance, \cite[Chapter VIII]{Stein1993} or \cite[Chapter 1]{Sogge2017}) yields the formula
\begin{equation*}
    \hat{\sigma}(\xi) := \sum_{\pm} e^{\pm i|\xi|} a_{\pm}(\xi)
\end{equation*}
for the Fourier transform of the measure $\sigma$, where $a_{\pm} \in S^{-(n-1)/2}$ are smooth symbols of order $-(n-1)/2$. Thus, one may write
\begin{equation*}
    A_tf(x) = \sum_{\pm} \frac{1}{(2\pi)^n}\int_{\hat{\R}^n} e^{i(\langle x, \xi \rangle \pm t|\xi|)} a_{\pm}(t\xi)\hat{f}(\xi)\,\ud \xi;
\end{equation*}
note that the operators appearing in this formula agree with those arising in Example \ref{wave propagator example} except for the choice of symbol. 
\end{example}




\subsection{Distributions defined by oscillatory integrals}\label{oscillatory integrals section} 

The remainder of this section will be dedicated to describing a more general framework for the study of FIOs. For much of this article the preliminary definition given in the preceding subsection is sufficient; the refined definitions are included here in order to relate this survey to the perspective espoused in many of the references, and in particular in the classical works \cite{Hormander1971, Duistermaat1972}. 

In contrast with the discussion in the previous subsection, here the operators will be defined in terms of a kernel. Formally, the kernel of the FIO in \eqref{preliminary FIO} is given by
\begin{equation*}
    K(x;y) := \frac{1}{(2\pi)^n}\int_{\hat{\R}^n}e^{-i(\langle y, \xi \rangle - \phi(x;\xi))} a(x;\xi) \,\ud \xi,
\end{equation*}
although without strong conditions on the symbol $a$ this integral is not defined in any classical sense. To give precise meaning to such expressions one appeals to the theory of distributions; the relevant concepts from this theory are reviewed presently.

\subsubsection*{Distributions} Given $W \subseteq \R^d$ open, let $\mathcal{D}(W)$ denote the space of test functions on $W$; that is, $\mathcal{D}(W)$ is the space $C_c^{\infty}(W)$ of $C^{\infty}$ functions with compact support in $W$ under the topology defined by $f_j \to f$ as $j \to \infty$ for $f_j, f \in \mathcal{D}(W)$ if
\begin{enumerate}[i)]
\item There exists a compact set $K\subset W$ containing  $\mathrm{supp}\, f$ and $\mathrm{supp}\, f_j$ for all $j \in \N$;
\item $\partial^{\alpha}_x f_j \to \partial^{\alpha}_x f$ uniformly as $j \to \infty$ for all $\alpha \in \N_0^d$.
\end{enumerate}
One then defines the space of distributions $\mathcal{D}'(W)$ on $W$ to be the dual topological vector space to $\mathcal{D}(W)$ endowed with the weak$^*$ topology. With this definition $\mathcal{D}'(W)$ is complete.\footnote{Here a sequence $(u_j)_{j=1}^n \subseteq \mathcal{D}'(W)$ is Cauchy if $(\langle u_j,f\rangle)_{j=1}^{\infty}$ is a Cauchy sequence of complex numbers for all $f \in \mathcal{D}(W)$.} 


\subsubsection*{Homogeneous oscillatory integrals} Now let $\varphi \colon W \times (\R^N\setminus \{0\}) \to \R$ be a smooth function, $a \in S^{\mu}(W \times \R^N)$ and consider the oscillatory integral formally defined by
\begin{equation}\label{oscillatory integral}
    I[\varphi,a](w) := \int_{\hat{\R}^N} e^{i \varphi(w;\theta)} a(w;\theta)\,\ud \theta \qquad \textrm{for $w \in W$}.
\end{equation}
If $\mu < -N$, then this integral converges absolutely and, moreover, the resulting function of $w$ defines a distribution on $W$ (by integrating $I[\varphi,a]$ against a given test function). If $\mu > -N$, then it is not clear that the expression \eqref{oscillatory integral} makes sense and additional hypotheses are required on $\varphi$ to give the integral meaning. In particular, suppose that 
\begin{equation}\label{phase condition 1}
    \textrm{$\varphi$ is homogeneous of degree 1 in $\theta$}
\end{equation}
and
\begin{equation}\label{phase condition 2}
    \textrm{$\nabla_{w,\theta}\varphi(w;\theta) \neq 0$ for all $(w;\theta) \in W \times (\R^N \setminus \{0\})$}
\end{equation}
where the gradient $\nabla_{w,\theta}$ is taken with respect to all the variables $(w; \theta)$. Now let $\beta \in C^{\infty}_c(\R^N)$ satisfy $\beta(0) = 1$ and consider the truncated integral
\begin{equation*}
    I^j[\varphi,a](w) := \int_{\hat{\R}^N} e^{i \varphi(w;\theta)} \beta(2^{-j} \theta)a(w;\theta)\,\ud \theta.
\end{equation*}
Each $I^j[\varphi,a]$ is a well-defined function which induces a distribution. Moreover, under the conditions \eqref{phase condition 1} and \eqref{phase condition 2}, a simple integration-by-parts argument allows one to deduce that for any $K \subseteq W$ compact
\begin{equation*}
    |\langle I^j[\varphi,a], f \rangle| \leq C_K2^{-j} \sum_{|\alpha| \leq k} \|\partial_x^{\alpha}f\|_{L^{\infty}(K)} \qquad \textrm{for all $f \in \mathcal{D}(W)$ with $\mathrm{supp}\,f \subseteq K$}
\end{equation*}
where $k$ satisfies $\mu < -N + k$ (see, for instance, \cite[Theorem 0.5.1]{Sogge2017}). Since $\mathcal{D}'(W)$ is complete, one may therefore define $I[\varphi,a]$ to be the distribution given by the limit of the sequence of distributions $I^j[\varphi,a]$.

\begin{definition} The distribution $I[\varphi,a]$, defined for $\varphi$ satisfying \eqref{phase condition 1} and \eqref{phase condition 2}, will be referred to as  \emph{(local)\footnote{The terminology \textit{local}, as opposed to \textit{global}, will become clearer in $\S$\ref{global section}.} homogeneous oscillatory integral}. By a slight abuse of notation, the distribution $I[\varphi, a]$ will also be denoted by the formal expression  \eqref{oscillatory integral}.
\end{definition}

In what follows, it will be useful to assume a further condition on the phase $\varphi$.

\begin{non degeneracy hypothesis} A smooth function $\varphi \colon W \times (\R^N \setminus \{0\}) \to \R$ satisfying \eqref{phase condition 1} and \eqref{phase condition 2} is a \emph{non-degenerate phase function} if, in addition, it satisfies
\begin{equation}\label{phase condition 3}
    \textrm{if $\partial_{\theta}\varphi(w;\theta) = 0$, then $\bigwedge_{j=1}^N \nabla_{w,\theta} \partial_{\theta_j}\varphi(w;\theta) \neq 0$.}
\end{equation}
\end{non degeneracy hypothesis}

The rationale behind this additional hypothesis will become apparent in $\S$\ref{invariance of phase section}.




\subsection{Local Fourier integral operators} For $X \subseteq \R^n$ and $Y \subseteq \R^m$ open, any distribution $K \in \mathcal{D}'(X \times Y)$ defines a natural continuous linear mapping $T \colon \mathcal{D}(Y) \to \mathcal{D}'(X)$ given by
\begin{equation}\label{Schwartz kernel}
    \langle T(f), g \rangle := \langle K, f \otimes g \rangle \qquad \textrm{for all $(f,g) \in \mathcal{D}(Y) \times \mathcal{D}(X)$.}
\end{equation}
In fact, a converse to this observation also holds, which is the content of the celebrated Schwartz kernel theorem (see, for instance, \cite[\S 5.2]{HormanderI}). In particular, given any continuous linear mapping $T \colon \mathcal{D}(Y) \to \mathcal{D}'(X)$ there exists a unique distribution $K\in \mathcal{D}'(X \times Y)$, referred to as the \emph{(Schwartz) kernel} of $T$, such that \eqref{Schwartz kernel} holds. 

\begin{definition} A continuous linear operator $\mathcal{F} \colon C^{\infty}_c(Y) \to \mathcal{D}'(X)$ is a \emph{(local) Fourier integral operator} if the Schwartz kernel is given by a homogeneous oscillatory integral $I[\varphi, a]$ for some non-degenerate phase function $\varphi \colon X \times Y \times \R^N \setminus \{0\} \to \R$ and amplitude $a \in S^{\mu}(X \times Y \times \R^N)$.
\end{definition}

Given a test function $f \in C^{\infty}_c(Y)$, by an abuse of notation the distribution $\mathcal{F}f$ will also be denoted by
\begin{equation}\label{FIO kernel definition}
    \mathcal{F}f(x) = \int_{\R^n} \int_{\R^N} e^{i \varphi(x;y;\theta)} a(x;y;\theta) \,\ud\theta\, f(y)\,\ud y.
\end{equation}

\begin{example}\label{averaging operator example kernel} The averaging operator from Example \ref{averaging operator example} can be expressed as
\begin{equation}\label{averaging operator n Fourier variables}
    A_tf(x) = \sum_{\pm} \frac{1}{(2\pi)^n} \int_{
    \R^n}\int_{\hat{\R}^n} e^{i(\langle x-y, \xi \rangle \pm t|\xi|)} a_{\pm}(t\xi)\,\ud \xi\,f(y)\,\ud y,
\end{equation}
where this formal expression is interpreted in the above distributional sense. Note that the phase $\varphi(x;y;\xi) := \langle x-y, \xi \rangle \pm t|\xi|$ satisfies the desired conditions \eqref{phase condition 1}, \eqref{phase condition 2} and \eqref{phase condition 3}.
\end{example}

There are significant short-comings in defining FIOs in this way. In particular, there are fundamental problems regarding uniqueness: a given operator will admit many distinct representations of the form \eqref{FIO kernel definition}. 

\begin{example}\label{alternative averaging operator example} Once again recall the operator $A_tf$ from Example \ref{averaging operator example}, which can be interpreted as taking an average of $f$ over the sphere $x + t\mathbb{S}^{n-1}$. For fixed $x, t$, this surface corresponds to the zero locus of the defining function
\begin{equation*}
    \Phi(x;t;y) := \frac{|x-y|^2}{t^2} - 1.
\end{equation*}
This allows one to rewrite the averaging operator as
\begin{equation*}
    A_tf(x) := \int_{\R^n} f(y)a(x;t;y) \,\delta(\Phi(x;t;y))\ud y 
\end{equation*} 
 where $\delta(\Phi(x;t;y))\ud y $ is the normalised induced Lebesgue measure on $x + t\mathbb{S}^{n-1}$ (see, for instance, \cite[Chapter XI, \S 3.1.2]{Stein1993} or \cite[Chapter VIII, \S3]{Stein2011}) and $a(x;t;y)$ is an appropriate choice of bump function. Using the heuristic identity
 \begin{equation*}
     \delta(x) = \frac{1}{2\pi} \int_{\R} e^{i \lambda x}\,\ud \lambda
 \end{equation*}
for the Dirac $\delta$ function, this leads to the expression
 \begin{equation}\label{one Fourier variable}
    A_tf(x) = \frac{1}{2\pi}\int_{\R^n} \int_{\hat{\R}} e^{i \lambda \Phi(x;t;y)}  a(x;t;y)\,\ud\lambda\, f(y)\,\ud y.
  \end{equation} 
Thus, one arrives at an alternative Fourier integral representation of the average $A_tf$ to that in \eqref{averaging operator n Fourier variables}. Although this argument has been presented as a heuristic, it is not difficult to make the details precise, provided \eqref{one Fourier variable} is interpreted correctly (that is, as converging in the sense of oscillatory integrals); the full details can be found, for instance, in \cite[Chapter XI]{Stein1993}.
\end{example}




\subsection{Wave front sets and equivalence of phase}\label{invariance of phase section} Examples \ref{averaging operator example kernel} and \ref{alternative averaging operator example} show that very different phase/amplitude pairs $[\varphi, a]$ can define the same Fourier integral operator. It is natural to ask whether one can formulate a ``coordinate-free'' or ``invariant'' definition of FIOs which does not rely on fixing a choice of phase and amplitude. Such a \emph{global} definition does indeed exist and is discussed in detail in \cite{Hormander1971, Duistermaat1972} as well as the texts \cite{Duistermaat2011, Sogge2017, Treves1980}. The full details of the global theory of Fourier integral operators is, however, beyond the scope of this article, but nevertheless here some of the basic ideas are presented. 

To arrive at a global definition of Fourier integral operators, it is necessary to analyse the geometry of the singularities of the underlying Schwartz kernel. This leads to the construction of a geometric object known as the \textit{canonical relation} for a given FIO, which is in some sense independent of the choice of phase function used to define the kernel (this is the content of H\"ormander's equivalence of phase theorem, discussed below). The idea is then to think of the FIO purely in terms of the canonical relation (and some order), without reference to a particular choice of pair $[\varphi, a]$.

To carry out the above programme, some basic definitions from microlocal analysis (which may be described as the geometric study of distributions) are required. 


\subsubsection*{The singular support} Once again, let $W \subseteq \R^d$ be some fixed open set. 

\begin{definition} The \emph{singular support} $\textrm{sing supp}\,u$ of $u \in \mathcal{D}'(W)$ is defined to be the set of points in $w \in W$ for which there exists no open neighbourhood upon which $u$ agrees with a $C^{\infty}$ function.
\end{definition}

The singular support identifies the location of the singularities of a distribution, but for a complete geometric description one must also understand the associated ``directions'' of the singularities. 

\begin{example}\label{measure on curve example} Let $\gamma \colon [0,1] \to \R^2$ be (a parametrisation of) a smooth curve and let $\mu$ be a smooth density on $\gamma$ in $\R^2$, viewed as a measure (and therefore a distribution) on the plane. In particular, there exists some $a \in C^{\infty}_c(\R)$ with support in $(0,1)$ such that $\mu$ is defined by
\begin{equation*}
    \int_{\R^2} f \ud \mu := \int_0^1 (f\circ \gamma)(t) a(t)\,\ud t \qquad \textrm{for all $f \in C(\R^2)$.}
\end{equation*}
It is immediate that the singular support of $\mu$ consists of the support of the measure (and is therefore a subset of the curve). Given $x_0 \in \mathrm{supp}\,\mu$, one expects that the singular direction should the normal to the curve at $x_0$. A rigorous formulation of this intuitive statement is discussed below. 
\end{example}

To identify the singular directions, one appeals to the correspondence between regularity and the decay of the Fourier transform. For instance, recall that any $f \in C_c^{\infty}(\R^n)$ satisfies
\begin{equation}\label{rapid decay of Fourier transform}
    |\hat{f}(\xi)| \lesssim_N (1 + |\xi|)^{-N} \qquad \textrm{for all $N \in \N$}
\end{equation}
 for $\xi \in \hat{\R}^n$ (and, moreover, the property $f \in C_c^{\infty}(\R^n)$ is completely characterised in terms of the decay of the Fourier--Laplace transform by the Paley--Weiner theorem \cite[\S 7.3]{HormanderI}). If $u \in \mathcal{D}'(W)$ and $w \notin \textrm{sing supp}\,u$, then it follows that there exists some $\psi \in C^{\infty}_c(W)$ satisfying $\psi(w) \neq 0$ for which $f := \psi u$ is a $C_c^{\infty}$ function satisfying \eqref{rapid decay of Fourier transform}. Given $w \in \textrm{sing supp}\,u$ and $\psi \in C^{\infty}_c(W)$ satisfying $\psi(w) \neq 0$ as above, the idea is now to analyse the directions in which \eqref{rapid decay of Fourier transform} fails for $f := \psi u$. Since in this case $\psi u$ is no longer guaranteed to be a function, the precise definition requires some facts about Fourier transforms of distributions, which are recalled presently. 


\subsubsection*{Tempered distributions and the Fourier transform} If $\mathcal{S}(\R^n)$ denotes the Schwartz space, then recall that the space of \emph{tempered distributions} is the dual $\mathcal{S}'(\R^n)$, which can be identified with a subspace of $\mathcal{D}'(\R^n)$. The Fourier transform $\hat{u} \in \mathcal{S}'(\R^n)$ of a tempered distribution $u \in \mathcal{S}'(\R^n)$ is defined by $\langle \hat{u}, f \rangle := \langle u, \hat{f} \rangle$ for all $f \in \mathcal{S}(\R^n)$. If $u \in \mathcal{D}'(W)$ is compactly-supported (in the sense that there exists a compact set $K \subset W$ such that $\langle u, f \rangle = 0$ whenever $f \in \mathcal{D}(W)$ is supported outside $K$), then $u$ automatically extends to a tempered distribution and, moreover, the Fourier transform $\hat{u}$ is a $C^{\infty}$ function which grows at most polynomially. For proofs of these facts see, for instance, \cite{HormanderI}.


\subsubsection*{The wave front set} Let $u \in \mathcal{D}'(W)$ be compactly supported. Define $\Gamma(u)$ to be the set of points $\eta \in \hat{\R}^n\setminus \{0\}$ for which there does not exist an open conic neighbourhood $\mathcal{C}$ upon which 
\begin{equation*}
    |\hat{u}(\xi)| \lesssim_N (1 + |\xi|)^{-N} \qquad \textrm{for all $N \in \N$ and all $\xi \in \mathcal{C}$.}
\end{equation*}
Given $u \in \mathcal{D}'(W)$ and $w \in W$ one then defines
\begin{equation*}
    \Gamma_w(u) := \bigcap_{\substack{\psi \in C^{\infty}_c(W)\\ \psi(w) \neq 0}} \Gamma(\psi u).
\end{equation*}
By the discussion following \eqref{rapid decay of Fourier transform}, it is clear that if $\Gamma_w(u) \neq \emptyset$, then $w \in \textrm{sing supp}\,u$.
\begin{definition} [Wave front set] If $u \in \mathcal{D}'(W)$, then the \emph{wave front set} $\mathrm{WF}(u)$ of $u$ is defined
\begin{equation*}
        \mathrm{WF}(u) := \big\{ (w;\xi) \in W \times (\R^d \setminus \{0\}) : \xi \in \Gamma_w (u) \big\}.
\end{equation*}
\end{definition}

\begin{example} Returning to the measure $\mu$ discussed in Example \ref{measure on curve example}, fix $x_0 \in \textrm{sing supp}\,\mu$ so that $x_0 = \gamma(t_0)$ for some $t_0 \in \mathrm{supp}\,a$. Suppose $\eta \in \hat{\R}^2 \setminus \{0\}$ does not lie in the linear subspace $N_{x_0}\gamma$ spanned by the normal to $\gamma$ at $x_0$. It is not difficult to show that there exists a conic neighbourhood $\mathcal{C}$ of $\eta$ and some $\varepsilon_0 > 0$ such that 
\begin{equation*}
    |\langle \xi, \gamma'(t) \rangle | \geq \varepsilon_0|\xi| \qquad \textrm{for all $\xi \in \mathcal{C}$ and $|t - t_0| < \varepsilon_0$.}
\end{equation*}
If $\psi \in C^{\infty}_c(\R^2)$ is chosen to have support in a sufficiently small neighbourhood of $x_0$, it therefore follows by non-stationary phase (that is, repeated integration-by-parts) that $|(\psi \mu)\;\widehat{}\;(\xi)| \lesssim_N (1 + |\xi|)^{-N}$ for all $N\in \N$ and $\xi \in \mathcal{C}$ and, consequently, $\Gamma_{x_0}(\mu) \subseteq  N_{x_0}\gamma$. On the other hand, if $\eta \in N_{x_0}\gamma \cap \mathbb{S}^1$ and $\psi \in C^{\infty}_c(\R^2)$ satisfies $\psi(x_0) \neq 0$, then the asymptotic expansion for oscillatory integrals (see, for instance, \cite[Chapter VIII]{Stein1993}) shows that $|(\psi \mu)\;\widehat{}\;(\lambda \eta)|$ fails to decay rapidly in $\lambda \geq 1$.
\end{example}

For a homogeneous oscillatory integral $I[\varphi, a]$, as defined in $\S$\ref{oscillatory integrals section}, it is not difficult to show that the wave front set of this distribution is contained in
\begin{equation*}
\Lambda_{\varphi} := \{(w, \partial_w\varphi(w; \theta)) \in W \times (\R^d \setminus \{0\}) : (w; \theta) \in \textrm{supp}\,a,\: \partial_{\theta}\varphi(w;\theta) = 0\}.
\end{equation*}
Indeed, as a rough sketch of why this should hold, taking the Fourier transform of $I[\varphi, a]\cdot \psi$ for some $\psi \in C^{\infty}_c(\R^d)$ yields an oscillatory integral in the $(w; \theta)$ variables with phase $\varphi(w;\theta) - \langle w, \xi \rangle$. By non-stationary phase (that is, integration-by-parts), one expects rapid decay away from the set of $(w;\xi)$ for which the phase admits a $(w;\theta)$-stationary point. Since the $w$-gradient of the phase is given by $\partial_w\varphi(w;\theta) - \xi$ and the $\theta$-gradient is $\partial_{\theta}\varphi(w;\theta)$, this naturally leads one to consider the set $\Lambda_{\varphi}$. The full details can be found, for instance, in \cite[Theorem 8.1.9]{HormanderI}.


\subsubsection*{Non-degeneracy and Lagrangian manifolds} If the phase function $\varphi$ is non-degenerate in the sense that \eqref{phase condition 3} holds, then it follows from the implicit function theorem that
\begin{equation*}
    \Sigma_{\varphi} := \big\{ (w, \theta) \in W \times (\R^N \setminus \{0\}) : \partial_{\theta}\varphi(w, \theta) = 0 \big\}
\end{equation*}
is a smooth $d$-dimensional submanifold. Moreover, one may readily verify that the map $\kappa \colon \Sigma_{\varphi} \to W \times (\R^d \setminus \{0\})$ given by $\kappa(w,\theta) := (w, \partial_w\varphi(w, \theta))$ is an immersion with image
$\Lambda_{\varphi}$. Typically, in this situation one identifies  $W \times (\R^d\setminus \{0\})$ with $T^*W\setminus 0$, the cotangent bundle of $W$ with the zero section removed. The rationale behind this is that, under the above identification, $\Lambda_{\varphi}$ has a special property defined with respect to the natural symplectic structure on $T^*W$.  Concepts from symplectic geometry form an important part of the analysis of FIOs, but will only be mentioned in passing here (see, for instance, \cite{Duistermaat2011} for a thorough introduction to symplectic differential geometry and its connection to Fourier integral theory).

\begin{definition} A smooth (immersed) submanifold $\Lambda \subseteq T^*W\setminus 0$ is \emph{conic} if it is conic in the fibres: that is, $(w, t\xi) \in \Lambda$ for all $t > 0$ whenever $(w, \xi) \in \Lambda$. Such a $\Lambda$ is a \emph{Lagrangian submanifold} if it is $d$-dimensional and the restriction of the \emph{canonical 1-form} $\omega := \sum_{j=1}^d  \xi_j \ud w_j$ on $T^*W$ to $\Lambda$ is identically zero. 
\end{definition}

It is not difficult to show that $\Lambda_{\varphi}$ is a (conic) Lagrangian submanifold.\footnote{The pull-back $\kappa^*\omega$ of $\omega$ is given by
\begin{equation*}
    \sum_{j=1}^d\partial_{w_j}\varphi(w,\theta)\ud w_j = \ud \phi - \sum_{i=1}^N\partial_{\theta_i}\varphi(w,\theta)\ud w_j.
\end{equation*}
This vanishes identically on $\Lambda_{\varphi}$ since $\partial_{\theta}\varphi(w, \theta) = 0$ and the homogeneity of $\varphi$ with respect to $\theta$ implies $\varphi(w,\theta) = \theta \partial_\theta \varphi (w, \theta) =0$.} Conversely, given any conic Lagrangian submanifold $\Lambda \subset T^*W\setminus 0$, one can show that locally $\Lambda$ agrees with $\Lambda_{\varphi}$ for some non-degenerate phase function $\varphi$  (see, for instance, \cite[\S 3.1]{Hormander1971}). 


\subsubsection*{Equivalence of phase} The correspondence between conic Lagrangian submanifolds and non-degenerate phase functions described above is clearly not unique: for instance, one may compose the phase function $\varphi$ with any fibre-preserving diffeomorphism $(w, \theta) \mapsto (w, \tilde{\theta}(w, \theta))$ to obtain a new phase function $\tilde{\varphi}$ which satisfies $\Lambda_{\varphi} = \Lambda_{\tilde{\varphi}}$. However, in this case it follows by the change of variables formula that $I[\varphi, a] = I[\tilde{\varphi}, \tilde{a}]$ where $\tilde{a}(w, \tilde{\theta}(w, \theta)) = a(w, \theta) |\partial_{\theta}\tilde{\theta}(w, \theta)|^{-1}$. Thus, provided the symbols are appropriately defined, the phases $\varphi$ and $\tilde{\varphi}$ define the same homogeneous oscillatory integral. 

Now suppose $\varphi$, $\tilde{\varphi}$ are two phase functions which satisfy $\Lambda_{\varphi} = \Lambda_{\tilde{\varphi}}$, but are not necessarily related by a fibre-preserving diffeomorphism. What can be said about the homogeneous oscillatory integrals in this case? A typical example of this situation has already appeared above. 

\begin{example}\label{spooky coincidence example} Fixing $t \in \R\setminus\{0\}$, consider the phase function $\varphi_t \colon \R^n \times \R^n \times (\R^n\setminus \{0\}) \to \R$ featured in Example \ref{wave propagator example} and Example \ref{averaging operator example}, given by $\varphi_t(x,y;\xi) := \langle x - y, \xi \rangle + t | \xi|$. Then a simple computation shows that
\begin{equation*}
    \Lambda_{\varphi_t} = \big\{(x,x + t \frac{\xi}{|\xi|},\xi,-\xi) : (x,\xi) \in \R^n\times \R^n \setminus\{0\}\big\}.
\end{equation*}
Now consider the phase function $\tilde{\varphi}_t \colon \R^n \times \R^n \times (\R \setminus \{0\}) \to \R$ given by 
\begin{equation*}
    \tilde{\varphi}_t(x,y;\lambda) := \lambda\Big(\frac{|x-y|^2}{t^2} - 1 \Big),
\end{equation*}
which is featured in the alternative representation for the averaging operator from Example \ref{alternative averaging operator example}. A simple computation shows that
\begin{equation*}
    \Lambda_{\tilde{\varphi}_t} = \Big\{\Big(x,y,2\lambda \frac{x-y}{t^2},-2\lambda \frac{x-y}{t^2}\Big) :  |x - y| = t\Big\}.
\end{equation*}
However, making the substitution $\xi = 2\lambda \frac{x-y}{t^2}$, this set agrees precisely with that in Example \ref{spooky coincidence example}.

Note that the fibres of $\varphi_t$ and $\tilde{\varphi}_t$ have different dimensions so clearly the two phases cannot be related via a fibre-preserving change of variables. 
\end{example}

It is still true that, for suitable choices of amplitude function, the phases discussed in Example \ref{spooky coincidence example} define the same homogeneous oscillatory integral (indeed, both phases are used to represent the same averaging operator $A_t$). These observations suggest the possibility of a \emph{unique} correspondence between conic Lagrangian manifolds $\Lambda$ and homogeneous oscillatory integrals. This correspondence is the subject of the following fundamental result of H\"ormander \cite{Hormander1971}.

\begin{theorem}[H\"ormander's equivalence of phase theorem \cite{Hormander1971}]\label{thm:inv phases} Suppose $\varphi$ and $\tilde{\varphi}$ are non-degenerate phase functions defined on a neighbourhood of $(w_0, \theta_0) \in W \times \R^N$ and $(w_0, \tilde{\theta}_0) \in W \times \R^{\tilde{N}}$, respectively, which define the same Lagrangian submanifold there.

Then every homogeneous oscillatory integral $I[\varphi,a]$ with $a \in S^{\mu + (d-2N)/4}(W \times \R^N)$ and $\mathrm{supp}\,a$ in a sufficiently small $\theta$-conic neighbourhood\footnote{Here a $\theta$-conic neighbourhood of $(w_0,\theta_0) \in W \times \R^N \setminus \{0\}$ is an open neighbourhood $U \subseteq W \times \R^N\setminus\{0\}$ of $(w_0,\theta_0)$ such that $(w, t\theta) \in U$ for all $t >0$ whenever $(w, \theta) \in U$.} of $(w_0,\theta_0)$ can also be written as $I[\tilde{\varphi},\tilde{a}]$ for some $\tilde{a} \in S^{\mu + (d-2\tilde{N})/4}(W \times \R^{\tilde{N}})$ with $\mathrm{supp}\,\tilde{a}$ in a small $\tilde{\theta}$-conic neighbourhood of $(w_0, \tilde{\theta}_0)$.
\end{theorem}

The equivalence of phase theorem suggests that rather than thinking of the distribution $I[\varphi,a]$ as defined by some choice of phase/amplitude pair $[\varphi,a]$, one should think of the distribution as determined by the conic Lagrangian submanifold $\Lambda$. This perspective is described in the following subsection.




\subsection{Global theory}\label{global section}  The equivalence of phase theorem allows much of the analysis of the previous sections to be lifted to the more general setting of smooth manifolds $W$. Given a conic Lagrangian submanifold $\Lambda \subset T^*W$, one roughly defines $I^{\mu}(W, \Lambda)$ to be the class of homogeneous oscillatory integrals which can be locally represented as some $I[\varphi, a]$ for some non-degenerate phase function $\varphi$ for which $\Lambda_{\varphi}$ locally agrees with $\Lambda$ and symbol $a$ belonging to the class $S^{\mu + (d-2N)/4}$. Here $N$ is the dimension of the fibres (that is, the number of Fourier variables $\theta$) in this local representation. To give more precise details of this definition requires a brief review of basic concepts pertaining to analysis on manifolds.

\subsubsection*{Distributions on manifolds}  Let $W$ be a $d$-dimensional smooth manifold so that $W$ is equipped with a system of coordinate charts $\kappa_{\alpha} \colon W_{\alpha} \to \tilde{W}_{\alpha}$, each of which is a diffeomorphism from some open subset $W_{\alpha} \subseteq W$ to an open subset $\tilde{W}_{\alpha} \subseteq \R^d$. 

\begin{definition} A \emph{distribution $u$ on $W$} is an assignment of a distribution $u_{\alpha} \in \mathcal{D}'(\tilde{W}_{\alpha})$ to every coordinate chart which satisfies the following consistency property: given two charts $\kappa_{\alpha_j} \colon W_{\alpha_j} \to \tilde{W}_{\alpha_j}$, $j=1,2$, the identity
\begin{equation*}
    \langle u_{\alpha_2}, f \rangle = \langle u_{\alpha_1}, f \circ \alpha_1 \circ \alpha_2^{-1} \rangle
\end{equation*}
holds whenever $f \in \mathcal{D}(\tilde{W}_{\alpha_2})$ is supported inside $\kappa_{\alpha_2}(W_{\alpha_1} \cap W_{\alpha_2})$. The space of all distributions on $W$ is denoted by $\mathcal{D}'(W)$.
\end{definition}

\begin{remark} It makes perfect sense, in analogy with the definition in the euclidean case, to consider the dual of the space of test functions on $W$. Elements of this dual are slightly different objects to the distributions defined above, however, and are known as \emph{distribution densities} (see, for instance, \cite[\S 6.3]{HormanderI}).
\end{remark}

\subsubsection*{Global homogeneous oscillatory integrals} Let $\Lambda \subseteq T^*W$ be a closed, conic Lagrangian submanifold and $\mu \in \R$. A \emph{global homogeneous oscillatory integral of order $\mu$} is a distribution $u \in \mathcal{D}'(W)$ such that for every coordinate chart $\kappa_{\alpha} \colon W_{\alpha} \to \tilde{W}_{\alpha}$ the associated distribution $u_{\alpha} \in \mathcal{D}'(\tilde{W}_{\alpha})$ is of the form $u_{\alpha} = I[\varphi_{\alpha}, a_{\alpha}]$ where:
\begin{enumerate}[i)]
    \item Each $\varphi_{\alpha}$ is a non-degenerate phase function defined on a $\theta$-conic open subset $U_{\alpha} \subseteq \tilde{W}_{\alpha} \times \R^{N_{\alpha}}$ for some integer $N_{\alpha} \in \N$. Furthermore, an open neighbourhood of $\Lambda$ is diffeomorphically mapped to
    \begin{equation*}
        \Lambda_{\varphi_{\alpha}} := \{(w, \partial_{w}\varphi(w, \theta)) : (w, \theta) \in U_{\alpha}, \quad \partial_{\theta}\varphi(w, \theta) = 0 \}
    \end{equation*}
    under the coordinates induced by $\kappa_{\alpha}$.\footnote{In particular, if $\pi \colon T^*W \to W$ denotes the projection onto the base point, then for each chart $\kappa \colon W_{\alpha} \to \tilde{W}_{\alpha}$ one may define the induced local coordinates on the tangent bundle $\tilde{\kappa} \colon \pi^{-1}(W_{\alpha}) \to \tilde{W}_{\alpha} \times \R^d$ by $\tilde{\kappa}(w, \xi) := (\kappa(w), (\ud \kappa_w)^{-\top} \xi).$} 
    \item Each $a_{\alpha} \in S^{\mu + (d - 2N_{\alpha})/4}(\R^d \times \R^{N_{\alpha}})$ is supported in $U_{\alpha}$ and has compact $w$-support. 
\end{enumerate}

\subsubsection*{Global FIOs and canonical relations} Fix a pair of manifolds $X$ and $Y$ with $\dim X = n$ and $\dim Y = m$ and let $\Lambda \subseteq T^* X\setminus 0 \times T^* Y \setminus 0$ be a conic Lagrangian submanifold. The (global) homogeneous oscillatory integrals in $I^{\mu}(X \times Y, \Lambda)$ define (global) Fourier integral operators via the Schwartz kernel on each coordinate chart. The resulting collection of operators is denoted by $I^{\mu}(X, Y, \mathcal{C})$ where $\mathcal{C}$ is what is known as the \emph{canonical relation}: it is the rotated and reflected copy of $\Lambda$ given by  
\begin{equation*}
    \mathcal{C} := \big\{ (x, \xi, y, -\eta) \in T^*X\setminus 0 \times T^* Y \setminus 0 : (x, y, \xi, \eta) \in \Lambda \big\}.
\end{equation*}
One typically works with the canonical relation $\mathcal{C}$ rather than $\Lambda$ since it is often easier (notationally speaking) to formulate various hypotheses over $\mathcal{C}$ and $\mathcal{C}$ arises naturally in the composition calculus for FIOs (see, for instance, \cite[Chapter 6]{Sogge2017}). Of course, $\mathcal{C}$ inherits the symplectic structure of $\Lambda$ and, in particular, if $\omega_X := \sum_{j=1}^n \xi_j \ud x_j$ and $\omega_Y := \sum_{j=1}^n \eta_j \ud y_j$ denote the canonical 1-forms on $X$ and $Y$, respectively, then $\omega_X - \omega_Y$ vanishes identically on $\mathcal{C}$.

\subsubsection*{Global versus local theory} In the \emph{global} approach to Fourier integral theory one typically frames the hypotheses on the operator in terms of geometric properties of the canonical relation (some examples of this will be given in \S\ref{sec:LS}, where the \emph{rotational} and \emph{cinematic} curvature conditions are discussed). For the majority of this article, however, it will suffice to work with a concrete representation of the operator given by a choice of phase and amplitude as in \eqref{preliminary FIO}. In this \emph{local} approach, the hypotheses on the operator are framed in terms of properties of the choice of phase function $\phi$ and its derivatives. 

The local approach will in fact afford no loss of generality, since the problems under consideration are all of a local nature and it is always possible to locally express any FIO as an operator of the form of \eqref{preliminary FIO}. Indeed, if $\text{dim }X=\text{dim }Y$, given any FIO as above, basic results in symplectic geometry (see, for instance, \cite[Proposition 3.7.3]{Duistermaat2011} or \cite[Proposition 6.2.4]{Sogge2017}) guarantee that the canonical relation $\mathcal{C}$  can be expressed locally as a graph (modulo a reflection) of the form
\begin{equation*}
(x, \eta) \mapsto (x, S(x,\eta), T(x,\eta), -\eta). 
\end{equation*}
Moreover, it is not difficult to show that $S = \partial_x\phi$ and $T = \partial_{\eta}\phi$ for some \emph{generating function} $\phi$. Indeed, the canonical 1-form $\omega_X - \omega_Y$ is given in the above coordinates by
\begin{equation*}
    \sum_{j=1}^n \big[S_j(x,\eta) - \sum_{i=1}^n \eta_i \partial_{x_j}T_i(x,\eta)\big]\ud x_j - \sum_{j=1}^n \big[\sum_{i=1}^n \eta_i \partial_{\eta_j}T_i(x,\eta)\big]\ud \eta_j.
\end{equation*}
By the Lagrangian property of $\mathcal{C}$, the coefficient functions must all vanish identically on the domain, which implies that $\phi(x,\eta) := \langle \eta, T(x,\eta)\rangle$ is a suitable generating function. Thus, the canonical relation induced by the phase function $\varphi(x,y, \xi) := \phi(x,\xi) - \langle y, \xi\rangle$ agrees locally with $\mathcal{C}$ and consequently, by the equivalence of phase theorem, the FIO admits a local expression of the form \eqref{preliminary FIO}.




\section{Local smoothing estimates}\label{sec:LS}

This survey is primarily concerned with the continuity of FIOs as maps between certain function spaces. Such problems are inherently local in nature and, consequently, for the majority of the discussion it will suffice to work FIOs of the form \eqref{preliminary FIO}. In particular, let $\mathcal{F}$ be an operator given by
\begin{equation}\label{FIO def again}
    \mathcal{F}f(x) := \frac{1}{(2\pi)^n}\int_{\hat{\R}^n}e^{i\phi(x;\xi)} a(x;\xi) \hat{f}(\xi)\,\ud \xi
\end{equation}
for a choice of phase $\phi$ and symbol $a \in S^\mu(\R^n \times \R^n)$ satisfying the conditions described in \S\ref{motivating examples section}. We are interested in two kinds of estimates:
\begin{enumerate}[1)]
\item $L^p$-Sobolev bounds 
\begin{equation}\label{fixed time inequality}
    \|\mathcal{F}f\|_{L^p_s(\R^n)} \lesssim \|f\|_{L^p(\R^n)}.
\end{equation}
Here $L^p_s(\R^n)$ denotes the standard Sobolev (Bessel potential) space defined with respect to the Fourier multipliers $(1 + |\xi|^2)^{s/2}$ (see, for instance, \cite[Chapter V]{Stein1970}).
\item Given a 1-parameter family of FIOs $(\mathcal{F}_t)_{t \in I}$ for $I \subseteq \R$ a compact interval, we will consider inequalities of the form
\begin{equation}\label{space-time inequality}
\Big(\int_I \|\mathcal{F}_tf\|_{L^p_s(\R^n)}^p\ud t\Big)^{1/p} \lesssim \|f\|_{L^p(\R^n)}.
\end{equation}
\end{enumerate}

A prototypical case which motivates \eqref{space-time inequality} is given by the family of half-wave propagators $\mathcal{F}_t := e^{it\sqrt{-\Delta}}$. In this case, taking $\mathcal{F} = \mathcal{F}_t$ for a given $t$ (or the composition of this operator with a pseudo-differential operator) in \eqref{fixed time inequality} leads to fixed-time estimates for solutions to the wave equation; owing to this, such $L^p_s$ bounds will often be referred to as \emph{fixed-time estimates} (regardless of whether the operator involves a time parameter). On the other hand, \eqref{space-time inequality} is a \emph{``space-time" estimate}. 

Clearly, if one has a uniform bound of the kind \eqref{fixed time inequality} for every operator belonging to a 1-parameter family $(\mathcal{F}_t)_{t \in I}$, then \eqref{space-time inequality} follows directly by integrating these estimates over the time interval $I$. However, in many situations averaging over time has an additional smoothing effect; this allows for stronger space-time estimates than those obtained trivially by averaging fixed-time inequalities. This phenomenon is referred to as \emph{local smoothing}. 

In this section known and conjectured local smoothing properties of FIOs are described. In contrast with the fixed time case, the necessary conditions on $p$ for which an inequality of the form \eqref{space-time inequality} can hold can be quite subtle, depending on various geometric properties of the phase. An indication of the key considerations is provided below.

It transpires that local smoothing estimates are substantially stronger than their fixed-time counterparts and have many applications and implications in harmonic analysis and PDE. For instance, as is well-known, the sharp range of local smoothing inequalities for the wave propagator is known to imply numerous major open problems in harmonic analysis, including the Bochner--Riesz problem on convergence of Fourier series and the Kakeya conjecture. Non-sharp local smoothing estimates can also be very useful, and provide powerful tools for studying a wide range of maximal and variational problems in harmonic analysis (see, for instance, \cite{BRS, GHLR,  GRY} for recent examples of this). An introduction to the vast array of applications of local smoothing estimates is provided in \S\S\ref{section: maximal estimates}-\ref{section: local smoothing and oscillatory integral estimates}.

It appears that sharp local smoothing inequalities are extremely difficult to prove and, indeed, in the prototypical case of the half-wave propagator $e^{it\sqrt{-\Delta}}$ in the plane obtaining the sharp range of exponents remains a challenging open problem (although there are numerous partial results: see \S\ref{euclidean local smoothing section} below and the appendix). However, there is a fairly complete understanding in cases where the operator has a particularly badly behaved underlying geometry. For these, somewhat pathological, examples, geometric considerations place rather stringent constraints on the range of admissible $p$ values; consequently, it has been possible to obtain the full range of $L^p$ local smoothing estimates. This was observed recently in \cite{BHS} and relies heavily on fundamental work of Wolff \cite{Wolff2000} on the local smoothing problem and an important extension of Wolff's work due to Bourgain--Demeter \cite{Bourgain2015}. These topics are discussed in detail in \S\S\ref{sec:Wolff}-\ref{sec:decoupling}. 




\subsection{Fixed-time estimates for FIOs}\label{fixed time estimates section} Before discussing local smoothing in detail, it is instructive to first consider fixed-time estimates \eqref{fixed time inequality} for FIOs, which are somewhat easier to understand and help motivate the local smoothing theory. Consider a Fourier integral operator $\mathcal{F}$ of order $\mu$ as in \eqref{FIO def again} and suppose the symbol $a$ is compactly supported in the $x$ variable. In order to obtain a non-trivial $L^p$ theory, it is necessary to impose some further conditions on the phase.

\begin{mixed hessian condition} The phase $\phi$ satisfies
\begin{equation}\label{local graph condition}
    \det \partial^2_{x \xi} \phi(x;\xi) \neq 0 \qquad \textrm{for all $(x;\xi) \in \supp a$.}
\end{equation}
\end{mixed hessian condition}

An obvious example of a phase function satisfying \eqref{local graph condition} is $\phi(x, \xi) := \langle x, \xi \rangle$, corresponding to the case of pseudo-differential operators. The phase function appearing in the euclidean wave propagator in Example \ref{wave propagator example} also satisfies this hypothesis, as do those arising in the manifold setting in Example \ref{wave equation on a manifold example}. It transpires that \eqref{local graph condition} has a natural geometric interpretation in terms of the canonical relation $\mathcal{C}$ introduced in \S\ref{global section}; this is described below in \S\ref{geometry canonical relation section}. 

It was shown by Eskin \cite{Eskin} and H\"ormander \cite{Hormander1971} that FIOs of order 0 satisfying the above hypotheses are bounded on $L^2$. In general, for $p \neq 2$, they are not bounded on $L^p$ but $L^p$-Sobolev estimates do hold with some (necessary) loss in regularity. The sharp range of estimates of this form were established by Seeger, Stein and the third author \cite{Seeger1991}.

\begin{theorem}[\cite{Seeger1991}]\label{fixed-time theorem}
If $\mathcal{F}$ is a FIO of order $\mu$ satisfying the mixed Hessian condition, then for all $1 < p < \infty$ the fixed-time estimate
\begin{equation}\label{fixed time FIO} 
    \| \mathcal{F}f \|_{L^p_{-\mu -\bar{s}_p }(\R^n)} \lesssim \| f \|_{L^p(\R^n)} 
\end{equation}
holds for $\bar{s}_p:= (n-1)\big|\frac{1}{p}-\frac{1}{2}\big|$.
\end{theorem}

The proof of \eqref{fixed time FIO} in \cite{Seeger1991} follows from a $H^1(\R^n)$ to $L^1(\R^n)$ bound for FIOs of order $-\frac{n-1}{2}$ and interpolation with the aforementioned $L^2(\R^n)$ estimate for FIOs of order 0; this yields the results for $1 < p\le 2$ and the results for $2<p<\infty$ follow by duality.

As an example of an application of this theorem, one may apply the estimate to the FIOs arising in the parametrix for the half-wave propagator $e^{it\sqrt{-\Delta_g}}$ on a compact Riemannian manifold $(M,g)$. If $u$ is the solution to \eqref{wave equation on a manifold}, then one obtains the bound
\begin{equation}\label{fixed time estimate}
\|u(\,\cdot\, , t)\|_{L^p_{s - \bar{s}_p}(M)} \lesssim_{M,g} \|f_0\|_{L^p_{s}(M)} + \|f_1\|_{L^p_{s-1}(M)}
\end{equation}
for all $s \in \R$. Here $L^p_{s}(M)$ denotes the standard Sobolev (or Bessel potential) space, defined with respect to the spectral multiplier $(1 + \lambda^2)^{s/2}$ (see, for instance, \cite[Chapter 4]{Sogge2017}). Moreover, provided $t$ avoids a discrete set of times, the estimate \eqref{fixed time estimate} is sharp for all $1 < p < \infty$ in the sense that one cannot replace $\bar{s}_p$ with $\bar{s}_p - \sigma$ for any $\sigma > 0$. This provides an analogue of earlier bounds for solutions to the euclidean wave equation from \cite{Peral, MiyachiWave}. Theorem \ref{fixed-time theorem} can also be applied to solutions of more general strictly-hyperbolic equations, of any order: see \cite{Seeger1991} for further details.

\begin{remark}[Sharpness of fixed-time estimates]\label{counterexamples fixed time}
An integration-by-parts argument shows that for any $\alpha > 0$ the (distributional) inverse Fourier transform of $e^{-i|\xi|}(1 + |\xi|^2)^{-\alpha/2}$ agrees with a function $f_{\alpha}$. Moreover, $f_{\alpha}$ is rapidly decaying for $|x| \geq 2$ and for $|x| \leq 2$ satisfies
\begin{equation*}
    |f_{\alpha}(x)| \sim |1 - |x||^{-(n+1)/2 + \alpha}.
\end{equation*}
Thus, if $\alpha > \frac{n+1}{2}- \frac{1}{p}$, then $f_{\alpha} \in L^p(\R^n)$. On the other hand, 
\begin{equation*}
    |(1 - \Delta)^{-s/2} e^{i\sqrt{-\Delta}}f_{\alpha}(x)| \gtrsim |x|^{-(n - \alpha - s)} \qquad \textrm{for $|x| \lesssim 1$.} 
\end{equation*}
Thus, if $\alpha \leq n-\frac{n}{p} - s$, then $e^{i\sqrt{-\Delta}}f_{\alpha} \notin L^p_{-s}(\R^n)$. Comparing these two conditions shows that Theorem \ref{fixed-time theorem} is optimal for $2 \leq p < \infty$, in the sense that $\bar{s}_p$ cannot be replaced with any smaller exponent. The range $1 < p \leq 2$ then follows by duality. See \cite[Chapter IX, $\S$6.13]{Stein1993} for further details.
\end{remark}

\begin{remark}[Sharpness in the range $1<p \leq  2$]\label{counterexample dual regime}
In fact, one may also deduce the sharpness of Theorem \ref{fixed-time theorem} in the regime $1 < p \leq 2$ from a direct example rather than from a duality argument. Reasoning as in Remark \ref{counterexamples fixed time}, given $\alpha>0$, let $g_\alpha$ be a function whose distributional Fourier transform is $(1+|\xi|^2)^{-\alpha/2}$. The function $g_\alpha$ is rapidly decreasing at infinity and
$$
|g_\alpha(x)| \sim |x|^{-(n-\alpha)} \qquad \textrm{for $|x| \lesssim 1$;} 
$$
thus $g_\alpha \in L^p(\R^n)$ if $\alpha>n-\frac{n}{p}$. On the other hand,
$$
    |(1 - \Delta)^{-s/2} e^{i\sqrt{-\Delta}}g_{\alpha}(x)| \gtrsim |1-|x||^{-(n+1)/2 + \alpha + s},
$$
so $e^{i \sqrt{-\Delta}} g_\alpha \not \in L^p_{-s}(\R^n)$ if $\alpha < \frac{n+1}{2} - s - \frac{1}{p}$. Combining both conditions on $\alpha$ yields $\bar s_p$ in Theorem \ref{fixed-time theorem} cannot be replaced with any smaller exponent if $1 < p \leq 2$.
\end{remark}




\subsection{Local smoothing estimates for FIOs} We now turn to the subject of local smoothing estimates. Recall that the problem is to analyse a 1-parameter family $(\mathcal{F}_t)_{t \in I}$ of FIOs, the prototypical example being the wave semigroup  $e^{i t \sqrt{-\Delta}}$. It is convenient to formulate the problem in terms of a single Fourier integral operator mapping functions on $\R^n$ to functions on $\R^{n+1}$. In particular, consider an operator
\begin{equation}\label{FIO def n to n+1}
    \mathcal{F}f(x, t) := \frac{1}{(2\pi)^n}\int_{\hat{\R}^n}e^{i\phi(x,t;\xi)} a(x, t;\xi) \hat{f}(\xi)\,\ud \xi \qquad (x,t) \in \R^{n+1}
\end{equation}
where the symbol $a \in S^\mu(\R^{n+1} \times \R^n)$ is compactly supported in $x$ and $t$ and the phase function $\phi$ is homogeneous of degree 1 in $\xi$ and smooth away from $\xi=0$. The conditions on the phase are now formulated with respect to the space-time variables $(x,t)$ and the analogous condition to \eqref{local graph condition} reads as follows. 

\begin{mixed hessian condition} The phase $\phi$ satisfies:
\begin{itemize}
\item[H1)]$\mathrm{rank}\, \partial_{ \xi z}^2 \phi(x,t;\xi) = n$ for all $(x,t; \xi) \in \mathrm{supp}\,a \setminus 0$.  
\end{itemize}
Here and below $z$ is used to denote vector in $\R^{n+1}$ comprised of the space-time variables $(x,t)$.
\end{mixed hessian condition}

Trivially, under these hypotheses Theorem \ref{fixed-time theorem} implies the space-time estimate 
\begin{equation}\label{local smoothing FIO}
    \Big(\int_{\R} \| \mathcal{F}f(\,\cdot\,,t) \|_{L^p_{-\mu-\bar{s}_p}(\R^{n})}^p \,\ud t \Big)^{1/p} \lesssim \| f \|_{L^p(\R^n)}. 
\end{equation}
This range of exponents, which follows from fixed-time estimates alone, does not encapsulate any additional smoothing arising from taking the average in time. Moreover, without further conditions on the phase, no such additional smoothing is possible, in general, and the above range is in fact sharp (a standard example which demonstrates this is given by the Radon transform in the plane: see Example \ref{Radon revisited} below or \cite[Chapter 6]{Sogge2017}). In order to establish non-trivial local smoothing estimates, one restricts to the class of phases satisfying the following additional hypothesis.

\begin{curvature condition} The phase $\phi$ satisfies:
\begin{itemize}
\item[H2)] The generalised Gauss map, defined by $G(z; \xi) := \frac{G_0(z;\xi)}{|G_0(z;\xi)|}$ for all $(z; \xi) \in \mathrm{supp}\,a \setminus 0$ where
\begin{equation*}
G_0(z;\xi) :=  \bigwedge_{j=1}^{n} \partial_{\xi_j} \partial_z\phi(z;\xi),
\end{equation*}
satisfies
\begin{equation*}
\mathrm{rank}\,\partial^2_{\eta \eta} \langle \partial_z\phi(z;\eta),G(z; \xi)\rangle|_{\eta = \xi} = n-1
\end{equation*}
for all $(z; \xi) \in \mathrm{supp}\,a \setminus 0$.
\end{itemize}
\end{curvature condition}

Geometrically, the curvature condition means that for fixed $z_0$ the cone
\begin{equation}\label{cones in terms of phase}
\Gamma_{z_0}:=\{ \partial_z \phi(z_0;\eta): \: \eta \in \R^n \backslash 0 \textrm{ in a conic neighbourhood of } \eta_0\}
\end{equation}
is a smooth (conic) manifold of dimension $n$ with $n-1$ non-vanishing principal curvatures at every point. One may readily verify that the phase featured in the prototypical example of the half-wave propagator $e^{i t \sqrt{-\Delta}}$ satisfies both conditions H1) and H2). The same is also true for the phases arising from the parametrix construction for $e^{i t \sqrt{-\Delta_g}}$ in Example \ref{wave equation on a manifold example}.

Under the conditions H1) and H2), it is possible to show that for $2 < p < \infty$ there exists some  $\sigma(p)>0$ such that inequality \eqref{local smoothing FIO} holds with $\bar{s}_p$ replaced with $\bar{s}_p - \sigma(p)$. This corresponds to a regularity gain over the estimate \eqref{fixed time FIO} and is an example of the local smoothing phenomenon. The existence of local smoothing estimates of the type \eqref{local smoothing FIO} was first observed by the third author \cite{Sogge1991} in the context of the euclidean half-wave propagator $e^{i t \sqrt{-\Delta}}$. Shortly after, Mockenhaupt, Seeger and the third author \cite{Mockenhaupt1992, Mockenhaupt1993} established stronger local smoothing estimates in the general context of Fourier integral operators satisfying H1) and H2).




\subsection{The local smoothing conjecture}\label{euclidean local smoothing section} A natural question is to quantify the precise range of exponents for which \eqref{local smoothing FIO} holds for a given FIO $\mathcal{F}$ satisfying the hypotheses H1) and H2). It transpires that this is a difficult and largely unresolved problem, involving a subtle dependence on certain geometric properties of $\mathcal{F}$. 

\subsubsection*{The euclidean wave semigroup} To begin the discussion, we first consider the prototypical case of the wave semigroup $e^{it\sqrt{-\Delta}}$. In \cite{Sogge1991}, the following conjecture was formulated. 

\begin{conjecture}[Local smoothing conjecture]\label{LS conj euclidean}
For $n \geq 2$ the inequality
\begin{equation}\label{LS conj euclidean equation}
\Big( \int_1^2 \|e^{i t \sqrt{-\Delta}} f \|^p_{L^p_{-\bar{s}_p + \sigma}(\mathbb{R}^n )} \ud t \Big)^{1/p} \lesssim \|f\|_{L^p(\R^n)}
\end{equation}
holds for all $\sigma < 1/p $ if $\frac{2n}{n-1} \leq p < \infty$ and $\sigma<\bar{s}_p$ if $2 < p \leq \frac{2n}{n-1}$. 
\end{conjecture}

Note that the order of the half-wave propagator $e^{i t \sqrt{-\Delta}}$ is $\mu=0$, so the conjecture claims a $\sigma$-regularity gain with respect to the fixed time estimates \eqref{fixed time FIO} in Theorem \ref{fixed-time theorem}. This conjecture is open in all dimensions, although there have been numerous partial results which establish \eqref{LS conj euclidean equation} either for:
\begin{itemize}
    \item A restricted range of regularity \cite{Sogge1991, Mockenhaupt1992, BourgainSF, TV2, Lee} or
    \item A sharp gain in regularity for a restricted range of Lebesgue exponent \cite{Wolff2000, Laba2002, Lee2006, Garrigos2009, Garrigos2010, Heo2011, Bourgain2015}. 
\end{itemize}
It is remarked that in \cite{Heo2011} a strengthened version of the conjecture was in fact established, involving estimates with the endpoint regularity index $\sigma = 1/p$ (for a restricted range of $p$). The history of the problem is discussed in more detail in the appendix. 

For $p = 2$, Plancherel's theorem implies the energy conservation identity
\begin{equation}\label{energy conservation}
\| e^{it\sqrt{-\Delta}} f \|_{L^2(\R^{n} \times [1,2])} = \| f \|_{L^2(\R^n)} 
\end{equation}
whilst for $p = \infty$ one may show 
\begin{equation}\label{infty endpoint}
    \| e^{it\sqrt{-\Delta}} f \|_{L^\infty_{-\frac{(n-1)}{2} - \varepsilon}(\R^{n} \times [1,2])} \lesssim \| f \|_{L^\infty(\R^n)}.
\end{equation}
The estimate \eqref{infty endpoint} follows by bounding certain localised pieces of the kernel in $L^1$; this kind of argument is described in detail in \S\ref{bounding localised pieces section}. On a heuristic level, \eqref{infty endpoint} can be understood by comparison with the averaging operators $A_t$ from Example \ref{averaging operator example} which are automatically bounded on $L^{\infty}$ and roughly correspond to the composition of $e^{it\sqrt{-\Delta}}$ with a multiplier in $S^{-(n-1)/2}$. 

As a consequence of these simple estimates, \eqref{LS conj euclidean equation} is strongest at $p=\frac{2n}{n-1}$: the estimates for all other $p$ follow from the $p=\frac{2n}{n-1}$ case via interpolation with \eqref{energy conservation} and \eqref{infty endpoint}. Thus, Conjecture \ref{LS conj euclidean} amounts to the assertion that $e^{i t \sqrt{-\Delta}}$ is essentially (that is, modulo a necessary loss of $\varepsilon>0$ derivatives) bounded on $L^{2n/(n-1)}(\R^{n+1})$ locally in time.

\begin{remark}\label{no local smoothing at L2 remark} Historically, the local smoothing phenomenon was first observed in the context of $L^2$-type bounds for dispersive equations \cite{Constantin1988, Kato1983, Sjolin1987, Vega1988}. Here the setup is somewhat different. For instance, in the case of the Schr\"odinger equation a gain of $1/2$ a derivative is obtained when one integrates the solution locally with respect to time \textit{over a compact spatial region}:
\begin{equation}\label{Schrodinger local smoothing}
    \Big(\int_1^2\|e^{-it\Delta}f\|_{L^2(B(0,1))}^2\,\ud t \Big)^{1/2} \lesssim \|f\|_{L^2_{1/2}(\R^n)}.
\end{equation}
Of course, the spatial localisation is essential in \eqref{Schrodinger local smoothing}: the estimate cannot hold with a global $L^2(\R^n)$-norm owing to conservation of energy. In the case of the wave equation, the operator $e^{it\sqrt{-\Delta}}$ is local at scale $t$ (as is clear either from the Kirchhoff formula for the solution (see, for instance, \cite[Chapter 1]{Sogge2008}) or by  analysing the kernel of $e^{it\sqrt{-\Delta}}$ via (non-) stationary phase). Consequently, local and global $L^2$ estimates for the half-wave propagator are essentially equivalent and, thus, conservation of energy prohibits any analogous inequality of the form \eqref{Schrodinger local smoothing} for $e^{it\sqrt{-\Delta}}$. 
\end{remark}

It is worthwhile examining the examples which dictate the numerology appearing in Conjecture \ref{LS conj euclidean}. First of all, it is clear that no local smoothing is possible for $p = 2$ for reasons described in Remark \ref{no local smoothing at L2 remark} above. Furthermore, the example in Remark \ref{counterexample dual regime} showing the sharpness of the fixed-time estimates if $1 < p < 2$ immediately yields that no local smoothing estimates hold in this regime. The situation is different if $2 < p <\infty$.

\begin{remark}[Sharpness of local smoothing conjecture] The example used in Remark \ref{counterexamples fixed time} can be used to show that a gain of 1/p derivatives in the local smoothing conjecture would be best possible. In particular, let $f_{\alpha}$ be as defined in Remark \ref{counterexamples fixed time}, so that if $\alpha > \frac{n+1}{2} - \frac{1}{p}$, then $f_{\alpha} \in L^p(\R^n)$. Moreover, one may show that 
\begin{equation*}
    |(1 - \Delta)^{-s/2} e^{it\sqrt{-\Delta}}f_{\alpha}(x)| \gtrsim 
    |x|^{-(n-1)/2}|t-1-|x||^{-(n+1)/2 + \alpha + s} \quad \textrm{if $t\geq 2|x|+1$}
\end{equation*}
whenever $|x| \lesssim 1$. Thus, if $\alpha \leq n - \frac{n+1}{p} - s$, then
\begin{equation*}
    \Big( \int_1^2 \|e^{i t \sqrt{-\Delta}} f_{\alpha} \|^p_{L^p_{-s}(\mathbb{R}^n )} \ud t \Big)^{1/p} = \infty.
\end{equation*}
Comparing the two conditions on $\alpha$ shows that Conjecture \ref{LS conj euclidean} is optimal in the sense that $1/p$ cannot be replaced with any larger number. 
\end{remark}

\subsubsection*{Wave equations on manifolds} Given a compact $n$-dimensional Riemannian manifold $(M,g)$ one may also consider the local smoothing problem for the propagator $e^{i t \sqrt{-\Delta_g}}$, as defined in Example \ref{wave equation on a manifold example}. It is perhaps tempting to conjecture that \eqref{LS conj euclidean equation} should also hold for $e^{i t \sqrt{-\Delta_g}}$ for the same range of exponents as described in Conjecture \ref{LS conj euclidean}. This turns out to be somewhat na\"ive, however. In particular, Minicozzi and the third author \cite{Minicozzi1997} identified compact manifolds $(M,g)$ for which local smoothing fails to hold for all orders $\sigma < 1/p$ whenever $p<\bar{p}_{n,+}$ where
\begin{equation*}
    \bar{p}_{n,+}:=
\begin{cases}
 \frac{2(3n+1)}{3n-3} \quad \textrm{if $n$ is odd} \\[2pt]
\frac{2(3n+2)}{3n-2} \quad \textrm{if $n$ is even} \\
\end{cases};
\end{equation*}
see Figure \ref{conjectured exponents figure}. Furthermore, $(M,g)$ may be taken to be an arbitrarily small smooth perturbation of the euclidean metric on $\R^n$. Thus, one is led to the following conjecture.
\begin{conjecture}[Local smoothing conjecture: compact manifolds]\label{LS conj manifold}
For $n \geq 2$ and $(M,g)$  a compact Riemannian manifold of dimension $n$, the inequality
\begin{equation}\label{LS conj manifold equation}
\Big( \int_1^2 \|e^{i t \sqrt{-\Delta_g}} f \|^p_{L^p_{-\bar{s}_p + \sigma}(M)} \ud t \Big)^{1/p} \lesssim \|f\|_{L^p(M)}
\end{equation}
holds for all $\sigma < 1/p $ if $\bar{p}_{n, +} \leq p < \infty$. 
\end{conjecture}
Note that the conjecture would automatically imply bounds of the form \eqref{LS conj manifold equation} in the $2 \leq p \leq \bar{p}_{n,+}$ range via interpolation with the $L^2$ energy estimate. For simplicity, the values for $\sigma$ in this range of $p$ are omitted.

The counterexamples in \cite{Minicozzi1997} were inspired by earlier work of Bourgain \cite{Bourgain1991, Bourgain1995} in the context of oscillatory integral operators and are geometric in nature. In particular, obstacles to \eqref{LS conj manifold equation} arise owing to so-called Kakeya/Nikodym compression phenomena for geodesics in $(M,g)$. Some related examples are discussed in detail below in \S\ref{sharpness section}.

\subsubsection*{General FIOs} In approaching Conjecture \ref{LS conj manifold} one may work in local coordinates; the problem then boils down to establishing local smoothing estimates for the FIOs featured in the parametrix for $e^{it\sqrt{-\Delta_g}}$. One may ask more generally whether such local smoothing estimates hold for all $\mathcal{F}$ satisfying the conditions H1) and H2). It turns out, however, that there are further examples of FIOs (which do not arise in relation to the half-wave propagators $e^{it\sqrt{-\Delta_g}}$) for which local smoothing is only possible on a strictly smaller range of exponents than $p\geq \bar{p}_{n,+}$. These examples are of a slightly indirect nature. In particular, in \cite{BHS} it was shown that local smoothing estimates \eqref{local smoothing FIO} imply certain oscillatory integral bounds: see Theorem \ref{thmo}. Counterexamples of Bourgain \cite{Bourgain1991, Bourgain1995} in the latter context can then be applied to the problem; the details of this argument are reviewed below in \S\ref{section: local smoothing and oscillatory integral estimates}. In particular, there exist choices of $\mathcal{F}$ satisfying H1) and H2) for which local smoothing fails to hold for all orders $\sigma < 1/p$ whenever $p < \bar{p}_n$ where
\begin{equation}\label{exponents general}
\bar{p}_n := \begin{cases} \frac{2(n+1)}{n-1} \quad \textrm{if $n$ is odd} \\[2pt]
\frac{2(n+2)}{n} \quad \textrm{if $n$ is even} \\
\end{cases};
\end{equation}
see Figure \ref{conjectured exponents figure}. Thus, one is led to the following general conjecture.

\begin{conjecture}[Local smoothing conjecture: FIOs]\label{LS conj FIO} Suppose $\mathcal{F}$ is a FIO with symbol of order $\mu$ satisfying conditions H1) and H2) above. The inequality 
\begin{equation}\label{eq:LS FIO}
    \Big( \int_{1}^2 \| \mathcal{F}f(\,\cdot\,,t) \|^p_{L^p_{-\mu - \bar{s}_p + \sigma}(\R^{n})} \, \ud t \Big)^{1/p} \lesssim \| f \|_{L^p(\R^n)}
\end{equation}
holds for all $\sigma < 1/p$ if $\bar{p}_n \leq p < \infty$. 
\end{conjecture}
As in the case of the wave semigroup, the conjecture automatically implies bounds of the form \eqref{LS conj manifold equation} in the $2 \leq p \leq \bar{p}_{n}$ range, this time via interpolation with the $L^2$ bounds of Eskin \cite{Eskin} and H\"ormander \cite{Hormander1971}. 

It will be useful to introduce the following terminology.

\begin{definition} Given $2 < p < \infty$, we say there is \textit{$1/p-$ local smoothing} (or \textit{local smoothing of order $1/p-$}) for a FIO $\mathcal{F}$ as above if \eqref{eq:LS FIO} holds for all $\sigma < 1/p$. 
\end{definition}

In this language, the above conjectures may be stated succinctly as follows:\\

\noindent
\textbf{Conjecture \ref{LS conj euclidean}.} There is $1/p-$ local smoothing for $e^{it\sqrt{-\Delta}}$  for all $\frac{2n}{n-1} \leq p < \infty$.\\

\noindent
\textbf{Conjecture \ref{LS conj manifold}.} There is $1/p-$ local smoothing for $e^{it\sqrt{-\Delta}_g}$  for all $ \bar{p}_{n, +} \leq p < \infty$.\\

\noindent
\textbf{Conjecture \ref{LS conj FIO}.} If $\mathcal{F}$ is a FIO of order $\mu$ satisfying H1) and H2), then there is $1/p-$ local smoothing for $\mathcal{F}$ for all $\bar{p}_{n}\leq p < \infty$.\\

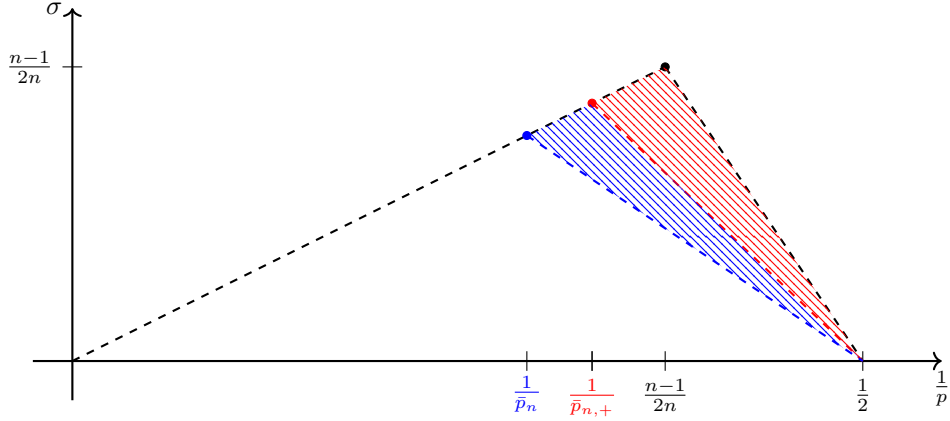
\begin{figure}
\begin{tikzpicture}[scale=2.6] 

\begin{scope}[scale=2]
\draw[thick,->] (-.1,0) -- (2.2,0) node[below] {$ \frac 1 p$};
\draw[thick,->] (0,-.1) -- (0,0.9) node[left] {$ \sigma$};

\draw (.025,3/4) -- (-.025,3/4) node[left] {$ \frac{n-1}{2n}$};
\draw (3/2,-0.025) -- (3/2,.025) node[below = 0.25cm] {$ \frac{n-1}{2n}$};

\draw (2,-0.025) -- (2,.025) node[below= 0.25cm] {$ \frac{1}{2}$}; 
\node[circle,draw=black, fill=black, inner sep=0pt,minimum size=3pt] (b) at (3/2,3/4) {};

\draw[thick, dashed]  (0,0)  -- (3/2,3/4); 

\draw[thick, dashed] (3/2,3/4) -- (2,0);

\draw[thick, dashed, color=red] (2.63/2,2.63/4) -- (2,0) ;

\fill[pattern=north west lines, pattern color=red] (2.63/2,2.63/4) -- (2,0) -- (3/2,3/4) -- (2.63/2,2.63/4);

\draw (2.63/2, -0.025) -- (2.63/2, 0.025) node[below=0.25cm] {\textcolor{red}{$\frac{1}{\bar{p}_{n, +}}$}}; 

\node[circle,draw=red, fill=red, inner sep=0pt,minimum size=3pt] (b) at (2.63/2,2.63/4) {};
\draw (2.63/2, .025) -- (2.63/2,-.025) ;

\draw (2.3/2,-0.025) -- (2.3/2,.025);
\draw (2.3/2, 0) node[below = 0.08cm] {\textcolor{blue}{$ \frac{1}{\bar{p}_n}$} };

\fill[pattern=north west lines, pattern color=blue] (2.3/2,2.3/4) -- (2,0) -- (2.63/2,2.63/4) -- (2.3/2,2.3/4);

\draw[thick, dashed, color=blue] (2.3/2,2.3/4) -- (2,0) ;

\node[circle,draw=blue, fill=blue, inner sep=0pt,minimum size=3pt] (b) at (2.3/2,2.3/4) {};

\end{scope}

\end{tikzpicture}

\caption{Exponents for various formulations of the local smoothing conjecture. In contrast to the euclidean case, for wave propagators on certain compact manifolds the {\color{red}red} region is inadmissible. There exist FIOs for which the {\color{blue}blue} region is also inadmissible.}
\label{conjectured exponents figure}
\end{figure}

The examples in \cite{BHS} show that Conjecture \ref{LS conj FIO} would be sharp \textit{across the entire class} of FIOs satisfying H1) and H2), but there are many situations where one expects a better range of estimates to hold (not least of all the case of the wave propagators described above). One may, in fact, formulate a more refined conjecture, which combines both Conjecture \ref{LS conj manifold} and Conjecture \ref{LS conj FIO} into a single statement, by considering finer geometric properties of the phase function and working with a more precise version of the hypothesis H2). This is discussed below in \S\ref{formulating a conjecture section}.




\subsection{Positive results} Recently in \cite{BHS}, the odd dimensional case of Conjecture \ref{LS conj FIO} was established. 

\begin{theorem}\label{thm:LS}
Let $\mathcal{F}$ be a FIO as in \eqref{FIO def n to n+1} satisfying H1) and H2) and with symbol of order $\mu$. There is $1/p-$ local smoothing for $\mathcal{F}$ for all $\frac{2(n+1)}{n-1} \leq p < \infty$.
\end{theorem}

This result extends earlier work of Wolff \cite{Wolff2000} and Bourgain--Demeter \cite{Bourgain2015} which establishes the theorem in the special case of the euclidean wave semigroup. 

Theorem \ref{thm:LS} is, up to endpoints, sharp across the entire class of FIOs in odd dimensions, in terms of both the regularity and the Lebesgue exponents. The question of what happens at the endpoint regularity index remains open; see \cite{Lee2013} for partial results in this direction. Thus, in order to prove estimates for a wider range of exponents than those provided by Theorem \ref{thm:LS} one must assume additional hypotheses on $\mathcal{F}$. In view of this, some natural refinements of the condition H2) are discussed in the following subsection. 

The method used to establish Theorem \ref{thm:LS} follows Wolff's approach to local smoothing \cite{Wolff2000}. This relies on establishing variable coefficient counterparts to the sharp $\ell^p(L^p)$ \textit{Wolff-type} (or \textit{decoupling}) inequalities of Bourgain--Demeter \cite{Bourgain2015}. It is a remarkable fact that the aforementioned decoupling inequalities are stable under smooth perturbations of the phase in the underlying operator, leading to the results in \cite{BHS}. A detailed review of this argument is provided in \S\ref{sec:Wolff}. An interesting aspect of this analysis is that the variable coefficient decoupling estimates can be derived rather directly as a consequence of the constant coefficient estimates, via an induction-on-scale argument. This is discussed in \S\ref{sec:decoupling}.

\subsection{Formulating a local smoothing conjecture for general FIOs}\label{formulating a conjecture section} Comparing Conjectures \ref{LS conj euclidean}, \ref{LS conj manifold} and \ref{LS conj FIO}, it is natural to ask what the special properties of the half-wave propagators $e^{it\sqrt{-\Delta_g}}$ and $e^{it\sqrt{-\Delta}}$ are which distinguish them from general FIOs and allow one to conjecture a larger range of local smoothing estimates in these cases. 

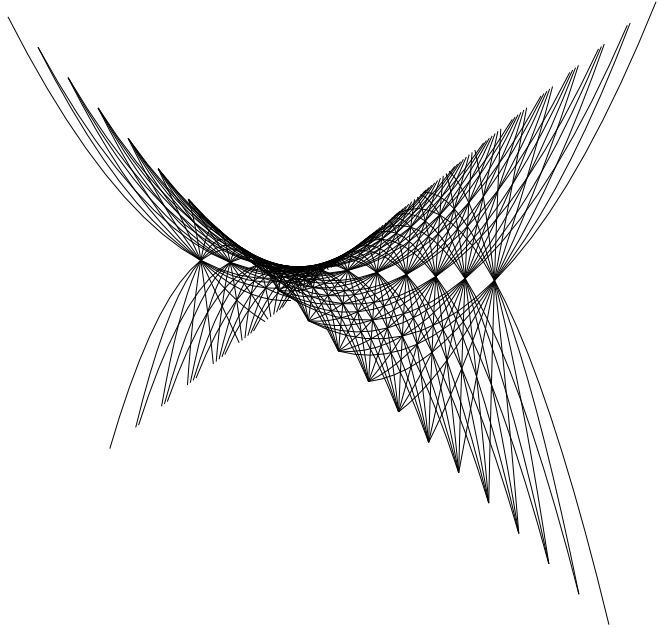
\begin{figure}
    \centering
    \begin{tikzpicture}[tdplot_main_coords, scale=2]
  
	\begin{scope}[rotate around x=5]
   \foreach \a in {-1,-0.8,...,1}
  {
    \foreach \b in {-1,-0.8,...,1}
    {
		\draw[black, very thin] 
		({-\a + \b*1.0},{-1.0},{\a*1.0-\b*(1.0*1.0)})
    \foreach \t in {-1.0,-0.9,...,1.0}
    {
        --({-\a - \b*\t},{\t},{-\a*\t-\b*(\t*\t)})
    }
    ;
    }
    }
		\end{scope}
			\end{tikzpicture}
    \caption{An example of a \emph{Kakeya/Nikodym set of curves}, arising from Bourgain's example \cite{Bourgain1991, Bourgain1995} (see also \cite{Guth}). Here a large family of distinct parabolas lies inside a 2-dimensional set (a hyperbolic paraboloid).}  
    \label{Kakeya figure}
\end{figure}

It is first remarked that the stronger numerology in the euclidean conjecture (Conjecture \ref{LS conj euclidean}) is related to deep questions in geometric measure theory. In particular, it is well-known that Conjecture \ref{LS conj euclidean} implies the Kakeya conjecture\footnote{This conjecture states that if $K \subseteq \R^n$ is compact and contains a unit line segment in every direction, then $K$ should have Hausdorff dimension $n$.} concerning the Hausdorff dimension of Kakeya sets in $\R^n$; for a discussion of the relationship between these and other important problems in harmonic analysis and geometric measure theory see, for instance, \cite{Minicozzi1997, Tao1999, Wolff1999} and the following section. A similar relationship holds when one considers wave propagators on manifolds and, moreover,  general FIOs. In particular, both Conjectures \ref{LS conj manifold} and \ref{LS conj FIO} imply bounds on the dimension of certain \emph{Kakeya} (or, more precisely, \emph{Nikodym}) \emph{sets of curves}. For instance, when dealing with the propagator $e^{it\sqrt{-\Delta_g}}$ the curves in question are geodesics in $(M,g)$. The precise definition of a Kakeya/Nikodym set of curves will not be recalled here, but the interested reader is directed to \cite{Bourgain2011, Guth, Minicozzi1997, Wisewell2005} or \cite[Chapter 9]{Sogge2017} for further details. The key observation is that, for certain examples, the families of curves which arise in this manner \textbf{fail} the Kakeya/Nikodym conjecture. More precisely, the curves can be arranged so that they lie in a set of small Hausdorff dimension; see Figure \ref{Kakeya figure}.\footnote{In the worst case scenario, the curves can be arranged to lie in a set of dimension $\lceil \frac{n+1}{2}\rceil$ where $n$ is the ambient dimension \cite{Bourgain1991, Bourgain1995, Bourgain2011}.} Such geometric configurations can be used to preclude local smoothing estimates near $\frac{2n}{n-1}$ for certain propagators $e^{it\sqrt{-\Delta_g}}$ and lead to the numerology in Conjecture \ref{LS conj manifold}. 

It remains to explain the difference in numerology between Conjecture \ref{LS conj manifold} for wave propagators on manifolds and Conjecture \ref{LS conj FIO} for general FIOs. Recall that the necessary condition in Conjecture \ref{LS conj FIO} arises from counterexamples of Bourgain \cite{Bourgain1991,Bourgain1995} for bounds for oscillatory integral operators; this is discussed in detail below in \S\ref{sharpness section}. It is remarked that one key feature of Bourgain's examples is that they give rise to \emph{hyperbolic} cones $\Gamma_{z_0}$: that is, the non-vanishing principal curvatures of $\Gamma_{z_0}$ have \emph{different signs}. Moreover, the analysis can be refined to give necessary conditions which depend on the difference between the number of positive and number of negative principal curvatures \cite{HI}.

\begin{figure}
\begin{center}
\begin{TAB}(c,1cm,2cm)[5pt]{|c|c|c|}{|c|c|c|}
 & $n$ odd & $n$ even    \\
\begin{tabular}[x]{@{}c@{}}$n-1$ non-vanishing \\curvatures\end{tabular} & $\displaystyle \textcolor{green!80!black}{\frac{2(n+1)}{n-1}}$ & $\displaystyle \frac{2(n+2)}{n}$  \\
\begin{tabular}[x]{@{}c@{}}$n-1$ positive \\curvatures\end{tabular}  & $\displaystyle \frac{2(3n+1)}{3n-3}$ & $\displaystyle \frac{2(3n+2)}{3n-2}$ \\
\end{TAB}
\end{center}
\caption{Conjectured endpoint values for the exponent $p$ for the sharp local smoothing estimates \eqref{local smoothing FIO} under various signature hypotheses on the phase. Theorem~\ref{thm:LS} establishes the odd dimensional case under the hypothesis of $n-1$ non-vanishing principal curvatures.}
\label{fig:signature}
\end{figure}

\begin{definition} Suppose $\mathcal{F}$ is an FIO which satisfies the conditions H1) and H2). We say $\mathcal{F}$ has \emph{signature} $\kappa$ for some integer $0 \leq \kappa \leq n-1$ if each of the cones $\Gamma_{z_0}$ satisfies
\begin{equation*}
    \kappa =|\textrm{\# positive principal curvatures} - \textrm{\# negative principal curvatures}| 
\end{equation*}
at every point. 
\end{definition}

With this definition, and in light of the modified versions of Bourgain's examples, one may formulate a refined version of Conjecture \ref{LS conj FIO}. In particular, letting
\begin{equation*}
\bar{p}_{n,\kappa} := 
\left\{\begin{array}{ll} 
2 \cdot \displaystyle \frac{\kappa + 2(n+1)}{\kappa + 2(n-1)} & \qquad \textrm{if $n$ is odd} \\[8pt]
 2 \cdot \displaystyle \frac{\kappa + 2n+3}{\kappa + 2n-1} & \qquad \textrm{if $n$ is even}
\end{array}\right. , 
\end{equation*}
the new conjecture reads thus. 

\begin{conjecture}[Local smoothing conjecture: FIOs]\label{LS conj FIO refined} Suppose $\mathcal{F}$ is a FIO with symbol of order $\mu$ satisfying conditions H1) and H2) and that $\mathcal{F}$ has signature $\kappa$. There is $1/p-$ local smoothing for $\mathcal{F}$ for all $\bar{p}_{n,\kappa} \leq p < \infty$. 
\end{conjecture}

Since there are $n-1$ non-vanishing principal curvatures, in the worst case scenario the signature is given by
\begin{equation*}
    \kappa = \left\{\begin{array}{ll} 
0 & \qquad \textrm{if $n$ is odd} \\
 1 & \qquad \textrm{if $n$ is even}
\end{array}\right. .
\end{equation*}
Substituting these values into the formula for $\bar{p}_{n,\kappa}$, one recovers the exponent $\bar{p}_n$ from \eqref{exponents general} and therefore Conjecture \ref{LS conj FIO refined} subsumes Conjecture \ref{LS conj FIO}. On the other hand, in the best case scenario the principal curvatures all have the same sign and $\kappa = n - 1$. In this case, we see that $\bar{p}_{n, n-1}$ agrees with the exponent $\bar{p}_{n,+}$ from Conjecture \ref{LS conj manifold}. Furthermore, it is indeed the case that the Fourier integral operators associated to the wave semigroups $e^{i t \sqrt{-\Delta_g}}$ have signature $n-1$ (see, for instance, \cite{Minicozzi1997} or \cite[Chapter 4]{Sogge2017}). Thus, Conjecture \ref{LS conj FIO refined} also subsumes Conjecture \ref{LS conj manifold}. See Figure \ref{fig:signature}. 

From the above discussion it is not at all clear \emph{why} the signature should be important in the analysis of these operators, other than it is a consideration in the construction of counterexamples. In the case of oscillatory integral operators, the precise r\^ole of the signature is fairly well understood and is discussed in detail in \cite{Guth} (see also \cite{Bourgain2011, HI, Lee2006}). It is highly likely that the signature will play a similar r\^ole in the analysis of FIOs.




\subsection{The geometric conditions in terms of the canonical relation}\label{geometry canonical relation section} To round off this section, we describe how the local results of the previous subsections can be transcribed into the broader setting of global FIOs. In particular, we provide a natural geometric interpretation of the mixed Hessian and curvature conditions in terms of the canonical relation $\mathcal{C}$.

\subsubsection*{$L^p$ estimates and canonical graphs} Fix $X,Y$ smooth manifolds of dimension $n$ and a canonical relation $\mathcal{C} \subseteq T^*X\setminus 0 \times T^*Y \setminus 0$. Theorem \ref{fixed-time theorem} can easily be extended to the setting of global FIOs $\mathcal{F} \in I^{\mu}(X,Y;\mathcal{C})$ once the mixed Hessian condition is correctly interpreted in terms of the geometry of $\mathcal{C}$.

We first observe that in the specific context of a local operator $\mathcal{F}$ given by \eqref{FIO def again}, with $\dim X = \dim Y = n$ and a symbol $a \in S^{\mu}(\R^n \times \R^n)$, the order of $\mathcal{F}$ corresponds to the order $\mu$ of the symbol, and therefore $I^{\mu}(X,Y;\mathcal{C})$ is the correct class to work in if one wishes to extend the local fixed-time results described above. Indeed, by the convention established in \S\ref{global section} (which is motivated by the equivalence of phase theorem), the order $m$ of the operator satisfies 
\begin{equation}\label{relationship between number of variables}
    m = \mu - \frac{d - 2N}{4}
\end{equation}
where $d$ is the number of $(x;y)$ variables and $N$ is the number of Fourier variables. In the case of \eqref{FIO def again}, we have $d = 2n$ and $N = n$, and so $m$ and $\mu$ coincide. 

We now turn to describing the appropriate generalisation of the mixed Hessian condition.

\begin{projection condition} The natural projection mappings $\Pi_{T^*X} \colon \mathcal{C} \to T^*X\setminus 0$ and $\Pi_{T^*Y} \colon \mathcal{C} \to T^*Y\setminus 0$ are local diffeomorphisms.
\end{projection condition}
\begin{equation*}
\begin{tikzcd}
& \mathcal{C} \arrow[ld, swap, "\Pi_{T^*X}"]   \arrow[rd, "\Pi_{T^*Y}"] &   \\
T^*X \setminus 0 &  & T^*Y \setminus 0
\end{tikzcd}.
\end{equation*}

The projection condition clearly forces $\dim X = \dim Y$. It is also not difficult to show that if $\dim X = \dim Y$ and either one of the projections $\Pi_{T^*X}$ or $\Pi_{T^*Y}$ is a local diffeomorphism, then so too is the other.\footnote{This can be seen by expressing the operator in local coordinates: see the proof of Lemma \ref{projection condition lemma} below for a very similar argument.} Thus, for instance, an equivalent formulation of the projection condition is that $\dim X = \dim Y$ and
\begin{equation}\label{equivalent projection condition}
    \mathrm{rank} \,\ud\Pi_{T^*Y} = 2n,
\end{equation}
where $n$ is the common dimension of $X$ and $Y$. Yet another way to interpret this property, which will be useful later in the discussion, is that $\mathcal{C}$ is \textit{locally a canonical graph}. In particular, for every $\gamma_0=(x_0,\xi_0,y_0,\eta_0) \in \mathcal{C}$ there exists a symplectomorphism $\chi$ defined on a neighbourhood of $(x_0,\xi_0) \in T^*X\setminus 0$ and mapping into $T^*Y\setminus 0$ such that on this neighbourhood $\mathcal{C}$ is given by the graph
\begin{equation*}
\{(x,\xi, y, \eta): (y,\eta)=\chi(x,\xi)\}.
\end{equation*}

With this definition, the global variant of Theorem \ref{fixed-time theorem} reads thus.

\begin{theorem}[\cite{Seeger1991}]\label{global fixed-time theorem}
If $\mathcal{F} \in I^{\mu}(X,Y; \mathcal{C})$ is a global FIO where $\mathcal{C}$ satisfies the projection condition, then for all $1 < p < \infty$ the fixed-time estimate\footnote{Here, an $L^p_{\mathrm{comp}}(\R^n) \to L^p_{s, \,\mathrm{loc}}(\R^n)$ bound is interpreted as follows: for any pair of compact sets $\Omega_1, \Omega_2 \subseteq \R^n$ the \emph{a priori} estimate $\| \mathcal{F}f \|_{L^p_{s}(\Omega_2)} \lesssim_{\Omega_1, \Omega_2} \| f \|_{L^p(\Omega_1)}$ holds whenever $f \in C^{\infty}_c(\Omega_1)$.}
\begin{equation*}
    \| \mathcal{F}f \|_{L^p_{-\mu -\bar{s}_p, \,\mathrm{loc}}(\R^n)} \lesssim \| f \|_{L^p_{\mathrm{comp}}(\R^n)} 
\end{equation*}
holds.
\end{theorem}

Using the theory described in \S\ref{section: introduction}, it is not difficult to deduce Theorem \ref{global fixed-time theorem} as a fairly direct consequence of its local counterpart Theorem \ref{fixed-time theorem}. In particular, since the result is inherently local, one may assume that $\mathcal{F}\in I^\mu(X,Y;\mathcal{C})$ is given in local coordinates by some kernel
\begin{equation*}
    K(x;y)=\int_{\hat{\R}^N} e^{i \varphi (x,y;\theta)} a(x,y;\theta) \, \ud \theta,
\end{equation*}
where $a \in S^\mu (\R^n \times \R^n \times \hat{\R}^N )$ and $\varphi: \R^n \times \R^n \times \hat{\R}^N\backslash \{0\} \to \R$ is a non-degenerate phase function. Thus, one may write
\begin{equation}\label{local canonical relation}
    \mathcal{C} = \{(x,\partial_x \varphi(x,y;\theta), y, -\partial_y \varphi (x,y;\theta)) : \partial_\theta \varphi (x,y;\theta) =0\}
\end{equation}
and it follows that if $\mathcal{C}$ is a local canonical graph, then there exist smooth solutions to the equations
\begin{equation*}
    \xi=\partial_x \varphi (x,y;\theta), \qquad \partial_\theta \varphi (x,y;\theta)=0
\end{equation*}
in $(y,\theta)$. By the inverse function theorem, this amounts to the condition that the Jacobian of the map $(y,\theta) \mapsto (\partial_x \varphi (x,y;\theta), \partial_{\theta} \varphi (x,y;\theta))$ is non-vanishing: 
\begin{equation}\label{local graph coord}
    \det 
    \begin{pmatrix}
        \partial^2_{xy} \varphi & \partial^2_{x \theta} \varphi \\
        \partial^2_{\theta y} \varphi & \partial^2_{\theta \theta} \varphi
    \end{pmatrix}(x,y;\theta) \neq 0
    \qquad \textrm{whenever $\partial_\theta \varphi (x,y;\theta)=0$.}
\end{equation}
As described in \S\ref{global section}, one may further assume that $N = n$ and $\varphi$ has the special form $\varphi(x,y;\eta)=\langle y, \eta \rangle - \phi(x;\eta)$, where $\phi$ is smooth and homogeneous of degree 1 in $\eta$. In this case, the condition \eqref{local graph coord} then becomes 
\begin{equation*}
    \det \partial^2_{x \eta} \phi \neq 0,
\end{equation*}
which corresponds precisely with the mixed Hessian condition from \eqref{local graph condition}. 


\begin{example}[Variable coefficient averaging operators]\label{remark:RotCurv}
The class of FIOs of order $-\frac{n-1}{2}$ which satisfy the projection condition includes averaging operators over variable families of hypersurfaces which satisfy the \textit{rotational curvature} condition of Phong and Stein \cite{Phong1986} (see also \cite[Chapter XI $\S$3.1]{Stein1993}). Indeed, consider the family of hypersurfaces
\begin{equation*}
S_{x,t}=\{y \in \R^n : \Phi_t (x;y) =0\}
\end{equation*}
where $\Phi_t$ is a smooth defining function of $(t,x,y) \in [1,2] \times \R^n \times \R^n$. We say that $\Phi_t$ satisfies the rotational curvature condition if the Monge--Amp\`ere matrix associated to $\Phi_t$ is non-singular on $\Phi_t=0$: that is,
\begin{equation}\label{RotCurv condition}
 \mathrm{RotCurv}(\Phi_t)(x;y) :=  \det
    \begin{pmatrix}
        \Phi_t & \partial_y \Phi_t \\
        \partial_x \Phi_t & \partial^2_{xy} \Phi_t
    \end{pmatrix}(x;y)
    \neq 0 \qquad \textrm{whenever $\Phi_t(x;y)=0$.}
\end{equation}
As in Example \ref{averaging operator example}, the averaging operator
\begin{equation*}
A_t f(x):= \int_{\R^n} f(y) a(t,x,y) \delta(\Phi_t(x;y)) \, \ud y
\end{equation*}
may be written as
\begin{equation*}
A_t f(x):=\frac{1}{2\pi} \int_{\R^n} \int_\R e^{i \theta \Phi_t (x;y)} a(t,x,y) f(y) \, \ud \theta \, \ud y;
\end{equation*}
here $a \in S^0(\R \times \R^n \times \R^n)$. By Theorem \ref{thm:inv phases}, $A_t$ is a FIO of order $-\frac{n-1}{2}$ and one may readily verify that if $\Phi_t$ satisfies \eqref{RotCurv condition}, then the phase function $\varphi_t(x,y;\theta)=\theta \Phi_t(x;y)$ satisfies the condition \eqref{local graph coord}.
\end{example}

\begin{example}[Spherical averages]\label{example: sph rotcurv}
As a special case of the previous example, let $A_t$ denote the averaging operator associated to the family of spheres $x+t\mathbb{S}^{n-1}$, with defining function
\begin{equation*}
\Phi_t(x;y)=\frac{|x-y|^2}{t^2}-1.
\end{equation*}
In this case, $\mathrm{RotCurv}(\Phi_t)(x;y) = (-2)^{n+1}t^{-2n}$, which is non-vanishing. In general, whenever the operator is \emph{translation-invariant}, in the sense that the family of hypersurfaces is given by $x \mapsto x+tS_0$ for some fixed $S_0$, the rotational curvature is nonvanishing if and only if the Gaussian curvature of $S_0$ is non-vanishing.
\end{example}

\begin{example}[Radon transform]\label{remark: hyp rotcurv}
As another special case of Example \ref{remark:RotCurv}, consider the \textit{Radon transform} $A_t$ which is the averaging operator with defining function $\Phi_t(x;y)=\langle x , y \rangle -t$ for some $t\neq 0$. Observe that $\mathrm{Rot Curv}(\Phi_t)(x;y)=-\langle x, y \rangle$, so that the rotational curvature condition is satisfied. However, in contrast with Example \ref{example: sph rotcurv}, each hyperplane $S_{x,t} = \{y \in \R^n : \Phi_{t}(x;y) = 0\}$ has zero Gaussian curvature. In this case, the rotational curvature is capturing the rotation of the planes $S_{x,t}$ as $x$ varies, rather than curvature of the planes themselves.
\end{example}

\subsubsection*{Local smoothing estimates and cinematic curvature condition} Fix $Y$ and $Z$ smooth manifolds of dimension $n$ and $n+1$ respectively, with $n \geq 2$, and let $\mathcal{C}$ be a canonical relation in $T^*Z \backslash 0 \times T^*Y \backslash 0$. Thus, $\mathcal{C}$ is a conic submanifold of dimension $2n+1$ which is Lagrangian with respect to the 1-form $\omega_Z-\omega_Y= \sum_{j=1}^{n+1} \zeta_j \ud z_j - \sum_{i=1}^{n} \eta_i \ud y_i$. The local smoothing estimates in Theorem \ref{thm:LS} hold for global FIOs $\mathcal{F} \in I^{\mu-1/4}(Z,Y;\mathcal{C})$ which satisfy certain conditions on $\mathcal{C}$. 

Note, in contrast with the fixed-time estimates described above, here one works with operators of order $\mu - 1/4$ so that the FIO in that class admit the local expression \eqref{FIO def n to n+1} with a symbol $a \in S^{\mu}(\R^{n+1} \times \R^n)$ of order $\mu$. This is a quirk of the order convention from \S\ref{global section}. Indeed, if we consider the local operator \eqref{FIO def n to n+1}, which is interpreted as mapping functions of $n$ variables to functions of $n+1$ variables, the number $d$ of $(x,t,y)$ variables is equal to $2n + 1$ whilst the number $N$ of Fourier variables is $n$. Thus, recalling \eqref{relationship between number of variables}, we see the order $m$ of the operator \eqref{FIO def n to n+1} is indeed $\mu - 1/4$.\footnote{If $\mathcal{F}$ is viewed as a 1-parameter family of operators $(\mathcal{F}_t)_{t \in I}$, then each $\mathcal{F}_t$ is a FIO of order $\mu$.}

We now turn to describing the hypotheses on the canonical relation $\mathcal{C}$ which generalise properties H1) and H2) from the local theory. The first hypothesis corresponds to the mixed hessian condition H1) and is the natural analogue of the projection condition featured in Theorem \ref{global fixed-time theorem} (see \eqref{equivalent projection condition}). 

\begin{projection condition} If $\Pi_{T^*Y} \colon \mathcal{C} \to T^*Y \setminus 0$ denotes the natural projection mapping, then
\begin{equation}\label{nondeg 1}
    \mathrm{rank}\,\ud \Pi_{T^*Y} = 2n . 
\end{equation}
\end{projection condition}

\begin{equation*}
\begin{tikzcd}
& \mathcal{C} \arrow[ld, swap, "\Pi_{T^*_{_{}}Y}"]  \arrow[d, "\Pi_{Z}"] \arrow[rd, "\Pi_{T^*_{z_0}Z}"] &   \\
T^*Y\setminus 0 & Z & T^*_{z_0}Z \setminus 0
\end{tikzcd}
\end{equation*}

 Geometrically, this condition has the following interpretation. Fix $z_0 \in \Pi_Z(\mathcal{C})$ and let $\Pi_{T_{z_0}^* Z}$ denote the projection $\mathcal{C} \to T_{z_0}^* Z \backslash 0$. Define
\begin{equation*}
\Gamma_{z_0}:=\Pi_{T^*_{z_0}Z}(\mathcal{C}),
\end{equation*}
which is a conic subset of $T^*_{z_0} Z \backslash 0$. The projection condition implies that $\Gamma_{z_0}$ is in fact a \emph{smooth $n$-dimensional surface}. Indeed, this is a consequence of the following lemma.

\begin{lemma}\label{projection condition lemma} The condition \eqref{nondeg 1} implies that $\ud \Pi_{T_{z_0}^*Z}$ has constant rank $n$.
\end{lemma} 

\begin{proof} Recall from \S\ref{global section} that the operator can be expressed locally in the form \eqref{FIO def again}, in which case the phase function $\varphi$ appearing in the expression \eqref{local canonical relation} is given by $\varphi(x,y; \xi) := \phi(x; \xi) - \langle y, \xi \rangle$. In particular, local coordinates may be chosen so that $\mathcal{C}$ is locally parametrised as a graph (modulo a reflection)
\begin{equation}\label{local form of canonical relation}
(z, \eta) \mapsto (z, \partial_z \phi(z;\eta), \partial_{\eta} \phi(z;\eta) ,-\eta),
\end{equation}
where $\phi$ is homogeneous in $\eta$. Thus, computing the differential of $\Pi_{T^*Y}$ in these coordinates, the condition \eqref{nondeg 1} implies that the map $(z,\eta) \mapsto (\partial_{\eta} \phi(z;\eta), \eta )$ is a submersion; this of course reduces to 
\begin{equation}\label{full rank}
    \rank \partial_{z\eta}^2 \phi(z,\eta) = n.
\end{equation}
By combining \eqref{local form of canonical relation} and \eqref{full rank}, it immediately follows that the differentials of $\Pi_{T_{z_0}^*Z}$ must have rank $n$, as required. 
\end{proof}

The second condition concerns the curvature of the cones $\Gamma_{z_0}$.

\begin{cone condition} For every $z_0 \in \Pi_{Z}(\mathcal{C})$ the cone $\Gamma_{z_0}$ has $n-1$ non-vanishing principal curvatures at every point.
\end{cone condition}

 If $\mathcal{C}$ satisfies both the projection and the cone condition, then, following \cite{Sogge1991}, it is said to satisfy the \textit{cinematic curvature condition}.

\begin{theorem}[\cite{BHS}]\label{global local smoothing theorem}
Suppose $\mathcal{F} \in I^{\mu-1/4}(Y,Z; \mathcal{C})$ is a global FIO where $\mathcal{C} \subset T^*Y\setminus 0 \times T^*Z \setminus 0$ satisfies the cinematic curvature condition. If $ \frac{2(n+1)}{n-1} \leq p < \infty$, then
\begin{equation*}
\Big( \int_1^2 \| \mathcal{F} f(\,\cdot\,,t) \|_{L^p_{-\mu - \bar{s}_p + \sigma,\,\mathrm{loc}}(\R^{n})}^p \, \ud t \Big)^{1/p} \lesssim \| f \|_{L^p_{\mathrm{comp}}(\R^n)}
\end{equation*}
holds for all $\sigma < 1/p$.
\end{theorem}

Once again, it is not difficult to deduce Theorem \ref{global local smoothing theorem} as a direct consequence of its local counterpart, Theorem \ref{thm:LS}. Most of this argument has already been described in the proof of the first claim above. In particular, in local coordinates one may express $\mathcal{C}$ as a graph as in \eqref{local form of canonical relation}. The projection condition then implies \eqref{full rank}, which is precisely the condition H1) in the local theorem. On the other hand, the cones $\Gamma_{z_0}=\Pi_{T^*_{z_0}Z(\mathcal{C})}$ take the form \eqref{cones in terms of phase}, and so the cone condition clearly amounts to H2).



\begin{example}[Variable coefficient averaging operators]\label{remark:CinCurv} We return to the variable hypersurfaces $S_{x,t}$ and associated averaging operators $A_t$ discussed in Example \ref{remark:RotCurv}. Suppose that the defining function $\Phi_t$ satisfies the rotational curvature condition \eqref{RotCurv condition} for all $t$ in the $t$-support of $a$. Thus, $A_t \in I^{-(n-1)/2}(X,Y;\mathcal{C}_t)$ for a canonical relation $\mathcal{C}_t$ which is locally a canonical graph. Note that the rotational curvature condition applies to each $A_t$ individually and, in particular, does not take into account how the family of surfaces $S_{x,t}$ vary in $t$.

The cinematic curvature condition, on the other hand, provides additional information about the behaviour of the $S_{x,t}$ under changes of $t$. Indeed, let $\mathcal{C}$ denote the canonical relation associated to the \emph{family} of averages $A_t$ (viewed as an operator taking functions on $\R^n$ to functions on $\R^{n+1}$). It follows from the rotational curvature hypothesis that
\begin{equation*}
    \mathcal{C}=\{(x,t,\xi,\tau,y,\eta): (y,\eta)=\chi_t(x,\xi), \tau=q(x,t,\xi)\}
\end{equation*}
where:
\begin{itemize}
    \item $\chi_t$ is a symplectomorphism
    \item the function $q$ is homogeneous of degree 1 in $\xi$ and smooth if $\xi \neq 0$.
\end{itemize}
Indeed, the function $\chi_t$ arises from the canonical graph property, satisfied by each $\mathcal{C}_t$. Note that the variable $\tau$ may be written in terms of $x,t$ and $\xi$ because $\chi_t$ is a diffeomorphism and $\mathcal{C}$ is a $2n+1$ dimensional manifold. Moreover, $q$ is necessarily homogeneous of degree 1 in $\xi$ due to the conic nature of $\mathcal{C}$ in the $\eta$ variable. Having written the canonical relation in the above form, the cone condition requires that
\begin{equation*}
\rank \partial^2_{\xi \xi} q=n-1,
\end{equation*}
which is the maximum possible rank in view of the homogeneity of $q$. This additional hypothesis takes into account the change in $t$. 

Finally, if one represents the averaging operator using a single Fourier variable, as in Example \ref{alternative averaging operator example}, then it is possible to obtain a formula for computing the function $q$. Indeed, the phase function is given by $\varphi(x,t,y;\theta)=\theta \Phi_t(x;y)$ and so in $\mathcal{C}$ we have
\begin{equation*}
    \tau=\partial_t \varphi (x,t,y;\theta) = \theta \partial_t \Phi_t (x;y) \quad \textrm{and} \quad \xi=\partial_x \varphi(x,t,y;\theta)= \theta \partial_x \Phi_t(x;y).
\end{equation*}
The condition $\tau=q(x,t,\xi)$ therefore becomes 
\begin{equation*}
q(x, t, \partial_x \Phi_t(x;y))=\partial_t \Phi_t(x;y) \qquad \text{whenever} \quad \Phi_t(x;y)=0,
\end{equation*}
due the homogeneity of $q$ in the $\xi$ variable.

%
\end{example}

\begin{example}[Spherical averages]
For $t > 0$ let $A_t$ denote the averaging operator associated to the defining function $\Phi_t(x;y)=\frac{|x-y|^2}{t^2}-1$. It was observed in Example \ref{example: sph rotcurv} that each $A_t$ satisfies the rotational curvature condition. Moreover, the family of operators satisfies the cinematic curvature condition, since $q(x,t;\xi)=-|\xi|$.
\end{example}

\begin{example}[Radon transform]\label{Radon revisited}
For $t \neq 0$ let $A_t$ denote the averaging operator associated to the defining function $\Phi_t(x;y)=\langle x , y \rangle -t$. It was observed in Example \ref{remark: hyp rotcurv} that each $A_t$ satisfies the rotational curvature condition. However, the cinematic curvature condition is violated, as there is no change in the curvatures of the $S_{x,t}$ as $t$ varies. In particular, $q(x,t;\xi)=-\frac{\langle x, \xi \rangle}{t}$, so that $\partial^2_{\xi \xi} q = 0$.

Incidentally, for $n=2$ this example can also be used to show the necessity of the cinematic curvature hypothesis for local smoothing (see, for instance, \cite[Chapter 6]{Sogge2017}). 

\end{example}




\section{Local smoothing and maximal estimates}\label{section: maximal estimates}

In the next two sections we investigate some of the many applications of local smoothing estimates to problems in harmonic analysis. Here we review connections with (maximal) Bochner--Riesz multipliers and circular maximal function theorems.




\subsection{Bochner--Riesz estimates}\label{subsec:BR}

Recall that the Bochner--Riesz multipliers of order $\delta > 0$ are defined by
\begin{equation*}
S_t^{\delta} f(x) := \frac{1}{(2\pi)^{n}}\int_{\hat{\R}^n} e^{i \langle x,  \xi \rangle} (1-|t\xi|)^\delta_+ \, \Hat f(\xi) \, \ud\xi \qquad \textrm{for $t > 0$}.
\end{equation*}
A classical problem in harmonic analysis is to determine whether these multipliers constitute a Fourier summation method: in particular, one is interested in whether
\begin{equation*}
    S_t^{\delta} f \to f \quad \textrm{as} \quad t \to 0_+
\end{equation*}
for a given mode of convergence (typically convergence in $L^p$ or almost everywhere convergence). By a simple rescaling argument, together with some standard functional analysis, the $L^p$ convergence question is equivalent to determining the range of $L^p$ boundedness for the operators $S^{\delta} := S^{\delta}_1$ (see, for instance, \cite[Chapter IX]{Stein1993} for further details). 

\begin{conjecture}[Bochner--Riesz conjecture]\label{conjecture:BR}
Let $1 \leq p \leq \infty$. If $\delta>\delta(p) := \max \{n|\tfrac12-\tfrac 1p|-\tfrac12, 0\}$, then
\begin{equation}\label{1.1}
\| S^\delta f \|_{L^p(\R^n)} \lesssim \| f \|_{L^p(\R^n)}.
\end{equation}
\end{conjecture}
It is known that $\delta>\delta(p)$ is a necessary condition for \eqref{1.1} to hold whenever $p\ne 2$.  The results for $p=\infty$ are trivial and it is also well known that one would obtain this conjecture for all $1\le p\le\infty$ by interpolation and duality if the bounds held for $p\ge \frac{2n}{n-1}$.

It was observed by the third author \cite{Sogge1991} that the local smoothing conjecture for $e^{i t \sqrt{-\Delta}}$ formally implies Conjecture \ref{conjecture:BR}.

\begin{proposition}\label{prop:LS implies BR}
Let $ \frac{2n}{n-1} \leq p < \infty$ be given. If there is $1/p-$ local smoothing for $e^{i t \sqrt{-\Delta}}$, then the Bochner--Riesz estimate \eqref{1.1} holds for all $\delta > \delta(p)$.
\end{proposition}

It is remarked that the Bochner--Riesz conjecture is itself known to imply the Fourier restriction conjecture for spheres and paraboloids, which in turn implies the Kakeya conjecture: see \cite{Tao1999, Wolff1999} for a discussion of these problems and the relationships between them. Thus, we see that the local smoothing conjecture sits at the top of a chain of implications relating important central questions in harmonic analysis and geometric measure theory.
\begin{equation*}
    \textrm{Local smoothing} \Rightarrow \textrm{Bochner--Riesz} \Rightarrow \textrm{Restriction} \Rightarrow \textrm{Kakeya}.
\end{equation*}

\begin{proof}[Proof (of Proposition \ref{prop:LS implies BR})] Note that
\begin{equation}\label{1.2}
\bar s_p-1/p=\delta(p) \quad \text{if } \quad p\ge \frac{2n}{n-1}
\end{equation}
and that one may write
\begin{equation}\label{1.22}
(1-|\xi|)^\delta_+ = r(|\xi|) +\sum_{k=1}^\infty 2^{-k\delta} \psi\bigl(2^k(1-|\xi|)\bigr),\end{equation}
where $r=r_\delta \in C^\infty_0([0,\infty))$ and $\psi=\psi_\delta\in C^\infty_0([1/2,2])$.

Since $r$ is smooth and compactly supported the Fourier multiplier operator associated with $r(|\xi|)$ is bounded on $L^p(\Rn)$ for all $1\le p\le \infty$, and one concludes that \eqref{1.1} would follow for a given $p\ge \frac{2n}{n-1}$ if the inequality
\begin{equation*}
    \Big\|\int_{\hat{\R}^n} e^{i \langle x,  \xi \rangle} \psi\bigl(2^k(1-|\xi|)\bigr) \, \Hat f(\xi) \, \ud\xi\Big\|_{L^p(\R^n)} \lesssim_{\varepsilon} 2^{k(\delta(p)+\e)} \, \|f\|_{L^p(\R^n)} 
\end{equation*}
holds for all $k \in \N$ and all $\varepsilon > 0$. By a simple change of variables argument, the inequality in the above display holds if and only if
\begin{equation}\label{1.3}
\| A^{\lambda}f \|_{L^p(\R^n)} \lesssim_{\varepsilon} \la^{\delta(p)+\e}\, \|f\|_{L^p(\R^n)} \qquad \textrm{for all $\la \gg 1$}
\end{equation}
where
\begin{equation*}
A^{\lambda}f(x) := \int_{\hat{\R}^n} e^{i \langle x, \xi \rangle} \,  \psi(\la-|\xi|) \, \Hat f(\xi) \, \ud\xi.
\end{equation*}
In proving \eqref{1.3} for a given $\la \gg 1$, since $\text{supp}\, \psi\in [1/2,2]$, one may assume that
\begin{equation}\label{1.4}
\text{supp}\, \Hat f\subset \{\xi: \, |\xi|\in [\la/2,2\la] \, \}.
\end{equation}
Also, if one writes
$$\psi(\la-|\xi|) = (2\pi)^{-1}\int_{\R} \check{\psi}(t) e^{-i\la t} \, e^{it|\xi|} \, \ud t,$$
then H\"older's inequality in the $t$-variable after multiplying and dividing by $(1+|t|)$ implies that
\begin{align}
\label{1.5}
\| A^{\lambda}f \|_{L^p(\R^n)} &\lesssim \Bigl\| (1+|t|) \check{\psi}(t)  \int_{\hat{\R}^n} e^{i \langle x, \xi \rangle} e^{it|\xi|} \, \Hat f(\xi)\, \ud\xi \, \Bigr\|_{L^p(\Rn\times \R)} \\
\nonumber
&\lesssim_N \bigl\| (1 + |t|)^{-N} e^{it\sqrt{-\Delta}}f\, \bigr\|_{L^p(\Rn \times \R)}.
\end{align}
for all $N \in \N$ and one may write
\begin{equation}\label{eq:split}
    \| A^{\lambda}f \|_{L^p(\R^n)} \lesssim_N \bigl\| e^{it\sqrt{-\Delta}}f\, \bigr\|_{L^p(\Rn \times [-1,1])} + \sum_{k \in \N} \bigl\| (1 + |t|)^{-N} e^{it\sqrt{-\Delta}}f\, \bigr\|_{L^p(\Rn \times I_k)},
\end{equation}
where $I_k:=[-2^{k-1}, -2^{k}] \cup [2^{k-1}, 2^k]$. In view of \eqref{1.4} and \eqref{1.2}, $1/p-$ local smoothing for $e^{it\sqrt{-\Delta}}$ implies that
\begin{equation}\label{1.6}
\bigl\|e^{it\sqrt{-\Delta}} f\bigr\|_{L^p( \Rn\times [-1,1])} \lesssim_\e \la^{\delta(p)+\e} \, \|f\|_{L^p(\Rn)}.
\end{equation}
Thus, the first term in the right-hand-side of \eqref{eq:split} is controlled by the right-hand side of \eqref{1.3}. For the remaining terms, the rapid decay in \eqref{1.5} together with \eqref{1.6} and a simple change of variables
argument yields
$$
\bigl\| (1 + |t|)^{-N} e^{it\sqrt{-\Delta}}f\, \bigr\|_{L^p(\Rn \times I_k)} \lesssim 
2^{-k}\la^{\delta(p)+\e}\|f\|_p
$$
uniformly in $k \in \N$, and then the desired result just follows from summing a geometric series in $k \in \N$.
\end{proof}




\subsection{Maximal Bochner--Riesz estimates}\label{subsec:MBR}

When studying almost everywhere convergence of the Bochner--Riesz summation method, one naturally considers the maximal estimates
\begin{equation}\label{l.4}
\Bigl(\, \int_{\R^n} \sup_{t>0} |S^\delta_tf(x)|^p \, \ud x \, \Bigr)^{1/p}\lesssim \|f\|_{L^p(\R^n)}
\end{equation}
for the operators $S_t^{\delta}$. It transpires that $1/p-$ local smoothing for $e^{i t \sqrt{-\Delta}}$ also implies inequalities of this form.

\begin{proposition}\label{maxboch}  Let $ \frac{2n}{n-1} \leq p < \infty$ be given. If there is $1/p-$ local smoothing
for $e^{it\sqrt{-\Delta}}$, then the maximal Bochner--Riesz estimate \eqref{l.4} holds for all $\delta>\delta(p)$.
\end{proposition}

To prove this, note that if $r(|\xi|)$ is as in \eqref{1.22}, then
$$\sup_{t>0} 
\Bigl| \, \int_{\hat{\R}^n} e^{i \langle x, \xi \rangle} \, r(t|\xi|) \, \Hat f(\xi) \, \ud\xi\, \Bigr| \lesssim Mf (x),$$
where $M$ denotes the Hardy--Littlewood maximal function.  Since $\|Mf\|_p\lesssim \|f\|_p$ for all $p>1$, the previous arguments reveal that \eqref{l.4} would follow if one can show that the maximal version of \eqref{1.3} holds. Explicitly, it suffices to show that
\begin{equation}\label{l.5}
\Bigl(\, \int_{\R^n} \sup_{t>0} \, |A^\la_t f(x)|^p \, \ud x\, \Bigr)^{1/p} \lesssim_\e \la^{\delta(p)+\e}\|f\|_{L^p(\R^n)}
\end{equation}
holds for all $\lambda \gg 1$ and all $\varepsilon > 0$ where 
\begin{equation*}
A^\la_tf(x) := \frac{1}{(2\pi)^{n}}\int_{\hat{\R}^n} e^{i \langle x, \xi \rangle} \psi\bigl(\la-|t\xi|\bigr) \, \Hat f(\xi) \, \ud\xi.
\end{equation*}

To prove this, we appeal to the following simple lemma.

\begin{lemma}\label{maxdyad}
Suppose that
$$\mathrm{supp}\,m\subset \{\xi\in 
\hat{\R}^n: \, \,
|\xi|\in (\la_0/2,2\la_0)\, \}$$
for some fixed $\la_0>0$ and set
$$A_tf(x) := \frac{1}{(2\pi)^{n}}\int_{\hat{\R}^n} e^{i \langle x, \xi \rangle}
\, m(t\xi)\, \Hat f(\xi)\, \ud\xi, \, \, \, t>0.$$
If $2\le p<\infty$ and the inequality
\begin{equation}\label{l.1}
\Bigl(\, \int_{\R^n} \, \sup_{t\in [1,2]}
\, |A_tf(x)|^p \, \ud x\, \Bigr)^{1/p}\le \bar{C}_p
\|f\|_{L^p(\Rn)},
\end{equation}
holds for some fixed constant $\bar{C}_p >0$, then it follows that
\begin{equation*}
\Bigl(\, \int_{\R^n} \, \sup_{t>0} \, |A_tf(x)|^p\, \ud x \, \Bigr)^{1/p} \lesssim C_p \, \|f\|_{L^p(\Rn)}.
\end{equation*}
\end{lemma}

\begin{proof}
Let $k_0\in {\mathbb Z}$ be the unique integer such that $\la_0\in [2^{k_0},2^{k_0+1})$.
Next, choose a Littlewood--Paley bump function
$\beta\in C^\infty_0((1/2,2))$ satisfying
$\sum_{-\infty}^\infty \beta(2^{-j}r)=1$, $r>0$, and define
$P_\ell f$ by 
$(P_\ell f)\;\widehat{}\;(\xi) := \beta(2^{-\ell}|\xi|)\Hat f(\xi)$.
Then, by
Littlewood--Paley theory (see, for instance, \cite[Chapter IV]{Stein1970}),
\begin{equation}\label{l.3}
\bigl\| \, (\sum_{\ell=-\infty}^\infty |P_\ell f|^2)^{1/2}\,
\bigr\|_{L^p(\Rn)}\lesssim \|f\|_{L^p(\Rn)} \quad
\textrm{for all $1<p<\infty$.}
\end{equation}
To use this, note that
$$A_tf=A_t\bigl( \sum_{\{\ell \in \Z : \, |\ell-(k+k_0)|\le 10\, \}} P_\ell f
\bigr) \quad \text{if } \, \, t\in [2^{-k},2^{-k+1}],$$
since it is assumed that $m(\xi)=0$ if $|\xi|\notin (2^{k_0-2},2^{k_0+2})$. By this observation together with a scaling argument, the assumption \eqref{l.1} yields
\begin{align*}
\int_{\R^n} \sup_{t\in [2^{-k},2^{-k+1}]} |A_tf(x)|^p\, \ud x
&= \int_{\R^n} \sup_{t\in [2^{-k},2^{-k+1}]}
\bigl| A_t(\sum_{|\ell -(k+k_0)|\le 10}P_\ell f(x)) \bigr|^p\, \ud x
\\
&\le \bar{C}_p^p \int_{\R^n} \bigl| \, \sum_{|\ell-(k+k_0)|\le 10}P_\ell f(x)\, \bigr|^p\, \ud x \\
&\lesssim \bar{C}_p^p \sum_{|\ell-(k+k_0)|\le 10} \int_{\R^n} |P_\ell f(x)|^p \, \ud x
\end{align*}
for all $k \in \Z$. Consequently,
\begin{align*}
\int_{\R^n} \sup_{t>0} |A_tf(x)|^p \, \ud x
&\le \sum_{k=-\infty}^\infty \int_{\R^n} \sup_{t\in [2^{-k},2^{-k+1}]} |A_tf(x)|^p \, \ud x
\\
&\lesssim \bar{C}_p^p \sum_{k=-\infty}^\infty \sum_{|\ell-(k+k_0)|\le 10}
\int_{\R^n} |P_\ell f(x)|^p \, \ud x
\\
&\lesssim \bar{C}_p^p \int \bigl( \, \sum_\ell |P_\ell f(x)|^2\, \bigr)^{p/2} \, \ud x
\\
&\lesssim \bar{C}_p^p \|f\|_{L^p(\R^n)}^p.
\end{align*}
Here it was used that $p\ge 2$
yields $\ell^2 \subseteq \ell^p$ in the second to last
inequality and \eqref{l.3} in the last inequality.
\end{proof}

\begin{proof}[Proof (of Proposition \ref{maxboch})] The preceding observations together with the above lemma reduce the proof of \eqref{l.5} to showing that for $\la\gg1$ one has
$$\Bigl( \, \int_{\R^n} \sup_{1\le t\le 2}\bigl| A^\la_tf(x)\bigr|^p \, \ud x\, \Bigr)^{1/p} \lesssim_\e \la^{\delta(p)+\e}
\|f\|_{L^p(\R^n)}.$$
This in turn would follow by showing that 
\begin{equation}\label{l.6}
\Bigl(\, \int_{\R^n} \bigl| (A^\la_{t(x)}f)(x)\bigr|^p \, \ud x\, \Bigr)^{1/p} \lesssim_\e \la^{\delta(p)+\e}\|f\|_{L^p(\R^n)}
\end{equation}
holds for any measurable function $t(x): \, \Rn \to [1,2]$. 

To adapt the earlier argument, write
\begin{align*}
\psi\bigl(\la-|t(x)\xi|\bigr)&=(2\pi)^{-1}\int_{\R} \check{\psi}(s) \, e^{-i\la s} e^{is t(x)|\xi|} \, \ud s
\\
&=(2\pi t(x))^{-1}\int_{\R} \check{\psi}(s/t(x)) \, e^{-i\la s/t(x)} e^{is|\xi|} \, \ud s.
\end{align*}
Since it is assumed that $1\le t(x)\le 2$ and since $\check{\psi}$ is rapidly decreasing, one can use H\"older's inequality and argue as in \eqref{1.5}
to see that the left-hand-side of \eqref{l.6} is dominated by
$$\bigl\|(1+|s|)^{-N} e^{is\sqrt{-\Delta}}f\bigr\|_{L^p(\Rn \times \R)}$$
for all $N \in \N$ (with a constant depending on $N$).

Repeating the earlier arguments, one sees that the $1/p-$ local smoothing estimates \eqref{1.6} imply that this last expression is dominated by $\la^{\delta(p)+\e}\|f\|_{L^p(\R^n)}$. This establishes the desired estimate \eqref{l.6} and thereby finishes the proof of Proposition \ref{maxboch}.
\end{proof}




\subsection{Circular maximal function estimates}\label{subsec:circular}

One may use local smoothing estimates
for the half-wave propagator in two spatial dimensions
to give an alternative proof of Bourgain's celebrated circular maximal function theorem
\cite{Bourgain1986}. Letting $\sigma_t$ denote the normalised Lebesgue measure on the dilated circle $t\mathbb{S}^1$, recall that this theorem states the following:

\begin{theorem}[Bourgain \cite{Bourgain1986}]\label{cicular maximal theorem} For all $p > 2$,
\begin{equation}\label{b.1}
\Bigl(\, \int_{{\mathbb R}^2}
\sup_{t>0} | f \ast \sigma_t (x)|^p
\, \ud x \, \Bigr)^{1/p}\lesssim \|f\|_{L^p({\mathbb R}^2)}.
\end{equation}
\end{theorem} 

Theorem \ref{cicular maximal theorem} extends Stein's \cite{Stein1976} earlier spherical maximal theorem which states that for $n\ge 3$ the maximal operator associated with spherical averages in $\Rn$ is bounded on $L^p(\Rn)$ for $p > \frac{n}{n-1}$. Stein~\cite{Stein1976} also showed that these bounds fail for any $p\le \frac{n}{n-1}$. In particular, the circular maximal function featured in \eqref{b.1} is not bounded on $L^2({\mathbb R}^2)$; this partially accounts for the added difficulties in two dimensions (which were later overcome by Bourgain \cite{Bourgain1986}). 

It is remarked that, in contrast to the applications featured in the previous sections, to prove the sharp maximal function result from local smoothing estimates one does \emph{not} require a sharp gain in regularity in the hypothesised local smoothing estimates; in fact, as will be discussed in Remark \ref{remark: not 1/p needed} below, \emph{any} non-trivial gain in regularity over the fixed-time estimate for \emph{any} $2 < p < \infty$ yields the sharp maximal inequality (however, for concreteness, we shall work with the $L^6$ local smoothing estimate). That local smoothing estimates can be used to give an alternative proof of \eqref{b.1} was first observed by Mockenhoupt, Seeger and the third author in \cite{Mockenhaupt1992}.

\begin{proof}[Proof (of Theorem \ref{cicular maximal theorem})] It suffices to prove the maximal estimates for nonnegative $f$.  Also, since the result is trivial when $p=\infty$, it suffices to prove the
bounds under this assumption when $2<p<\infty$.

Let $\beta\in C^\infty_0((1/2,2))$ be
the Littlewood--Paley bump function occurring in the proof
of Lemma~\ref{maxdyad}. As discussed in Example \ref{averaging operator example}, the Fourier transform
of the arc length measure $\sigma$ on $\mathbb{S}^1$ may be written as
\begin{equation*}
\hat{\sigma}(\xi)=
 \sum_\pm a_\pm(|\xi|) e^{\pm i|\xi|}
\end{equation*}
where $a_{\pm} \in S^{-1/2}$. Consequently, if $\beta_0(|\xi|) := 1-\sum_{j=1}^{\infty}\beta(2^{-j}|\xi|)
\in C^\infty_0$, one may write
\begin{multline}\label{b.4}
f \ast \sigma_t (x)
=(2\pi)^{-2}\int_{\hat{\R}^2} e^{i \langle x, \xi \rangle}
\beta_0(t|\xi|) \hat{\sigma}(t \xi ) \, 
\Hat f(\xi) \, \ud\xi
+(2\pi)^{-2}\sum_\pm\sum_{j=1}^\infty F^j_\pm (x,t)
\end{multline}
where
\begin{equation*}
F^j_\pm(x,t) := \int_{\hat{\R}^2} e^{i \langle x, \xi \rangle \pm it|\xi|}
\beta(2^{-j}t|\xi|) \, a_\pm(t|\xi|)
\, \Hat f(\xi) \, \ud\xi.
\end{equation*}
Note that the $F^j_\pm$ correspond to the half-wave propagator $e^{i t \sqrt{-\Delta}}$ except for the choice of symbol and the frequency localisation to the dyadic scale $2^j$.

Since $\beta_0 \hat{\sigma}\in C^\infty_0$, it follows
that the maximal operator associated with the first term
in the right-hand side of \eqref{b.4} is dominated by the
Hardy--Littlewood maximal function of $f$ and thus 
bounded on all $L^p({\mathbb R}^2)$ for $p>1$.  It therefore suffices to show that for all $2<p<\infty$ the remaining
terms satisfy
\begin{equation*}
\Bigl(\, \int_{{\mathbb R}^2} \,
\sup_{t>0} \, | F^j_\pm (x,t) |^p
\, \ud x\, \Bigr)^{1/p}
\lesssim 2^{-j\e_p}\|f\|_{L^p({\mathbb R}^2)}
\end{equation*}
for some $\varepsilon_p > 0$. On account of the support properties of $\beta$ and
Lemma~\ref{maxdyad}, it suffices to show that
\begin{equation}\label{b.6}
\Bigl( \int_{{\mathbb R}^2} 
\sup_{1\le t\le2} \, | F^j_\pm (x,t)|^p
 \ud x \Bigr)^{1/p}
\lesssim 2^{-j\e_p}\|f\|_{L^p({\mathbb R}^2)}.
\end{equation}
To prove this, we appeal to the following elementary lemma.

\begin{lemma}\label{sob}  Suppose that $F\in C^1(\R)$ and that $p>1$.  Then
\begin{equation}\label{b.7}
\sup_{1\le t\le 2}|F(t)|^p
\le |F(1)|^p + p\Bigl(\, \int_1^2 |F(t)|^p \, \ud t\, \Bigr)^{(p-1)/p}
\Bigl(\, \int_1^2 |F'(t)|^p \, \ud t\, \Bigr)^{1/p}.
\end{equation}
\end{lemma}

\begin{proof} The proof of \eqref{b.7} is very simple.  If one first
writes
$$|F(t)|^p = |F(1)|^p +p\int_1^t |F(s)|^{p-1} \cdot
F'(s) \, \ud s,$$
then \eqref{b.7} follows via H\"older's inequality.
\end{proof}

To use this lemma to prove \eqref{b.6} we shall exploit the fact that the operators $F^j_{\pm}$ have symbols of order $-1/2$ which are localised to frequencies $|\xi|\sim 2^j$. Consequently, the fixed-time estimates \eqref{fixed time FIO} for $e^{i t \sqrt{-\Delta}}$ give
$$\Bigl(\, \int_{\R^2} |F^j_\pm(x,1)|^p\, \ud x\, \Bigr)^{1/p}
\lesssim 2^{-j(1/2-\bar s_p)}\|f\|_{L^p({\mathbb R}^2)} \quad \textrm{for all $1<p<\infty$.}$$
Note that in two dimensions $1/2-\bar s_p>0$. As a result, by H\"older's inequality after integrating \eqref{b.7} in the $x$-variables, it suffices to prove that for all $2 < p < \infty$ the inequality
\begin{multline}\label{b.9}
\Bigl( \int_1^2\int_{\R^2} \bigl|F^j_\pm(x,t)\bigr|^p  \ud x \,\ud t
\Bigr)^{(p-1)/p}\,
\Bigl(\int_1^2\int_{\R^2} \bigl| \frac{d}{dt}F^j_\pm(x,t)\bigr|^p
 \ud x \, \ud t  \Bigr)^{1/p}
\lesssim 2^{-jp\e_p}\|f\|_{L^p({\mathbb R}^2)}
\end{multline}
holds for some $\varepsilon_p > 0$.

Using, for instance, the  $1/6-$ $L^6$ local smoothing
estimates for the half-wave operators $e^{it\sqrt{-\Delta}}$ in $\R^2$ from Theorem \ref{thm:LS} one has
\begin{equation}\label{b.10}
\Bigl(\, \int_1^2 \int_{\R^2} |F^j_\pm(x,t)|^6 \, \ud x \, \ud t \, 
\Bigr)^{1/6}\lesssim_\e 2^{(-\frac13+\e)j}\|f\|_{L^6(\R^2)} \, \, \, \,
\text{for } \, \e>0.
\end{equation}
Here we use the fact that $F^j_\pm$ incorporates a symbol
of order $-1/2$, is frequency localised at scale $2^j$
and $-1/2+(\bar s_p-1/p)=-1/3$ if $p=6$.  Since
$\tfrac{d}{dt}F^j_\pm$ is as $e^{i t \sqrt{-\Delta}}$ with symbol of order
$1/2$, one similarly obtains
\begin{equation}\label{b.11}
\Bigl(\, \int_1^2\int_{\R^2} \bigl| \, \frac{d}{dt}F^j_\pm(x,t)
\bigr|^6 \, \ud x \, \ud t\, \Bigr)^{1/6}\lesssim_\e
2^{(\frac23+\e)j}\|f\|_{L^6(\R^2)} \, \, \, \,\text{for } \, \e>0,
\end{equation}
using the fact that $2/3=(\bar s_p-1/p)+1/2$ if $p=6$. Clearly \eqref{b.10} and \eqref{b.11} together imply that
\eqref{b.9} holds for $p = 6$ and any $0 < \e_6<1/6$. 

Note that, by Plancherel's theorem and Lemma~\ref{sob},
if $p=2$ we have \eqref{b.6} for $\e_2=0$.  Also, the kernels of the operators $f\to F^j_\pm(t, \, \cdot\, )$ are easily seen to be in $L^1(\R^n)$ uniformly in $t>0$ and $j\in \N$. This yields the analogue of \eqref{b.6} with $p=\infty$ and $\e_\infty=0$. Interpolating between
these two easier cases and the non-trivial bounds for
$p=6$, one obtains \eqref{b.6} for any $2<p<\infty$, which
completes the proof of Bourgain's circular maximal function
theorem.

\end{proof}

\begin{remark}\label{remark: not 1/p needed}
Note that any $\varepsilon_6>0$ suffices to obtain $\varepsilon_p>0$ for $2 < p< \infty$ in \eqref{b.9} after interpolating with $\varepsilon_2=\varepsilon_\infty=0$. Thus, as is remarked at the beginning of this subsection, the full strength of $1/6-L^6$ local smoothing for $e^{i t \sqrt{-\Delta}}$ is not needed here (nor is the particular choice of exponent $p = 6$): \textit{any} non-trivial local smoothing suffices. This is in contrast with $\S\S$\ref{subsec:BR}-\ref{subsec:MBR}. Similar considerations will apply for the variable coefficient variants in the next subsection.
\end{remark}




\subsection{Variable coefficient circular maximal function estimates}

Using local smoothing estimates for general Fourier
integral operators (as opposed to simply the euclidean half-wave propagators $e^{i t \sqrt{-\Delta}}$), one may modify the argument in $\S$\ref{subsec:circular} to
obtain a generalization of Bourgain's circular maximal function
theorem for geodesic circles on Riemannian surfaces. This
was originally shown by the third author in \cite{Sogge1991}.

Before describing the results, it is perhaps useful to review the relevant concepts from Riemannian geometry. If $(M,g)$ is a Riemannian manifold, then for any point $x \in M$ and tangent vector $v \in T_xM$ there exists a unique geodesic $\gamma_v$ such that $\gamma_v(0) = x$ and $\gamma_v'(0) = v$. Moreover, there exists some open neighbourhood $U \subseteq T_x M$ of the origin such that the \emph{exponential map} $\exp_x \colon U \to M$ taking $v \in U$ to $\exp_x(v) := \gamma_v(1)$ is well-defined. The \emph{injectivity radius $\mathrm{Inj}_x\,M>0$ of $M$ at $x$} is the supremum over all $r > 0$ for which $\exp_x$ may be defined on $B(0, r) \subset T_xM$. The \emph{injectivity radius $\mathrm{Inj}\,M\geq 0$ of $M$} is then defined to be the infimum of $\mathrm{Inj}_x\,M$ over all $x \in M$. If $M$ is compact, then $\mathrm{Inj}\,M > 0$ and given any $x \in M$ and $0 < t < \mathrm{Inj}\,M$ one may define the \emph{geodesic circle}
\begin{equation*}
    S_{x,t} := \big\{\exp_x(v) : v \in T_{x}M \textrm{ such that } |v| = t \big\}.
\end{equation*}
Note that in the case $M=\mathbb{S}^2$, a geodesic circle amounts to a great circle. See, for instance, \cite[Chapter III]{Chavel2006} for more details.

Now suppose $(M,g)$ is a two-dimensional compact
Riemannian manifold. Define the average over the geodesic circle $S_{x,t}$ about $x \in M$ of radius $0<t<\text{Inj }M$ by
\begin{equation*}
    A_tf(x) := \int_{S_{x,t}} f(y) \, \ud\sigma_{x,t}(y),
\end{equation*}
where $\sigma_{x,t}$ denotes the normalised (to have unit mass) arc length measure on $S_{x,t}$.  Fixing $0<r_0<\text{Inj }M$, a natural
variable coefficient version of Theorem \ref{cicular maximal theorem} is as follows:

\begin{theorem}[\cite{Sogge1991}] With the above definitions, for all $p > 2$,
\begin{equation}\label{b.13}
\Bigl(\, \int_M \sup_{0<t<r_0}
|A_tf(x)|^p \, \ud x\, \Bigr)^{1/p} \lesssim_p
\|f\|_{L^p(M)}.
\end{equation}
 Here $\ud x$ is the volume element on $(M,g)$ and the
$L^p$-norm on the right is associated with this measure.
\end{theorem}

\begin{proof} To prove these general maximal inequalities, it suffices to establish the analogue of 
\eqref{b.13} where the norm is taken over $\Omega\subset M$,
a relative compact subset of a coordinate patch and
$f$ is assumed to be supported in $\Omega$; of course the estimate should be established uniformly over all such $\Omega$. Working in local coordinates, and if 
$\beta$ is the Littlewood--Paley bump function used before, for $0<t<r_0$ and
$x\in \Omega$ and $\mathrm{supp}\,f\subset \Omega$ one may write
$$A_tf(x)=A^0_tf(x)+\sum_{j=1}^\infty A^j_tf(x),$$
where $A^0_t f$ is dominated by the Hardy--Littlewood
maximal function of $f$ and
$$A^j_tf(x) := \int K_j(x,t;y) \, f(y)\, \ud y$$
for all $j \in \N$, where
\begin{equation*}
K_j(x,t;y) :=
\int_{\hat{\R}^2} \hat{\sigma}_{x,t}(\xi) \, \beta(2^{-j}t|\xi|) \, 
e^{i \langle x-y,  \xi \rangle} \, \ud\xi.
\end{equation*}
Here, and in what follows, the Fourier transforms are taken with respect to the local coordinates in which we are working.

The maximal operator associated with
$A^0_t$ is trivial to handle. Thus,
\eqref{b.13} would follow if one can show that for all $p>2$ and all $j \in \N$ the inequality
\begin{equation}\label{b.15}
\Bigl(\, \int_M \sup_{0<t<r_0} |A^j_tf(x)|^p \, \ud x
\, \Bigr)^{1/p} \lesssim 2^{-j\e_p}\|f\|_{L^p(M)}
\end{equation}
holds for some $\e_p>0$. If, as before, $\Hat f_\ell(\xi)=\beta(2^{-\ell}|\xi|)\Hat f(\xi)$,
then
\begin{equation*}
   A^j_tf=\sum_{|\ell-(k+j)|\le 10}A^j_tf \qquad  \textrm{for $t\in [2^{-k},2^{-k+1}]\cap (0,\, \text{Inj M})$.}
\end{equation*}
Based on this, one may adapt the earlier arguments of Lemma \ref{maxdyad} to see that \eqref{b.15} would follow from favourable
bounds for the maximal operators associated with
dyadic intervals: in particular, it suffices to show that for all $p>2$ there exists some $\varepsilon_p > 0$ such that
\begin{equation}\label{b.16}
\Bigl(\int \sup_{t\in [2^{-k},2^{-k+1}]
\cap (0,r_0]}|A^j_tf(x)|^p \, \ud x \,
\Bigr)^{1/p}\lesssim 2^{-j\e_p}\|f\|_p.
\end{equation}

If $2^{-k}$ is bounded away from zero, then the operators
$f\to A_tf(x)$ for $t\in [2^{-k},2^{-k}]\cap (0,\text{Inj }M)$
are a family of Fourier integral operators of order $-1/2$
satisfying the cinematic curvature condition (see \cite{Sogge2017}). Thus, \eqref{b.16} easily
follows from the $1/p-$ local
smoothing estimates for FIOs when $p\ge 6$ (that is, Theorem \ref{global local smoothing theorem}) and the above arguments.
One may also handle the case where $2^{-k}\ll r_0$
by using a dilation argument and the
local smoothing estimates in Theorem \ref{global local smoothing theorem}. This is due to the fact that
for $x,y\in \Omega$ the Riemannian distance function in 
our local coordinates satisfies $$d_g(x;y)=\sqrt{\sum_{1 \leq j, k \leq 2} g_{jk}(x)(x_j-y_j)(x_k-y_k)}+O(|x-y|^2),$$
where $g_{jk}(x)\ud x^j\ud x^k$ is the Riemannian metric written
in our local coordinates.  See \cite{Mockenhaupt1993} or
\cite{Sogge2017} for more details.
\end{proof}

\subsection{Maximal bounds for half-wave propagators}

Consider half-wave propagators $e^{it\sqrt{-\Delta_g}}$ either
on euclidean space or on a compact Riemannian
manifold $(M,g)$ of dimension $n\ge2$.  By the previous
arguments one has
\begin{equation}\label{h.1}
\Bigl(\, \int_M \sup_{0<t<1}\bigl|e^{it\sqrt{-\Delta_g}}f(x)
\bigr|^p \, \ud x\, \Bigr)^{1/p}\lesssim_\e \|f\|_{L^p_{\bar s_p+\e}(M)}
\end{equation}
if there is $1/p-$ local smoothing for $e^{i t \sqrt{-\Delta_g}}$; see \S\ref{fixed time estimates section} for the relevant definitions. Thus, Theorem \ref{global local smoothing theorem}
yields the following:

\begin{theorem}  Under the above assumptions
\eqref{h.1} holds for all $p\ge \frac{2(n+1)}{n-1}$.  Consequently,
for this range of exponents,
$$e^{it\sqrt{-\Delta_g}}f(x)\to f(x) \, \, a.e. \, \,
\, \text{if } \, f\in L^p_s \quad
\text{with } \, \, s>\bar s_p.$$
\end{theorem}

It is noted that for a given $p$, \eqref{h.1} is sharp.  For instance, in the euclidean case if $s<\bar s_p$ and
$t\ne0$ there are $f\in L^p_s$ for which
$e^{it\sqrt{-\Delta}}f\notin L^p(\Rn)$ by a counterexample
of Littman~\cite{Littman}, and, in the manifold case, the
same is true when $t$ avoids a discrete set of times (see
\cite{Seeger1991}).




\section{Local smoothing and oscillatory integral estimates}\label{section: local smoothing and oscillatory integral estimates}

The aim of this section is to explore connections between local
smoothing for Fourier integral operators and $L^p$ bounds for oscillatory
integrals. As a consequence of this investigation, we will establish the necessary conditions for Conjecture \ref{LS conj FIO}.




\subsection{$L^p$ estimates for oscillatory integrals satisfying the Carleson--Sj\"olin condition}

Consider oscillatory integral operators of the form 
\begin{equation}\label{o.0}
T_\la f(x) := \int_{{\mathbb R}^{n-1}} e^{i\la \varphi(x;y')}
a(x;y') f(y')\, \ud y',
\end{equation}
sending functions of $(n-1)$ variables to functions of $n$ variables. Here $\varphi\in C^\infty(\R^n \times \R^{n-1})$ is assumed to be real-valued and $a \in C^\infty_0(\Rn\times {\mathbb R}^{n-1})$. It is also assumed that the phase functions satisfy the \textit{Carleson--Sj\"olin condition}, which has two parts:

\begin{mixed hessian condition}
\begin{equation}\label{o.1}
\rank \partial^2_{xy'}\varphi(x;y') \equiv n-1 \quad \text{for all } \, \,
(x;y')\in \text{supp }a.
\end{equation}
\end{mixed hessian condition}
Provided the support of $a$ is sufficiently small, this non-degeneracy condition ensures that for every $x_0$ in the $x$-support of $a$ the gradient graph
\begin{equation}\label{o.2}
\Sigma_{x_0}  := \{\nabla_x\varphi(x_0;y') : a(x_0; y') \neq 0 \} \subset T^*_{x_0}\Rn
\end{equation}
is a smooth hypersurface.

The other part of the Carleson--Sj\"olin
condition is the following curvature assumption.

\begin{curvature condition}
For each $x_0$ in the $x$-support of $a$, the hypersurface $\Sigma_{x_0}$ has non-vanishing Gaussian curvature at every point.
\end{curvature condition}

Under these assumptions, a problem of H\"ormander~\cite{Hormander1971} is to determine
for which $p\ge \frac{2n}{n-1}$ the estimate
\begin{equation}\label{o.4}
\|T_\la\|_{L^p({\mathbb R}^{n-1})\to L^p(\Rn)}=O_\e(\la^{-n/p+\e})
\end{equation}
holds for all $\varepsilon > 0$ (simple examples show that the constraint $p\ge \frac{2n}{n-1}$ is necessary). There are somewhat stronger formulations for $p>\frac{2n}{n-1}$
where $\e=0$ and $L^p({\mathbb R}^{n-1})$ is replaced
by $L^r({\mathbb R}^{n-1})$ for exponents $r<p$ satisfying
$\tfrac{n+1}{n-1}r'=p$; however, we shall focus on the 
formulation in \eqref{o.4} and its relation with local smoothing estimates.

\begin{theorem}[\cite{BHS}]\label{thmo}
Suppose that for a given $ \frac{2n}{n-1} \leq p < \infty$ there
is local smoothing of order $1/p-$ for all Fourier
integral operators satisfying the cinematic curvature
condition.  Then \eqref{o.4} holds for the same exponent $p$ for all phase functions
$\varphi$ satisfying the Carleson--Sj\"olin condition.
\end{theorem}

As Theorem \ref{global local smoothing theorem} ensures that there is local smoothing of order
$1/p-$ for all $ \frac{2(n+1)}{n-1} \leq p < \infty$ whenever the cinematic curvature condition holds, it follows that \eqref{o.4}
is valid for this range of exponents. This recovers a
slightly weaker version of Stein's \cite{Stein1986}
oscillatory integral theorem which says that the stronger $L^p-L^r$ estimates hold for $p \geq \frac{2(n+1)}{n-1}$ with $\e = 0$.

\begin{proof}[Proof (of Theorem~\ref{thmo})]  One may assume,
of course, that $a$ is supported in a small neighbourhood
of the origin in ${\mathbb R}^{n-1}\times \Rn$.  Also,
since replacing $\varphi$ by $\varphi(x;y')+Bx+Cy'$ where
$B:\, \Rn\to \Rn$ and $C: {\mathbb R}^{n-1}\to {\mathbb R}^{n-1}$ are linear does not change the operator norm
of $T_\la$, one may also assume that
\begin{equation}\label{o.5}
\nabla_{x;y'}\varphi(0;0)=0 \qquad \text{and} \qquad \det \partial^2_{x'y'} \varphi (0,0) \neq 0.
\end{equation}
If we set
\begin{equation}\label{o.55}
\Phi(x;y) := \varphi(x;y')+x_n+y_n, \quad y=(y',y_n),
\end{equation}
and if the support of $a$ is small enough, then
the Monge--Amp\`ere determinant of $\Phi$ satisfies
\begin{equation}\label{o.6}
{\rm det}\left(\begin{array}{cc}
0 & \partial_y\Phi\\
\partial_x \Phi &
\partial^2_{xy} \Phi
\end{array}\right)\neq 0 \qquad \textrm{for $(x;y')\in \text{supp }a$.}
\end{equation}
If $\rho \in C_0^\infty(\R)$ satisfies
$\rho\ge0$ and $\rho(0)=1$, then this implies that
\begin{equation}\label{o.7}
K(x,t;y) := a(x,y')\rho(y_n)\delta_0\bigl(t-\Phi(x;y)\bigr)
\end{equation}
is the kernel of a non-trivial Fourier integral of order $-(n-1)/2$ for each fixed $t$ near 0. Moreover, for each
$t\in \text{supp }\rho$, the associated Fourier integral operator satisfies the projection condition since \eqref{o.6} is 
equivalent to the fact that it has a canonical
relation which is a canonical graph; see Example \ref{remark:RotCurv}. 
For later use, note also that $K$ vanishes if $|t|$ is large.

Based on this, the canonical relation
$${\mathcal C}\subset T^*{\mathbb R}^{n+1}\, \backslash \, 0\times
T^*\Rn\, \backslash \, 0$$
arising from the Fourier integral
operator with kernel as in \eqref{o.7}, 
regarded as an operator sending smooth functions of
$y$ to smooth functions of $(x,t)$, satisfies the projection condition in the
cinematic curvature hypothesis; see Example \ref{remark:CinCurv}.  The cone condition
must also be valid since the image of the
projection onto the fibers $T^*_{x_0,t_0}{\mathbb R}^{n+1} \, \backslash \, 0$
for $(x_0,t_0)$ in the $(x,t)$ support of the kernel are just the cones
\begin{align}\label{cone}
\Gamma_{x_0,t_0}&=
\bigl\{\tau(\nabla_x\Phi(x_0;y),-1): \, \tau\in \R \, \backslash \, 0, \, \,
\Phi(x_0;y)=t_0, \, \, y\in \text{supp }a \, \rho \bigr\}
\\
&=\bigl\{\tau(z,-1): \, \, z\in \Sigma_{x_0}\}
\subset T^*_{x_0,t_0}{\mathbb R}^{n+1} \, \backslash \, 0 \notag
\end{align}
and these have $(n-1)$ non-vanishing principal curvatures
in view of the curvature condition.

Thus, the Fourier integral operators
\begin{equation}\label{o.88}
f\in C^\infty_0(\Rn)\to {\mathcal F}_sf(x,t) :=
\bigl(\sqrt{I-\Delta_x}\bigr)^{(n-1)/2-s}
\Bigl( \, \int_{\R^n} K(x,t;y)f(y)\, \ud y\, \Bigr)
\end{equation}
are Fourier integral operators of order $-s$ for each
fixed $t$ and the resulting family of Fourier integral operators
satisfies the cinematic curvature hypothesis.
As by hypothesis it is assumed that there is $1/p-$ local smoothing for such FIOs, one has
\begin{equation*}
\|{\mathcal F}_sf\|_{L^p(\Rn\times \R)}\lesssim_s
\|f\|_{L^p(\Rn)} \quad
\text{if } \, \, s>\bar s_p-1/p.
\end{equation*}

To see how this leads to \eqref{o.4}, observe first that since $\rho$ is non-trivial and $\Phi$ differs from $\varphi$ by terms which are linear in $x$ and $y$ (namely,
$x_n+y_n$) one must have
that
\begin{equation}\label{o.10}
\|T_\la\|_{L^p({\mathbb R}^{n-1})\to L^p(\Rn)}
\approx \|S_\la\|_{L^p(\Rn)\to L^p(\Rn)}
\end{equation}
for
$$S_\la f(x):=\int_{\Rn} e^{i\la\Phi(x;y)}
a(x;y')\rho(y_n) \, f(y)\, \ud y.$$

Next, let $m \in C^\infty(\R)$ satisfy
$m(r)=1$ if $r<1$ and $m(r)=0$ if $r>2$.
Then, since the Monge--Amp\`ere condition \eqref{o.6} implies
that $\nabla_x\Phi\ne 0$ on the support of the oscillatory
integral, a simple integration-by-parts argument shows that
\begin{equation}\label{o.11}
\bigl\| \, m(\sqrt{-\Delta_x}/c_o\la)\circ S_\la
\, \bigr\|_{L^p(\Rn)\to L^p(\Rn)}=O_N(\la^{-N})
\end{equation}
for all $N\in {\mathbb N}$ if  $c_0>0$ is chosen to be sufficiently small.  Furthermore,
$$\bigl\| \, (I-m(\sqrt{-\Delta_x}/c_0\la))\circ
(\sqrt{I-\Delta_x})^{-\gamma}\, \bigr\|_{L^p(\Rn)\to L^p(\Rn)}=O(\la^{-\gamma}), \quad \text{if } \, \gamma\ge 0.$$
Therefore, by \eqref{o.10} and \eqref{o.11}, for
such $\gamma$ one has 
\begin{equation}\label{o.12}
\|T_\la\|_{L^p({\mathbb R}^{n-1})\to L^p(\Rn)}
\lesssim \la^{-\gamma}
\bigl\| \, (\sqrt{I-\Delta_x})^{\gamma}\circ S_\la
\, \bigr\|_{L^p(\Rn)\to L^p(\Rn)} +O(\la^{-N}).
\end{equation}

On the other hand, if $K$ is as in \eqref{o.7} then
$$\int_{\R} e^{i\la t}K(x,t; y) \, \ud t =
e^{i\la \Phi(x;y)}\, a(x;y')\rho(y_n).$$
Since the right-hand-side is the kernel of the oscillatory
integral $S_\la$, 
one can use H\"older's inequality in $t$ to see that
if ${\mathcal F}_s$ is as in \eqref{o.88}, then
$$\bigl\| (\sqrt{I-\Delta_x})^{(n-1)/2-s}\circ S_\la\bigl\|_{L^p(\Rn)\to L^p(\Rn)}\lesssim
\bigl\| {\mathcal F}_s\bigr\|_{L^p(\Rn)\to 
L^p(\Rn\times \R)}.$$
Therefore, by \eqref{o.12}, if $(n-1)/2-s\ge 0$, then
\begin{equation}\label{o.13}
\|T_\la\|_{L^p({\mathbb R}^{n-1})\to L^p(\Rn)}
\lesssim \la^{-(n-1)/2+s}\, \|{\mathcal F}_s
\|_{L^p(\Rn)\to L^p(\Rn\times \R)}.
\end{equation}
Taking
 $s=\bar s_p-1/p+\e$ with
$\e>0$ small, \eqref{o.13} yields
$$\|T_\la\|_{L^p({\mathbb R}^{n-1})\to L^p(\Rn)}
=O(\la^{-(n-1)/2+\bar s_p-1/p+\e})=O(\la^{-n/p+\e}),$$
as $\bar s_p=(n-1)(\tfrac12-\tfrac1p)$.  Since this is 
\eqref{o.4}, the proof is complete.
\end{proof}




\subsection{Necessary conditions in Conjecture \ref{LS conj FIO}: sharpness of Theorem \ref{global local smoothing theorem} in odd dimensions}\label{sharpness section}

Bourgain identified in \cite{Bourgain1995} counterexamples to the estimates \eqref{o.4} for $p < \bar p_n$, where the exponent $\bar p_n$ is as defined in \eqref{exponents general}. Using this and Theorem \ref{thmo}, it follows that there are Fourier integral operators satisfying the cinematic curvature
hypothesis for which there cannot be local smoothing
of order $1/p-$ for any $p < \bar p_n$, leading to the range of exponents featured in Conjecture \ref{LS conj FIO}. In particular, this shows that the local smoothing estimates in Theorem \ref{global local smoothing theorem} are sharp in odd dimensions. Furthermore, Theorem \ref{thmo} may be used to formulate the more refined Conjecture \ref{LS conj FIO refined} through appropriate refined examples; details of the last fact will be omitted here and the reader is referred, for instance, to \cite[$\S$2.1]{Guth} for the heuristics behind such examples.

Bourgain's counterexample to \eqref{o.4} if $p < \frac{2(n+1)}{n-1}$ and $n\ge 3$ is odd is recalled presently.
The construction makes use of the
the symmetric
matrices
\begin{equation*}
A(s)=
\begin{pmatrix}
1 &s
\\
s &s^2
\end{pmatrix},
\end{equation*}
which depends on the real parameter $s$.
Observe that 
the matrices consisting of the derivatives of each component satisfy
\begin{equation*}
\det A'(s)\equiv -1,
\end{equation*}
while, on the other hand,
\begin{equation*}
\text{Rank }A(s)\equiv 1.
\end{equation*}

Using these matrices, if $x'=(x_1,\dots, x_{n-1})$, define the phase function $\varphi(x;y')$ on ${\mathbb R}^n\times {\mathbb R}^{n-1}$ by
\begin{equation*}
\varphi(x;y')= \langle x', y'\rangle +\frac12 \sum_{j=0}^{(n-3)/2} \bigl\langle \, A(x_n)(y_{2j+1}, y_{2j+2}), \, (y_{2j+1}, y_{2j+2})\, \bigr\rangle
\end{equation*}
and let $\Phi$ be as in \eqref{o.55}. If $T_\la$ is as in \eqref{o.0}, then stationary phase arguments yield (see, for example, \cite{Sogge2017})
\begin{multline}\label{badq}
\lambda^{-\frac{n-1}4 -\frac{n-1}{2p}} \lesssim \|T_\la\|_{L^\infty({\mathbb R}^{n-1})\to L^p({\mathbb R}^n)}
\lesssim \|T_\la\|_{L^p({\mathbb R}^{n-1})\to L^p(\Rn)}
\\
\text{if } \, \, \lambda \gg 1 \, \, \text{and } \, \, p\ge 2, \quad \text{provided that  } \, \, a(0;0)\ne 0.
\end{multline}

Clearly $\varphi$ satisfies the Carleson--Sj\"olin conditions
\eqref{o.1} and \eqref{o.2}.  Indeed, the surfaces
$\Sigma_{x_0}$ in \eqref{o.2} are, up to linear transformations,
the hyperbolic paraboloids in $\Rn$ parametrised by the graph
\begin{equation*}
\bigl(y', \, \frac12\sum_{j=0}^{(n-1)/2} y_{2j+1}y_{2j+2}
\bigr).
\end{equation*}

The counterexample now turns into a FIO counterexample for local smoothing following the proof of Proposition \ref{thmo}. Since \eqref{o.5} is also valid, if
${\mathcal F}_s$ is defined to be the Fourier integral operators
in \eqref{o.88}, then, by \eqref{o.13} and \eqref{badq}, 
one must have
$$\la^{-\frac{n-1}4-\frac{n-1}{2p}}
\lesssim \la^{-\frac{n-1}2+s}
\, \|{\mathcal F}_s\|_{L^p(\Rn)\to L^p(\Rn\times \R)}.$$
Since 
$$\tfrac{n-1}4+\tfrac{n-1}{2p}<\tfrac{n-1}2-\bigl(\bar s_p-\tfrac1p\bigr)
\quad \text{if } \, p<\tfrac{2(n+1)}{n-1},$$
one concludes that it does not hold that for sufficiently small $\sigma_p>0$
$$\|{\mathcal F}_s\|_{L^p(\Rn)\to L^p(\Rn\times \R)}<\infty
\, \, \, \text{if } \, \, 
p<\tfrac{2(n+1)}{n-1} \, \, \text{and } \, \,
s<(\bar s_p-1/p)+\sigma_p, 
$$
which means that the $1/p-$ local smoothing bounds break
down for these Fourier integral operators for this range
of exponents.

The construction can be modified to produce certain
negative results when $n\ge4$ is even.  In
this case one takes
\begin{multline*}
\varphi(x,y')= \langle x', y'\rangle +\tfrac12 \sum_{j=0}^{(n-4)/2} \bigl\langle A(x_n)(y_{2j+1},y_{2j+2}), \, (y_{2j+1}, y_{2j+2}) \, \bigr\rangle
+\tfrac12 (1+x_n) y^2_{n-1}
\end{multline*}
and defines $\Phi$ as in \eqref{o.55}.  The lower bound
for the resulting oscillatory integrals $S_\la$ in
\eqref{badq} changes to be
$$\lambda^{-\frac{n}4-\frac{n-2}{2p}}\lesssim \|T_\lambda\|_{L^p({\mathbb R}^{n-1})\to L^p({\mathbb R}^n)},$$
which, in turn, leads to the lower bound
$$\la^{-\frac{n}4-\frac{n-2}{2p}}
\lesssim \la^{-\frac{n-1}2+s}\|{\mathcal F}_s\|_{L^p(\Rn)
\to L^p(\Rn\times \R)}$$
for the Fourier integrals as in \eqref{o.88}.  A simple
calculation, as in the case of odd dimensions, now shows
that in the case of {\em even} $n\ge4$ these Fourier
integral operators, which satisfy the cinematic curvature
hypothesis, cannot have $1/p-$ local smoothing when
$p<\frac{2(n+2)}{n}$.

\begin{remark}[Odd versus even dimensional case]
As is mentioned in $\S$\ref{formulating a conjecture section}, the difference between the counterexamples for even and
odd dimensions can be explained by the geometry of the
cones \eqref{cone} associated to the Fourier integral
operators.  When $n$ is odd the cones involve dilates
of hyperbolic paraboloids and the number of
positive principal curvatures exactly matches the
number of negative ones: both equal $\frac{n-1}{2}$.  This is not possible when $n$ is even and in this case
there are only $\frac{n-2}{2}$ pairs of opposite signs; consequently, the resulting counterexamples involve larger exponents than those for odd $n$.
\end{remark}

It should be noted that the positive results of Stein \cite{Stein1986} showing that \eqref{o.4} holds for $p \geq \frac{2(n+1)}{n-1}$ are sharp in odd dimensions in view of Bourgain's counterexample. More recently, Bourgain and Guth \cite{Bourgain2011} obtained the positive results for $p\geq \frac{2(n+2)}{n}$ in the even dimensional case $n \geq 4$. Results for $n=2$ were obtained much earlier by Carleson and Sj\"olin \cite{Carleson1972}.

\begin{remark}[Signature hypothesis]
In view of Bourgain's counterexample, one should expect the estimate \eqref{o.4} to hold for $p < \bar p_n$ for oscillatory integrals $T_\lambda$ satisfying additional hypothesis on the signature of the associated hypersurfaces $\Sigma_{x_0}$. When all principal curvatures of $\Sigma_{x_0}$ are assumed to be of the same sign at each point, recent results
of Guth, Ilioupoulou and the second author \cite{Guth} show the favourable bounds for $T_\lambda$ in \eqref{o.4} for $p \geq \bar{p}_{n, +}$, which are sharp in view of previous counterexamples of Minicozzi and the third author \cite{Minicozzi1997} (see also \cite{Bourgain2011, Wisewell2005}).

In view of the connections between local smoothing estimates
and oscillatory integral theorems explored in this section, the results in \cite{Guth} suggest that
if the cones arising in the cinematic curvature hypothesis
have $(n-1)$ principal curvatures of the same sign, one should have $1/p-$ local smoothing for
all $ \bar{p}_{n,+} \leq p < \infty$; this corresponds to Conjecture \ref{LS conj FIO refined} with $\kappa=n-1$ and would imply Conjecture \ref{LS conj manifold}. Observe that Conjecture \ref{LS conj FIO refined} for $\kappa=n-1$ formally implies the results in \cite{Guth} via Theorem \ref{thmo}.
\end{remark}




\subsection{Maximal oscillatory integral estimates}

The above arguments also lead to maximal estimates for a natural class of oscillatory integral operators, including ones arising in spectral theory. As in the case of Bochner--Riesz operators, minor modifications of the proof that local smoothing implies oscillatory integral bounds yield corresponding maximal versions.

Consider oscillatory integrals of the form
\begin{equation*}
S_\la f(x):=\int_{\Rn} e^{i\la\Phi(x;y)} \, a(x;y) \, 
f(y)\, \ud y
\end{equation*}
where $a\in C^\infty_0(\Rn\times \Rn)$ and the
real smooth phase function $\Phi$ satisfies the 
\textit{$n\times n$ Carleson--Sj\"olin condition}, which has two parts:
\begin{rank condition}
\begin{equation}\label{m.2}
\rank \partial^2_{xy}\Phi(x;y) \equiv n-1 \quad
 \textrm{for all $(x;y)\in \text{supp }a$.}
\end{equation}
\end{rank condition}
This implies that for every fixed $x_0$ in the $x$-support of $a$ the gradient graph
\begin{equation*}
\Sigma_{x_0} := \{\, \nabla_x \Phi(x_0;y): \, \, a(x_0;y)\ne 0\, \} \subset T^*_{x_0}\Rn
\end{equation*}
is a smooth immersed hypersurface. 

The other part of the $n\times n$ Carleson--Sj\"olin
condition is identical to the curvature assumption which appeared earlier in this section.

\begin{curvature condition}
For each $x_0$ in the $x$-support of $a$, the hypersurface $\Sigma_{x_0}$ has non-vanishing Gaussian curvature at every point.
\end{curvature condition}

To be able to use
local smoothing estimates we shall also assume
that the \textbf{Monge--Amp\`ere condition} \eqref{o.6} holds on the support of $a$.

\begin{example}The class of oscillatory integral operators satisfying
these conditions includes ones arising in harmonic
analysis on Riemannian manifolds. In particular, if $d_g$
is the Riemannian distance function, then away from the diagonal $\Phi(x;y) := d_g(x;y)$ satisfies \eqref{m.2} and the curvature condition and has non-vanishing Monge--Amp\`ere determinant. 
\end{example}

The bounds \eqref{o.4} imply the corresponding bounds
\begin{equation}\label{m.4}
\|S_\la\|_{L^p(\Rn)\to L^p(\Rn)}=O_\e(\la^{-n/p+\e})
\end{equation}
for the same exponents.  Specifically, if \eqref{o.4} is valid for a given $p$ and all oscillatory integral satisfying the Carleson--Sj\"olin condition and possibly
an additional assumption on the geometry of the
hypersurfaces $\Sigma_{x_0}$ associated with $\varphi$, 
then \eqref{m.4} must be valid for the same exponent
$p$ for operators satisfying the $n\times n$ Carleson--Sj\"olin condition along with the same additional
geometric condition on the hypersurfaces $\Sigma_{x_0} \subset T^*_{x_0}\Rn$ associated
with $\Phi$. 

As a consequence of the preceding observation, the bounds \eqref{o.4} of Stein \cite{Stein1986} and Bourgain--Guth \cite{Bourgain2011} for the oscillatory
integral $T_\lambda$ imply that \eqref{m.4} holds for $p\ge \bar p_n$
if the $n\times n$ Carleson--Sj\"olin condition
is satisfied. Note that the counterexample of Bourgain \cite{Bourgain1995} described in the previous subsection also applies to operators $S_\lambda$, so such bounds are optimal in the sense that there are $S_\la$ for which \eqref{m.4}
cannot hold if $p<\frac{2(n+1)}{n-1}$ and $n$ is odd or $p < \frac{2(n+2)}{n}$ and $n \geq 4$ is even.

The local smoothing estimates in Theorem \ref{global local smoothing theorem} can be used to prove a maximal version of Stein's result for $S_\lambda$ which, by the previous discussion, is sharp in odd dimensions. It should be noted that the argument presented below would also yield maximal estimates for $p < \frac{2(n+1)}{n-1}$ in the even dimensional case or under stronger curvature hypotheses if one had the corresponding local smoothing estimates.

\begin{theorem}\label{mthm}
Suppose that $S_\la$ satisfies the $n\times n$ Carleson--Sj\"olin condition for all $\lambda \geq 1$ and that the phase
function $\Phi$ satisfies the Monge--Amp\`ere condition
\eqref{o.6} on the support of $a$.  For
$p\ge \frac{2(n+1)}{n-1}$ the maximal estimate
\begin{equation}\label{m.5}
\Bigl(\, \int_{\Rn}
 \, \sup_{\mu\in [\la,2\la]} \, 
\bigl| S_\mu f(x)\, \bigr|^p\, \ud x\, 
\Bigr)^{1/p}
\le C_\e \,  \la^{-n/p+\e}\, \|f\|_{L^p(\Rn)}
\end{equation}
holds for all $\e>0$.
\end{theorem}

\begin{proof} It suffices to show that, given $\e>0$ and 
$\frac{2(n+1)}{n-1} \leq p < \infty$, there
is a constant $C_\e$ such that 
\begin{equation*}
\Bigl(\, \int_{\Rn}\, |S_{\mu(x)}f(x)|^p\, \ud x
\, \Bigr)^{1/p}\le C_\e \la^{-n/p+\e}\, 
\|f\|_{L^p(\Rn)}
\end{equation*}
holds whenever $\mu(x): \Rn \to [\la,2\la]$ is measurable. To prove this, note that if 
$$K(x,t;y) := a(x;y)\, \delta_0\bigl(t-\Phi(x;y)\bigr)$$
then 
\begin{equation*}
{\mathcal F}f(x,t) := \int_{\Rn} K(x,t;y)\, f(y)\, \ud y
\end{equation*}
forms a one-parameter family of Fourier integrals of
order $-\frac{n-1}{2}$ satisfying the 
cinematic curvature condition. As in $\S$\ref{sharpness section}, the projection condition follows from the assumption that the 
Monge--Amp\`ere determinant associated with $\Phi$ never
vanishes, whilst the cone condition, as before, follows
from the curvature condition. Thus, Theorem \ref{global local smoothing theorem} implies that
\begin{equation}\label{m.8}
\bigl\|{\mathcal F}f\|_{L^p(\Rn\times \R)}
\lesssim_\e \|f\|_{L^p_{-n/p+\e}(\Rn)},
\end{equation}
since
\begin{equation*}
    -\tfrac{n-1}{2}+\big(\bar s_p-\tfrac{1}{p}\big)=-\tfrac{n}{p}.
\end{equation*}

Next, let $h\in C^\infty(\R)$ satisfy $h(r)=0$ for
$r<1/2$ and $h(r)=1$ for $r\ge 1$.  Then, since the Monge--Amp\`ere condition \eqref{o.6} implies that
$\nabla_y\Phi\ne0$ on the support of $a$, a simple integration-by-parts argument shows that, if $c_0>0$ is
chosen to be sufficiently small,
\begin{multline}\label{m.9}
S_\mu f(x)=S_\mu(f_\la)(x)+O(\la^{-N}\|f\|_p)
\\
\text{for}\quad f_\la := h(\sqrt{-\Delta_x}/c_o\la)f \quad
\text{and } \, \, \mu\in [\la,2\la],
\end{multline} 
for each $N\in {\mathbb N}$.  This is because for
$\mu\approx \la$ the kernel of $S_\mu\circ
(I-h(\sqrt{-\Delta_x}/c_0\la))$ is
$O(\la^{-N}(1+|y|)^{-N})$ for any $N$.

Next, use the fact that
$$S_{\mu(x)}g(x)=
\int_{\R} \int_{\R^n} e^{i\mu(x)t}\, K(x,t;y)\, g(y)\, \ud y  \,\ud t
=\int_{\R} e^{i\mu(x)t} \, {\mathcal F}g(x,t)\, \ud t.$$
Since ${\mathcal F}g(x,t)$ is compactly supported in
$t$, one may use H\"older's inequality and
\eqref{m.9} to deduce that
$$|S_{\mu(x)}f(x)|\lesssim
\Bigl(\, \int_{\R} |{\mathcal F}f_\la(x,t)|^p \, \ud t
\, \Bigr)^{1/p} + \, O(\la^{-N}\|f\|_p).$$
Since $S_{\mu(x)}f(x)$ vanishes for large $|x|$, this along
with \eqref{m.8} yields
\begin{align*}
\Bigl(\, \int_{\Rn}
|S_{\mu(x)}f(x)|^p \, \ud x \, \Bigr)^{1/p}
&\lesssim \|{\mathcal F}f_\la\|_{L^p(\Rn\times \R)}
+\la^{-N}\|f\|_{L^p(\R^n)}
\\
&\lesssim \|f_\la\|_{L^p_{-n/p+\e}(\Rn)}+\la^{-N}\|f\|_{L^p(\Rn)}
\\
&\lesssim \la^{-n/p+\e}\|f\|_{L^p(\Rn)},
\end{align*}
provided $N>n/p$. Here, in the last inequality, 
we used the fact that $\Hat f_\la(\xi)=0$ when $|\xi|$ is smaller than
a fixed multiple of $\la$. This completes the proof
of Theorem~\ref{mthm}.
\end{proof}

\begin{remark}
Under the assumption that the principal curvatures of $\Sigma_{x_0}$ are of the same sign, the counterexamples  in \cite{Minicozzi1997} show that \eqref{m.5} need not hold for $p<\bar{p}_{n,+}$. Furthermore, these counterexamples involve the model case where $\Phi(x;y) := d_g(x;y)$ for (certain choices of) Riemannian metrics $g$. In this setting, the recent results of Guth, Ilioupoulou and the second author \cite{Guth} concerning $T_\lambda$ suggest that \eqref{m.5} may hold for $p\ge \bar{p}_{n,+}$. It is not clear whether the additional hypothesis concerning the Monge--Amp\`ere determinant of $\Phi$ is necessary, since
the results of \cite{Guth} obtain \eqref{m.4} without this assumption. In any case, the proof of Theorem~\ref{mthm} required this assumption in order to be able to invoke the local smoothing estimates.
\end{remark}




\section{Wolff's approach to local smoothing estimates}\label{sec:Wolff}

The remaining sections of this survey discuss the proof of Theorem \ref{thm:LS}. Here we describe Wolff's approach which reduces local smoothing estimates to so-called \emph{decoupling inequalities} (see Theorem \ref{FIO decoupling theorem} below). His method has its roots in several ideas extensively used in harmonic analysis which go back to the work of Fefferman on the ball multiplier \cite{Fefferman1971}.




\subsection{Preliminary observations}\label{subsec:preliminary obs} Of course, Theorem \ref{thm:LS} follows from establishing
\begin{equation}\label{FIO estimate}
    \| \mathcal{F} f \|_{L^p(\R^{n+1})} \lesssim \| f \|_{L^p(\R^n)}
\end{equation}
for $\bar p_n \leq p < \infty$ and $\mu < -\alpha(p):= -\bar s_p +1/p$, where $\mathcal{F}$ is the operator \eqref{FIO def n to n+1}. We work with the representation of $\mathcal{F}$ in terms of an integral kernel: explicitly,
\begin{equation*}
    \mathcal{F}f(x,t) = \int_{\R^n} K(x,t;y)f(y)\,\ud y
\end{equation*}
where
\begin{equation}\label{kernel}
K(x,t;y) := \int_{\hat{\R}^n} e^{i (\phi(x,t; \xi)-\langle y, \xi \rangle)} b(x,t;\xi)(1 + |\xi|^2)^{\mu/2}\,\ud \xi
\end{equation}
and $b \in S^0(\R^{n+1} \times \R^n)$ is compactly supported in $x$ and $t$.

By the principle of stationary phase, one expects $K$ to be singular for those $(x,t;y)$ satisfying
\begin{equation*}
\nabla_\xi [ \phi(x,t;\xi) - \langle y, \xi \rangle] =0
\end{equation*}
for some $\xi \in \supp(b)$. In the prototypical case of the half-wave propagator $e^{it\sqrt{-\Delta}}$, 
for fixed $(x,t)$ this observation identifies the singular set of $K(x,t;\,\cdot\,)$ as lying in
\begin{equation}\label{singular sphere}
\Big\{y \in \R^n :  y-x = t \frac{\xi}{|\xi|} \textrm{ for some $\xi \in \mathrm{supp}\,b$}\Big\}
\end{equation}
and therefore inside the sphere
\begin{equation*}
    \Sigma_{(x,t)} := \{ y \in \R^{n} : |x-y| = t\}.
\end{equation*}
For general $\mathcal{F}$, as the map $\xi \mapsto \nabla_{\xi}\phi(x,t;\xi)$ is homogeneous of degree 0, the associated singular set for each fixed $(x,t)$ is typically an $(n-1)$-dimensional manifold. The relative complexity of the geometry of the singular sets places the study of such operators $\mathcal{F}$ well outside the classical Calder\'on--Zygmund theory; this is in contrast, for instance, with pseudo-differential operators, where the singularity occurs at an isolated point. 

The fundamental approach to understanding the kernel $K$ is to perform multiple decompositions of the $\xi$-support of $b$ and thereby break $K$ into pieces with a much simpler underlying geometry. 




\subsection{Basic dyadic decomposition} 
The first step is to break up $\mathcal{F}$ into pieces which are Fourier supported on dyadic annuli. Fix $\beta \in C^{\infty}_c(\R)$ with $\mathrm{supp}\,\beta \in [1/2,2]$ and such that $\sum_{\lambda >0  : \,\mathrm{dyadic}} \beta(r/\lambda) = 1$ for $r \neq 0$. Let $\mathcal{F}^{\lambda} := \mathcal{F} \circ \beta(\sqrt{-\Delta}/\lambda)$, so that $\mathcal{F}^{\lambda}f$ has kernel $K^{\lambda}$ given by introducing a $\beta(|\xi|/\lambda)$ factor into the symbol in \eqref{kernel}, and decompose $\mathcal{F}f$ as
\begin{equation*}
\mathcal{F}f =: \mathcal{F}\,^{\lesssim 1}f + \sum_{\lambda \in \N :\, \mathrm{dyadic}} \mathcal{F}^{\lambda}f.
\end{equation*}

It is not difficult to verify that $\mathcal{F}\,^{\lesssim 1}$ is a pseudo-differential operator of order 0 and therefore bounded on $L^p$ for all $1 < p < \infty$ by standard theory (see, for instance, \cite[Chapter VI, \S5]{Stein1993}). Thus, the problem is further reduced to showing that for any arbitrarily small $\varepsilon>0$ the estimate
\begin{equation*}
\|\mathcal{F}^{\lambda}f\|_{L^p(\R^{n+1})} \lesssim \lambda^{\alpha(p) + \mu + \varepsilon}\|f\|_{L^p(\R^n)}
\end{equation*}
holds for all $\lambda \geq 1$; letting $\varepsilon := -\frac{\mu + \alpha(p)}{2} > 0$ the estimate \eqref{FIO estimate} would then follow from summing a geometric series.

The remaining pieces $\mathcal{F}^{\lambda}$ (for $\lambda$ large) are more complicated objects. The uncertainty principle tells us that the singularity present in $K$ should have been ``resolved to scale $\lambda^{-1}$'' in $K^{\lambda}$. For instance, in the case of the wave propagator $e^{it\sqrt{-\Delta}}$ the kernel $K^{\lambda}$ should no longer be singular along $\Sigma_{(x,t)}$ but should satisfy:
\begin{enumerate}[i)]
\item $K^{\lambda}(x,t;\,\cdot\,)$ is concentrated in a $\lambda^{-1}$-neighbourhood of $\Sigma_{(x,t)}$, given by
\begin{equation}\label{fattened sphere}
    \big\{y \in \R^n : \big||x - y| - t\big| \lesssim \lambda^{-1} \big\};
\end{equation}
\item $\|K^{\lambda}(x,t;\,\cdot\,)\|_{\infty} \lesssim \lambda^\mu \lambda^n$.
\end{enumerate}
\begin{figure}
\begin{center}

\resizebox {0.7\textwidth} {!} {
\begin{tikzpicture}

{
    \filldraw[yellow!50, cm={cos(35) ,-sin(35) ,sin(35) ,cos(35) ,(0 cm, 0 cm)}](0,0) -- (1,4) -- (-1,4) -- (0,0); 
}

\fill[blue!20,even odd rule] (0,0) circle (3.2cm) (0,0) circle (2.8cm);

\draw[blue, thick] (0,0) circle (3cm);

{
    \draw[yellow!50, thick, cm={cos(35) ,-sin(35) ,sin(35) ,cos(35) ,(0 cm, 0 cm)}](1,4) -- (-1,4); 
}

{
    \draw[black,line width=0.5mm, ->, cm={cos(35) ,-sin(35) ,sin(35) ,cos(35) ,(0 cm, 0 cm)}](0,0) -- (0,2) node [above, right, scale = 1.3] {$\xi_{\nu}$}; 
}

   \fill[white,even odd rule] (0,0) circle (4.2cm) (0,0) circle (3.8cm);

{
    \draw[red, thick, <->, cm={cos(35) ,-sin(35) ,sin(35) ,cos(35) ,(0 cm, 0 cm)}](2.5/4+0.2,2.5) -- (3.2/4+0.2,3.2); 
}

{
    \draw[red, cm={cos(35) ,-sin(35) ,sin(35) ,cos(35) ,(0 cm, 0 cm)}] (3.2/4 + 0.7,3.2 + 0.2)  node [scale = 1.3] {$\lambda^{-1}$}; 
}

{
\draw [red,thick, <->, cm={cos(35) ,-sin(35) ,sin(35) ,cos(35) ,(0 cm, 0 cm)}, domain=-15:15] plot ({3.4*sin(\x)}, {3.4*cos(\x)});
}

\draw[red, cm={cos(35) ,-sin(35) ,sin(35) ,cos(35) ,(0 cm, 0 cm)}] (-0.3, 3.6) node[above, right, scale = 1.3] {$ t\lambda^{-1/2}$}; 

 \node[below, scale = 1.3] at  (0,0) {$x$};

		\end{tikzpicture}}
\caption{For fixed $(x,t)$, the kernel $K^{\lambda}(x,t;\,\cdot\,)$ associated to the half-wave propagator $e^{i t \sqrt{-\Delta}}$ is concentrated on an annulus around the circle $x + t\mathbb{S}^{n-1}$ of thickness $\sim \lambda^{-1}$ (denoted here in \colorbox{blue!20}{blue}). The piece $K^{\lambda}_{\nu}(x,y;\,\cdot\,)$ is further localised to an angular sector with angle $\lambda^{-1/2}$ (denoted here in \colorbox{yellow!50}{yellow}).
}
\label{cap diagram}
\end{center}
\end{figure}
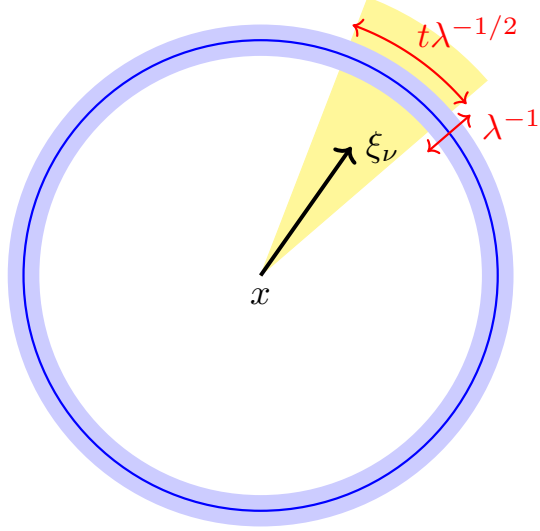
Here property i) is an uncertainty heuristic, whilst the second property trivially follows from the formula \eqref{kernel} for the kernel. These two features combine to give the crude estimate
\begin{enumerate}[i)]
\setcounter{enumi}{2}
\item $\int_{\R^n} |K^\lambda (x,t;y)|\, \ud y \lesssim \lambda^\mu \lambda^{n-1}$,
\end{enumerate}
which in turn yields an $L^{\infty} \to L^{\infty}$ bound for $\mathcal{F}^{\lambda}$. However, one may obtain a significant gain in the $\lambda$ exponent by subjecting $K^{\lambda}$ to a more refined stationary phase analysis. The method of stationary phase requires a uniform lower bound for $|\nabla_\xi \phi (x,t;\xi)|$ on $|\xi| \sim \lambda$; as $\xi \mapsto \nabla_\xi \phi (x,t;\xi)$ is homogeneous of degree 0, one should therefore decompose the angular variables into small regions in which $|\nabla_\xi \phi (x,t;\xi)|$ does not vary too much.




\subsection{Angular decomposition} For $\lambda$ fixed, let $\{\xi^\lambda_\nu\}_{\nu \in \Theta_{\lambda^{-1/2}}}$ be a maximal $\lambda^{-1/2}$-separated subset of $\mathbb{S}^{n-1}$, so that the indexing set satisfies $\#\Theta_{\lambda^{-1/2}} \sim \lambda^{(n-1)/2}$. Let 
\begin{equation*}
\Gamma_\nu^\lambda:=\{\xi\in \hat{\R}^n: \big|\pi_{\xi_{\nu}^{\lambda}}^{\perp} \xi\big| \lesssim \lambda^{-1/2}|\xi| \}
\end{equation*}
denote the sector of aperture $\sim\lambda^{-1/2}$ whose central direction is $\xi_\nu^{\lambda}$; here $\pi_{\xi^\lambda_\nu}^{\perp} $ is the orthogonal projection onto the hyperplane perpendicular to $\xi_{\nu}^{\lambda}$. Let $\{\chi^\lambda_\nu\}_{\nu \in \Theta_{\lambda^{-1/2}}}$ be a smooth partition of unity, homogeneous of degree 0, adapted to the $\Gamma^\lambda_\nu$, with $|D^\alpha \chi^\lambda_\nu (\xi) | \lesssim \lambda^{|\alpha|/2}$ for $\xi \in \mathbb{S}^{n-1}$ and all $\alpha \in \N_0^n$. Setting $b_\nu^\lambda(x,t;\xi):=b(x,t; \xi) \beta (|\xi|/\lambda) \chi_\nu^\lambda(\xi)$, the resulting operators $\mathcal{F}^{\lambda}_\nu$ have corresponding kernels
\begin{equation*}
K^{\lambda}_\nu (x,t;y):=\int_{\R^n} e^{i( \phi(x,t;\xi) - \langle y, \xi \rangle)} b_\nu^\lambda(x,t;\xi) (1+|\xi|^2)^{\mu/2} \, \ud \xi. 
\end{equation*}

To understand the effect of this frequency localisation on the spatial side, we once again consider the prototypical case of $e^{it\sqrt{-\Delta}}$. Recalling \eqref{singular sphere}, it follows from the choice of localisation that $K^{\lambda}_{\nu}$ should now be concentrated on the angular sector
\begin{equation*}
\big\{y \in \R^n : |\pi_{\xi_{\nu}^{\perp}}(x - y)|  \lesssim \lambda^{-1/2}|x-y|\big\}.
\end{equation*}
Combining this with property i) from the basic dyadic decomposition, it follows that:
\begin{enumerate}[i$'$)]
    \item $K^{\lambda}_{\nu}(x,t;\,\cdot\,)$ is concentrated in a $t\lambda^{-1/2}$ cap on the fattened sphere \eqref{fattened sphere}, centred at $t\xi_{\nu}^{\lambda}$ (see Figure \ref{cap diagram});
    \item $\|K^\lambda_\nu (x,t; \,\cdot\,)\|_{\infty} \lesssim \lambda^{\mu} \lambda^{(n+1)/2}$.
\end{enumerate}

It is not difficult to make these heuristics precise and, moreover, extend these observations to general variable-coefficient operators $\mathcal{F}$. In particular, the dyadic and annular decompositions allow one to linearise the phase $\phi(x,t;\xi)$ in the $\xi$-variable; this permits a standard stationary phase argument (see \cite{Seeger1991} or \cite[Chapter IX $\S\S$4.5-4.6]{Stein1993}) which reveals that the associated kernel $K^\lambda_\nu$ of $\mathcal{F}_\nu^\lambda$ satisfies the pointwise bound
\begin{equation}\label{kernel estimate}
|K_\nu^\lambda (x,t; y) | \lesssim \frac{\lambda^\mu \lambda^{(n+1)/2}}{(1+ \lambda|\pi_{\xi^\lambda_\nu} [ y - \nabla_\xi \phi(x,t,\xi_\nu^\lambda)]  | + \lambda^{1/2} |\pi_{\xi^\lambda_\nu}^{\perp} [y - \nabla_\xi \phi(x,t;\xi^\lambda_\nu)]  | )^{N}} 
\end{equation}
for all $N\geq 0$, where $\pi_{\xi^\lambda_\nu}$ denotes the projection onto the direction $\xi^\lambda_\nu$ and $\pi_{\xi^\lambda_\nu}^{\perp} $ its perpendicular projection. Note that \eqref{kernel estimate} immediately yields $\| K^\lambda_\nu (x,t; \,\cdot\,) \|_{1} \lesssim \lambda^{\mu}$, which together with the triangle inequality implies that
\begin{enumerate}[i$'$)]
\setcounter{enumi}{2}
    \item $\int_{\R^n} |K^\lambda(x,t;y)| \, \ud y \lesssim \lambda^\mu \lambda^{(n-1)/2};$
\end{enumerate}
note the square root gain over iii) obtained via the angular decomposition.




\subsection{Decoupling into localised pieces} Having found a natural decomposition of the operator 
\begin{equation*}
    \mathcal{F}^{\lambda} = \sum_{\nu \in \Theta_{\lambda^{-1/2}}} \mathcal{F}^{\lambda}_{\nu},
\end{equation*}
the problem is to effectively separate the contributions to $\| \mathcal{F}^{\lambda} f\|_{L^p(\R^{n+1})}$ coming from the individual the pieces. Since each $\mathcal{F}^{\lambda}_{\nu}f$ carries some oscillation, one may attempt to prove a square function estimate of the form
\begin{equation}\label{square function}
\| \mathcal{F}^{\lambda} f\|_{L^p(\R^{n+1})} \lesssim_{\varepsilon} \lambda^{\varepsilon} \big\|\big(\sum_{\nu \in \Theta_{\lambda^{-1/2}}} |\mathcal{F}^{\lambda}_{\nu}f|^2 \big)^{1/2}\big\|_{L^p(\R^{n+1})};
\end{equation}
here the appearance of the $\ell^2$ expression (rather than the $\ell^1$ norm which arises trivially from the triangle inequality) encapsulates the cancellation between the $\mathcal{F}^{\lambda}_{\nu}f$. Inequalities of the form \eqref{square function} were established in \cite{Mockenhaupt1992, Mockenhaupt1993}, albeit with a unfavourable dependence on $\lambda$, and these results have subsequently been refined by various authors \cite{BourgainSF, TV2, Lee2012, Lee}.

Unfortunately, establishing sharp versions \eqref{square function} appears to be an extremely difficult problem: indeed, the question is open even in the simplest possible case of the wave propagator $e^{it\sqrt{-\Delta}}$ with $n=2$. However, Wolff observed in \cite{Wolff2000} that sharp local smoothing inequalities can be obtained via a weaker variant of the estimate \eqref{square function} which is now known as a \emph{Wolff-type} or \emph{$\ell^p$-decoupling} inequality. Although still highly non-trivial, it transpires that the Wolff-type inequalities are nevertheless far easier to prove than their square function counterparts. In order to prove Theorem \ref{thm:LS} we will pursue Wolff's approach, and the key ingredient is the following estimate.

\begin{theorem}[Variable-coefficient Wolff-type inequality \cite{BHS}]\label{FIO decoupling theorem}
Let $2 \leq p \leq \infty$. For all $\varepsilon > 0$ the inequality
\begin{equation}\label{FIO decoupling estimate}
\|\mathcal{F}^{\lambda}f\|_{L^p(\R^{n+1})} \lesssim_{p, \varepsilon} \lambda^{\alpha(p) + \varepsilon} \Big(\sum_{\nu \in \Theta_{\lambda^{-1/2}}} \|\mathcal{F}^{\lambda}_{\nu}f\|_{L^p(\R^{n+1})}^p \Big)^{1/p}
\end{equation}
holds\footnote{Some slight technicalities have been suppressed in the statement of this theorem. In particular, the precise formulation includes some innocuous error terms: see \cite{BHS} for further details.}
, where
\begin{equation*}
\alpha(p) := \left\{ \begin{array}{ll}
\frac{\bar s_p}{2} & \textrm{if $2 \leq p \leq \frac{2(n+1)}{n-1}$}, \\[3pt]
\bar s_p - \frac{1}{p} & \textrm{if $\frac{2(n+1)}{n-1} \leq p < \infty$.}
\end{array}\right. 
\end{equation*}
\end{theorem}

\begin{remark}\label{trivial remark} \noindent
\begin{enumerate}[1)]
\item The value of $\alpha(p)$ coincides with that in $\S$\ref{subsec:preliminary obs}, which was only defined in the $\frac{2(n+1)}{n-1} \leq p < \infty$.
\item A necessary condition on $p$ for the square function estimate \eqref{square function} to hold is that $2 \leq p \leq \frac{2n}{n-1}$. For this range \eqref{square function} is stronger than \eqref{FIO decoupling estimate}, as can be seen by a simple application of Minkowski's and H\"older's inequalities. 
    \item It is instructive to compare \eqref{FIO decoupling estimate} with estimates obtained via trivial means. In particular, the triangle and H\"older's inequality imply that \eqref{FIO decoupling estimate} holds with the exponent $\alpha(p)$ replaced with $\frac{n-1}{2}\big(1 - \frac{1}{p}\big) = \bar s_p + \frac{n-1}{2p}$. Thus, the gain in the $\lambda$-power present in \eqref{FIO decoupling estimate} provides a measurement of the cancellation between the $\mathcal{F}_{\nu}^{\lambda}f$ arising from their oscillatory nature. 
\end{enumerate}
\end{remark}

Theorem \ref{FIO decoupling theorem} can be combined with simple estimates for the localised pieces (see  \eqref{recoupling inequality} below) to deduce the desired estimate
\begin{equation}\label{good times}
    \|\mathcal{F}^{\lambda}f\|_{L^p(\R^{n+1})} \lesssim _{s,p, \varepsilon} \lambda^{\alpha(p) + \mu +\varepsilon} \|f\|_{L^p(\R^n)};
\end{equation}
the details of this argument are discussed in the following subsection.


Theorem \ref{FIO decoupling theorem} is an extension of the result for the constant-coefficient operators
\begin{equation}\label{wave propagators}
    e^{i t h(D)} f(x):=\frac{1}{(2\pi)^n}\int_{\hat{\R}^n} e^{i( \langle x, \xi \rangle + t h(\xi))} \hat{f}(\xi) \, \ud \xi,
\end{equation}
which is a celebrated theorem of Bourgain--Demeter \cite{Bourgain2015, Bourgain2017a}; in line with our previous hypotheses on the phase, $h$ is assumed to be homogeneous of degree 1, smooth away from $
\xi = 0$ and such that the cone parametrised by $\xi \mapsto (\xi, h(\xi))$ has everywhere $(n-1)$ non-vanishing principal curvatures.

\begin{theorem}[Bourgain--Demeter \cite{Bourgain2015, Bourgain2017a}]\label{thm:BourgainDemeter} For all $2 \leq p \leq \infty$ and all $\varepsilon > 0$ the estimate
\begin{equation*}
\|e^{ith(D)}f\|_{L^p(\R^{n+1})} \lesssim_{\varepsilon, h} \lambda^{\alpha(p) +\varepsilon} \Big(\sum_{\nu \in \Theta_{\lambda^{-1/2}}} \|e^{ith(D)}f_{\nu}\|_{L^{p}(\R^{n+1})}^p \Big)^{1/p}
\end{equation*}
holds for all $\lambda \geq 1$ and functions $f$ such that $\supp(\widehat{f}) \subseteq \{ \xi \in \R^n:  \lambda \leq |\xi| \leq 2\lambda\}$.
\end{theorem}

The proof of Theorem \ref{thm:BourgainDemeter} is difficult and deep and relies on tools from multilinear harmonic analysis (in particular, the Bennett--Carbery--Tao multilinear Kakeya theorem \cite{Bennett2006} and the Bourgain--Guth method \cite{Bourgain2011}). These important ideas will not be addressed in this survey, and the interested reader is referred to the original papers \cite{Bourgain2015, Bourgain2017a} or the study guide \cite{Bourgain2017} for further information.

It transpires that the variable coefficient Theorem \ref{FIO decoupling theorem} can be deduced as a consequence of the constant coefficient Theorem \ref{thm:BourgainDemeter} via a relatively simple induction-on-scales and approximation argument. A sketch of the proof of Theorem \ref{FIO decoupling theorem} (avoiding many of the technical details) will be given in the next section.




\subsection{Bounding the localised pieces}\label{bounding localised pieces section} Given the variable-coefficient Wolff-type inequality from Theorem \ref{FIO decoupling theorem}, to conclude the proof of Theorem \ref{thm:LS} it suffices to show the localised pieces satisfy
\begin{equation}\label{recoupling inequality}
\Big(\sum_{\nu \in \Theta_{\lambda^{-1/2}}} \|\mathcal{F}^{\lambda}_{\nu}f\|_{L^p(\R^{n+1})}^p \Big)^{1/p} \lesssim \lambda^{\mu} \|f\|_{L^p(\R^n)}
\end{equation}
for $2 \leq p \leq \infty$. Indeed, combining this inequality with Theorem \ref{FIO decoupling theorem} one immediately obtains \eqref{good times}, as required. 

The inequality \eqref{recoupling inequality} is a simple consequence of the basic properties of the localised operators and, in particular, the kernel estimate \eqref{kernel estimate}. By real interpolation, it suffices to prove the bounds only at the endpoints $p = \infty$ and $p = 2$.

\subsubsection*{$L^{\infty}$-bounds} Observe that \eqref{kernel estimate} immediately implies
\begin{equation*}
\max_{\nu \in \Theta_{\lambda^{-1/2}}} \sup_{(x,t)\in \R^{n+1}}\|K^{\lambda}_{\nu}(x,t;\,\cdot\,)\|_{L^1(\R^n)} \lesssim \lambda^{\mu}.
\end{equation*}
From this, one deduces that 
\begin{equation*}
    \max_{\nu \in \Theta_{\lambda^{-1/2}}} \|\mathcal{F}^{\lambda}_{\nu}f\|_{L^{\infty}(\R^n)} \lesssim \lambda^{\mu} 
    \| f \|_{L^\infty(\R^n)}.
\end{equation*}
\subsubsection*{$L^2$-bounds} Useful estimates are also available at the $L^2$-level. For instance,  the wave propagator $e^{it\sqrt{-\Delta}}$
satisfies the conservation of energy identity
\begin{equation}\label{conservation of energy}
    \|e^{it\sqrt{-\Delta}}f\|_{L^2(\R^n)} = (2\pi)^{-n/2} \|f\|_{L^2(\R^n)} \qquad \textrm{for each fixed time $t \in \R$,}
\end{equation}
which, indeed, is a trivial consequence of Plancherel's theorem. This observation can be extended to the general variable coefficient setting at the expense of relaxing the equality to an inequality. In particular, a theorem of H\"ormander \cite{Hormander1973} (see also \cite[Chapter IX $\S$1.1]{Stein1993}) implies the bound 
\begin{equation}\label{Hormander L2}
 \| \mathcal{F}^{\lambda}_\nu f(\,\cdot\,,t) \|_{L^2(\R^{n})} \lesssim \| f_\nu^\lambda \|_{L^2(\R^n)}   \qquad \textrm{for each fixed time $t \in \R$,} 
\end{equation}
where $\widehat{f}^{\:\lambda}_\nu := \hat{f} \chi_\nu^\lambda$ is a piece of $f$ given by localising the frequencies to $\Gamma_{\nu}^{\lambda}$. Of course, in the general variable coefficient case Plancherel's theorem cannot be directly applied as in the proof of \eqref{conservation of energy}; nevertheless, \eqref{Hormander L2} can be established via a simple $T^*T$ argument and standard oscillatory integral techniques. 

One may now obtain space-time estimates for the $\mathcal{F}_{\nu}^{\lambda}f$ simply by integrating both sides of \eqref{Hormander L2} over a (compact) time interval containing the $t$-support of $b$. The almost orthogonality of the $f_\nu^\lambda$, given by Plancherel's theorem and the almost disjointness of $\Gamma_\nu^\lambda$, then readily implies that
\begin{equation*}
\Big(\sum_{\nu \in \Theta_{\lambda^{-1/2}}} \|\mathcal{F}^{\lambda}_{\nu}f\|_{L^2(\R^{n+1})}^2 \Big)^{1/2} \lesssim \lambda^{\mu} \|f\|_{L^2(\R^n)}.
\end{equation*}
This concludes the proof of Theorem \ref{thm:LS}. 




\section{Variable-coefficient Wolff-type inequalities}\label{sec:decoupling}

In the previous section the proof of the local smoothing estimate in Theorem \ref{thm:LS} was reduced to establishing the variable-coefficient Wolff-type inequalities in Theorem \ref{FIO decoupling theorem}. In this section we sketch the proof of Theorem \ref{FIO decoupling theorem}, which is in fact a consequence of the constant-coefficient case (that is, Theorem \ref{thm:BourgainDemeter}). 

\subsection{Preliminaries} It suffices to consider the case $\mu = 0$ (the general case then follows by writing any given operator as a composition of a pseudo-differential operator and an operator of order 0). By the homogeneity of $\phi(x,t;\cdot)$ and rescaling, Theorem \ref{FIO decoupling theorem} follows from its analogous statement when $|\xi| \sim 1$ and $(x,t) \in B(0,\lambda)$. Namely, it suffices to prove \eqref{FIO decoupling estimate} for the rescaled operators
$$
\mathcal{F}^\lambda f (x,t):=\int_{\hat{\R}^n} e^{i \phi^\lambda (x,t;\xi)} b^\lambda(x,t; \xi)  \widehat{f}(\xi) \, \ud \xi
$$
where
\begin{equation*}
\phi^{\lambda}(x,t;\omega) := \lambda\phi(x/\lambda, t/\lambda;\omega) \qquad \textrm{and} \qquad b^{\lambda}(x,t; \xi) := b(x/\lambda, t/\lambda,\xi)
\end{equation*}
and $b$ is supported in $B^{n+1} \times \Gamma$. Here $\Gamma$ is a conic domain of the type
$$
\Gamma:=\{\xi \in \hat{\R}^n : 1 \leq |\xi| \leq 2 \:\: \mathrm{and} \:\: |\xi/|\xi| - e| \lesssim 1\}
$$
for a unit vector $e \in \mathbb{S}^{n-1}$. Note that the notation $\mathcal{F}^\lambda$ is \emph{not} consistent with that used in the previous section.




\subsection{Inductive setup} The proof will involve an induction-on-scale procedure. To this end, an additional spatial scale parameter $R$ is introduced: it will be shown that for $1 \leq R \leq \lambda$ the inequality
\begin{equation}\label{decoupling inequality general R}
\|\mathcal{F}^{\lambda}f\|_{L^p(B_R)} \leq \bar{C}(\varepsilon, p) R^{\alpha(p) +\varepsilon} \big(\sum_{\nu \in \Theta_{R^{-1/2}}} \|\mathcal{F}^{\lambda}_\nu f\|_{L^{p}(B_R)}^p \Big)^{1/p}
\end{equation}
holds for a suitable choice of constant $\bar{C}(\varepsilon, p)$. Here $B_R \subseteq B(0,\lambda)$ is a ball of radius $R$ so that Theorem \ref{FIO decoupling theorem} follows by setting $R=\lambda$. 

By the trivial argument described in Remark \ref{trivial remark}, the inequality 
\begin{align*}
\|\mathcal{F}^{\lambda}f\|_{L^p(B_R)}  \lesssim R^{(n-1)/2p'} \big(\sum_{\nu \in \Theta_{R^{-1/2}}} \|\mathcal{F}^{\lambda}_\nu f \|_{L^{p}(B_R)}^p \Big)^{1/p}
\end{align*}
holds. This settles the desired decoupling inequality \eqref{decoupling inequality general R} for $R \sim 1$, and thereby establishes the base case for the induction.

Fix $1 \ll R \leq \lambda$ and assume the following induction hypothesis:

\begin{induction hypothesis} Assume \eqref{decoupling inequality general R} holds whenever $(R, \lambda)$ is replaced with $(R', \lambda')$ for any $1 \leq R' \leq R/2$ and $\lambda' \geq R'$.
\end{induction hypothesis}

In fact, one must work with a slightly more sophisticated induction hypothesis which involves not just a single operator $\mathcal{F}^{\lambda}$ but a whole class of related operators $\widetilde{\mathcal{F}}^{\lambda}$ which is closed under certain rescaling operations. The precise details are omitted here: see \cite{BHS} for further information. 





\subsection{Key ingredients of the proof} The proof of the inductive step comes in three stages:
\begin{enumerate}[1)]
\item At sufficiently small scales $1 \ll K \ll \lambda^{1/2}$, the operator $\mathcal{F}^{\lambda}$ may be effectively approximated by constant coefficient operators \eqref{wave propagators}.
\item For each of the approximating constant-coefficient operators, one may use the Bourgain--Demeter theorem at the small scale $K$.
\item The inherent symmetries of the inequality \eqref{decoupling inequality general R} allow one to propagate the gain arising from the constant-coefficient Bourgain--Demeter theorem at the small scale $K$ to larger scales. This is achieved via a parabolic rescaling argument, together with an application of the radial induction hypothesis.
\end{enumerate}
Further details of this simple programme are provided in the forthcoming subsections.




\subsection{Approximation by constant coefficient operators} Let $\mathcal{B}_K$ be a cover of $B_R$ by balls of radius $K$ for some value of $1 \ll K \ll \lambda^{1/2}$ to be determined later. Consider the spatially localised norm $\| \mathcal{F}^{\lambda} f \|_{L^p(B_K)}$ for $B_K = B(\bar z, K) \in \mathcal{B}_K$. By the uncertainty principle, localising to a spatial ball of radius $K$ should induce frequency uncertainty at the reciprocal scale $K^{-1}$. To understand what this means for our operator, we return once again to the prototypical case of the wave propagator. Observe that for any test function $\varphi \in C_c(\hat{\R}^{n+1})$ one has
\begin{equation}\label{distributional identity}
    \int_{\R^{n+1}} e^{it\sqrt{-\Delta}}f(x) \check{\varphi}(x,t)\,\ud x\ud t = \int_{\hat{\R}^n} \varphi(\xi, |\xi|) \hat{f}(\xi)\,\ud \xi 
\end{equation}
and therefore the space-time Fourier transform of $e^{it\sqrt{-\Delta}}f$ is distributionally supported on the light cone. For the general variable-coefficient case, the Fourier support properties of $\mathcal{F}^{\lambda}f$ involve a whole varying family of conic hypersurfaces $\Sigma_z \colon \xi \mapsto \partial_z\phi^{\lambda}(z;\xi)$, parametrised by $z \in \R^{n+1}$, and there is no clean distributional identity analogous to \eqref{distributional identity}. However, note that for $z \in B(\bar{z},K)$ one has
\begin{equation*}
    |\partial_z\phi^{\lambda}(z;\xi) - \partial_z \phi^{\lambda}(\bar{z};\xi)| \lesssim \frac{|z-\bar z|}{\lambda} \leq K^{-1}
\end{equation*}
provided $K \ll \lambda^{1/2}$, and so the uncertainty principle tells us that the surfaces $\Sigma_z$, and $\Sigma_{\bar{z}}$ should be essentially indistinguishable once the operator is spatially localised to $B_K$. It in fact follows that on $B_K$ the operator $\mathcal{F}^{\lambda}$ can be effectively approximated by a constant coefficient operator 
\begin{equation}\label{constant coefficient operator}
   T_{\bar{z}}g(z) := \int_{\hat{\R}^n}e^{i\langle \partial_z\phi^{\lambda}(\bar{z}; \xi),z\rangle}a(\xi)\hat{g}(\xi)\,\ud \xi
\end{equation}
associated to surface $\Sigma_{\bar{z}}$, where $a$ is a suitable choice of cut-off function. 

An alternative and slightly more accurate way to understand this approximation is to consider the first order Taylor expansion of the phase function 
\begin{equation*}
    \phi^{\lambda}(z; \xi) - \phi^{\lambda}(\bar{z};\xi)  = \langle \partial_z \phi^{\lambda}(\bar{z}; \xi), z - \bar{z} \rangle+ O(\lambda^{-1}|z - \bar{z}|^2).
\end{equation*}
Since $\lambda^{-1}|z - \bar{z}|^2 \ll 1$ for $z \in B_K$, the error term in the right-hand side does not contribute significantly to the oscillation induced by the phase. Consequently, over the ball $B_K$ one may safely remove this error and thereby replace the phase $\phi^{\lambda}$ by its linearisation $\phi^{\lambda}(\bar{z};\xi) + \langle \partial_z \phi^{\lambda}(\bar{z}; \xi), z - \bar{z} \rangle$. Observations of this kind lead to a statement of the form
\begin{equation*}
\| \mathcal{F}^{\lambda} f \|_{L^p(B_K)} \sim \| T_{\bar{z}}f _{\bar{z}}\|_{L^p(B(0,K))}
\end{equation*}
where $f_{\bar{z}}$ is defined by $\widehat{f}_{\bar{z}} := e^{i\phi^{\lambda}(\bar{z};\,\cdot\,)} \widehat{f}$ and $T_{\bar{z}}$ is as in \eqref{constant coefficient operator}. 

In practice, there are significant technical complications which arise in making these heuristics precise: the full details may be found in \cite{BHS}.

\begin{center}
\begin{figure}
\includegraphics[width=\textwidth]{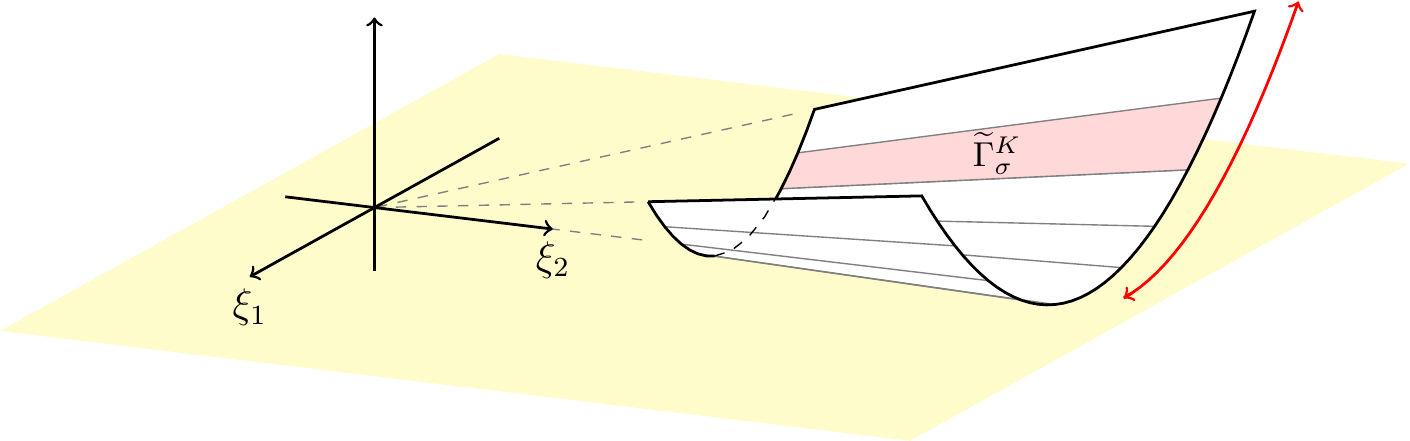}
\caption{The parabolic rescaling phenomenon for the phase $\phi(x,t;\xi)=x_1\xi_1+ x_2 \xi_2 + t \xi_1^2/\xi_2$. Here $\widetilde{\Gamma}_\sigma^K$ denotes the image of $\Gamma_\sigma^K$ under the map $\xi \mapsto(\xi, h_{\mathrm{par}}(\xi))$.}
\label{fig:parabolic rescaling}
\end{figure}
\end{center}




\subsection{Application of constant-coefficient decoupling}

The above approximation allows one to take advantage of the sharp $\ell^p$-decoupling theorem of Bourgain--Demeter for the constant coefficient operators $T_{\bar{z}}$ at scale $K$. In particular, on each $B_K = B(\bar z, K)$ one has
\begin{align*}
\| \mathcal{F}^{\lambda} f \|_{L^p(B_K)} & \sim \| T_{\bar{z}} f_{\bar z}\|_{L^p(B(0,K))} \\
&\lesssim_{\varepsilon} K^{\alpha(p) + \varepsilon/2} \big(\sum_{\sigma \in \Theta_{K^{-1/2}}}  \|T_{\bar{z}} f_{\bar z, \sigma} \|_{L^p(B(0,K))}^p \big)^{1/p} \\
& \sim K^{\alpha(p) + \varepsilon/2} \big(\sum_{\sigma \in \Theta_{K^{-1/2}}}  \| \mathcal{F}^{\lambda}_{\sigma} f \|_{L^p(B_K)}^p \big)^{1/p},
\end{align*}
where the first inequality is due to Theorem \eqref{thm:BourgainDemeter}. Summing over $B_K \subset B_R$, it follows that
\begin{equation}\label{approximation + BD}
\| \mathcal{F}^{\lambda} f \|_{L^p(B_R)}  \lesssim K^{\alpha(p) + \varepsilon/2} \big(\sum_{\sigma \in \Theta_{K^{-1/2}}} \| \mathcal{F}^{\lambda}_{\sigma} f \|_{L^p(B_R)}^p \big)^{1/p}.
\end{equation}
Thus, we have succeeded in decoupling $\mathcal{F}^{\lambda}f$ into scale $K^{-1/2}$ pieces, but we are still far from achieving the required decoupling at scale $R^{-1/2}$. 

At this point it is perhaps instructive to explain some of the ideas behind the proof, before fleshing out the details in the remaining subsections. The next step is to treat each of the summands on the right-hand side of \eqref{approximation + BD} individually. This is (essentially) done by repeating the above argument to successively pass down from decoupling at scale $K^{-1/2}$ to decoupling at scales $K^{-1}$, $K^{-3/2}$, $\dots$ until we reach the small scale $R^{-1/2}$. The key difficulty is to keep control of the constants in the inequalities which would otherwise build up over repeated application of the preceding arguments.\footnote{In the proof we will take $K \sim 1$: it is therefore necessary to iterate roughly $\log R$ times to pass all the way down to scale $R^{-1/2}$. If at each iteration we iterate we pick up an admissible constant $C$, then over all the iterations we pick up an inadmissible constant $C^{\log R} = R^{\log C}$.}

To control the constant build up, we assume a slightly different perspective. In particular, as in \cite{Bourgain2015}, we apply a parabolic rescaling in each stage of the iteration; this converts the improvement in the size of the decoupling regions to an improvement in the spatial localisation. In particular, \eqref{approximation + BD} can be thought of as passing from decoupling at scale 1 (the left-hand side) to decoupling at scale $K^{-1/2}$ (the right-hand side); after rescaling it roughly corresponds to passing from spatial localisation at scale $R$ to spatial localisation at scale $R/K$. The idea is then to iterate until we are spatially localised to $\sim 1$ scales, at which point the desired inequality becomes trivial. An advantage of the rescaling is that the repeatedly rescaled operators get closer and closer to constant coefficient operators over the course of the iterations. Thus, we find ourselves are in a more and more favourable situation as the argument progresses and this prevents a constant build up.

We shall see that the iteration argument sketched above can be succinctly expressed using our radial induction hypothesis.




\subsection{Parabolic rescaling}

By a parabolic rescaling argument, one can scale $\Gamma_\sigma^{K}$ to $\Gamma$, so that the support of $\widehat{f}_\sigma$ is at unit scale. This essentially reduces the spatial scale from $R$ to $R/K$ and anticipates an appeal to the radial induction hypothesis in \S\ref{induction hypothesis subsection}. 

\subsubsection*{A prototypical example} To illustrate the rescaling procedure, we consider the model operator $e^{ith_{\mathrm{par}}(D)}$ where
\begin{equation*}
    h_{\mathrm{par}}(\xi) := \frac{|\xi'|^2}{\xi_n} \qquad \textrm{for $\xi = (\xi',\xi_n) \in \hat{\R}^n$;}
\end{equation*}
this is a close cousin of the classical half-wave propagator $e^{it\sqrt{-\Delta}}$, but $e^{ith_{\mathrm{par}}(D)}$ enjoys some additional symmetries which make it slightly easier to analyse. 

Without loss of generality, one may interpret $\xi'$ as the angular variable; in particular, it is assumed that $\Gamma_{\sigma}^K$ is a sector of the form
\begin{equation*}
    \big\{(\xi',\xi_n) \in \hat{\R}^n : 1/2 \leq \xi_n \leq 2 \textrm{ and } |\xi'/\xi_n - \omega_{\sigma}| \leq K^{-1/2}\big\}
\end{equation*}
for some $\omega_{\sigma}\in B^{n-1}(0,1)$. The sector $\Gamma_{\sigma}^K$ is therefore mapped to $\Gamma$ under the transformation $(\Psi_{\sigma}^K)^{-1}$ where $\Psi_{\sigma}^K \colon (\xi',\xi_n) \mapsto (K^{-1/2}\xi' + \omega_\sigma \xi_n, \xi_n)$: see Figure~\ref{fig:parabolic rescaling}. 

 Let $\phi_{\mathrm{par}}$ be the phase associated to the operator $e^{ith_{\mathrm{par}}(D)}$. The scaling in the frequency variables can be transferred onto the spatial variables via the identity
 \begin{equation}\label{scaling identity}
     \phi_{\mathrm{par}}(x,t;\Psi_{\sigma}^K(\xi)) = \phi_{\mathrm{par}}(\Upsilon_{\sigma}^K(x,t); \xi) 
 \end{equation}
 where $\Upsilon_{\sigma}^K \colon (x,t) \mapsto (K^{-1/2}(x'+2t\omega_\sigma), \langle x', \omega_\sigma \rangle + x_n + t |\omega_\sigma|^2, K^{-1}t)$. Consequently, 
 \begin{equation*}
     \|e^{ith_{\mathrm{par}}(D)}f\|_{L^p(B_R)} = K^{(n+1)/2p} \|e^{ith_{\mathrm{par}}(D)}\tilde{f}_{\sigma}\|_{L^p(\Upsilon_{\sigma}^K(B_R))}
 \end{equation*}
 where $\tilde{f}_{\sigma}$ is defined by 
 \begin{equation}\label{rescaled function}
     [\tilde{f}_{\sigma}]\;\widehat{}\; := K^{-(n-1)/2}\widehat{f}_{\sigma}\circ \Psi_{\sigma}^K. 
 \end{equation}
  Observe that the set $\Upsilon_{\sigma}^K(B_R)$ is contained in an $R \times R/K^{1/2} \times \dots \times R/K^{1/2} \times R/K$ box: see Figure \ref{fig: spatial rescaling}. The longest side, which is of length $R$, points in the $(w_\sigma,1)$ direction whilst the shortest side, which of
length $R/K$, points in the time direction. The remaining sides, which are of length
$R/K^{1/2}$, point in spatial directions orthogonal to the long and short sides.

\begin{figure}
$(x,t)$ space

\begin{tikzpicture}[scale=1.6]

\begin{scope}[scale=1.3]

\filldraw[pattern=dots, pattern color=red!50] (0,0,1) -- (0,1,1)--(1,1,1)--(1,0,1)--(0,0,1);

\filldraw[pattern=dots, pattern color=red!50](0,1,1) -- (0,1,0) --(1,1,0) -- (1,1,1);

\filldraw[pattern=dots, pattern color=red!50] (1,1,0) -- (1,0,0) -- (1,0,1) -- (1,1,1);

\draw[<->] (0.08,0,1.2) -- (1.08,0,1.2);

\draw (0.58,0,1.2) node[below] {$R$};

\draw[<->] (-0.1,0,1) -- (-0.1,1,1);

\draw (-0.1,0.5,1) node[left] {$R$};

\draw[<->] (1.1,0,1+.1) -- (1.1,0,+0.1);

\draw (1.1,0,.65) node[right=0.1cm] {$R$};


\draw[->, thick, red] (1.5,0.5,0.5) to [bend left,looseness=0.8] node[above] {$\Upsilon_{\sigma}^{K}$} (3,0.5,0.5);


\filldraw[pattern=dots, pattern color=red!50]  (0+4,0+0.5,1) -- (0+4,1/6+0.5,1)--(1+4,1/6+0.5,1)--(1+4,0+0.5,1)--(0+4,0+0.5,1);

\filldraw[pattern=dots, pattern color=red!50] (0+4,1/6+0.5,1) -- (0+4,1/6+0.5,.5) --(1+4,1/6+0.5,0.5) -- (1+4,1/6+0.5,1);

\filldraw[pattern=dots, pattern color=red!50] (1+4,1/6+0.5,0.5) -- (1+4,0+0.5,0.5) -- (1+4,0+0.5,1);

\draw[<->] (0.08+4,0+0.5,1.2) -- (1.08+4,0+0.5,1.2);

\draw (0.58+4,0+0.5,1.2) node[below] {$R$};

\draw[<->] (-0.1+4,0+0.5,1) -- (-0.1+4,1/6+0.5,1);

\draw (-0.1+4,0.10+0.5,1) node[left] {$R/K$};

\draw[<->] (1.1+4,0+0.5,1+.1) -- (1.1+4,0+0.5,.5+0.1);

\draw (1.1+4,0+0.5,.9) node[right=0.1cm] {$R/K^{1/2}$};

\end{scope}

\end{tikzpicture}

\caption{The parabolic rescaling effect on the $(x,t)$-variables for the phase $\phi(x,t;\xi)=x_1 \xi_1 + x_2 \xi_2 + t \xi_1^2/\xi_2$.}
\label{fig: spatial rescaling}
\end{figure}
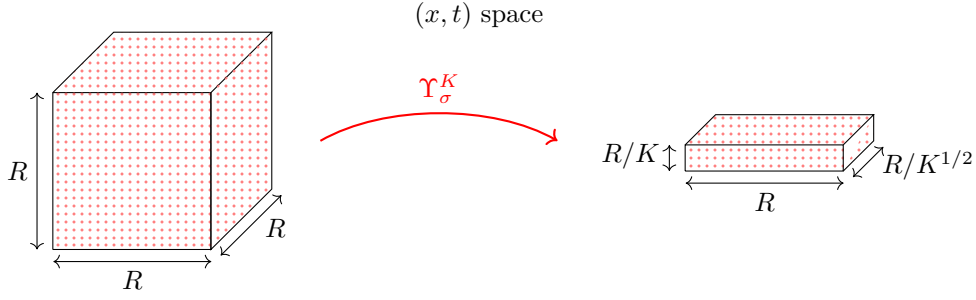
 
 \subsubsection*{The general case} The scaling procedure can be carried out for more general phases, albeit with notable additional complications. In particular, for each $\sigma$ one may identify changes of variable \begin{equation*}
     \Psi_\sigma^K \colon \Gamma_\sigma^K \to \Gamma \qquad \textrm{and} \qquad \Upsilon_{\sigma}^K \colon \R^{n+1} \to \R^{n+1}
 \end{equation*}
 and a suitable Fourier integral operator $\widetilde{\mathcal{F}}^{\lambda/K}_\sigma$ at scale $\lambda/K$ such that
\begin{equation*}
\| \mathcal{F}^{\lambda}_\sigma f \|_{L^p(B_R)} \sim K^{(n+1)/2p} \| \widetilde{\mathcal{F}}^{\lambda/K}_\sigma \tilde{f}_\sigma \|_{L^p(\Upsilon_{\sigma}^K(B_R))},
\end{equation*}
where $\tilde{f}_\sigma$ is defined as in \eqref{rescaled function} and $\Upsilon_{\sigma}^K(B_R)$ is contained in a rectangular box of dimensions $R \times R/K^{1/2} \times \cdots \times R/K^{1/2} \times R/K$. This situation is somewhat more involved than the prototypical case described above, due to the lack of any simple scaling identity \eqref{scaling identity}. In particular, the mapping $\Upsilon_{\sigma}^K$ will often be non-linear and the operators $\tilde{\mathcal{F}}^{\lambda}$ may not agree with the original $\mathcal{F}^{\lambda}$ (although they will of course be related). In order to deal with the latter point, it is necessary to formulate an induction hypothesis which applies to \textit{an entire class of FIOs} which is closed under the relevant rescalings. The (somewhat technically involved) details of the rigorous realisation of this strategy can be found in \cite{BHS}.




\subsection{Applying the induction hypothesis}\label{induction hypothesis subsection} Noting that $R':=R/K \leq R/2$, the (general) radial induction hypothesis implies that
\begin{equation*}
    \| \widetilde{\mathcal{F}}^{\lambda/K}_\sigma \tilde{f}_\sigma \|_{L^p(B_{R/K})} \\
\leq \bar{C}(p,\varepsilon) (R/K)^{\alpha(p) + \varepsilon} \big( \sum_{\substack{\nu \in \Theta_{(R/K)^{-1/2}}}}  \|\widetilde{\mathcal{F}}^{\lambda/K}_\sigma (\tilde{f}_\sigma)_\nu\|_{L^{p}(B_{R/K})}^p \big)^{1/p}
\end{equation*}
for any ball $B_{R/K}$ of radius $R/K$. Take a finitely-overlapping cover of $\Upsilon_{\sigma}^K(B_R)$ by such balls and apply the above inequality to each member of this cover. Taking $p$-powers, summing over each member of the cover and taking $p$-roots, one deduces that
\begin{equation*}
    \| \widetilde{\mathcal{F}}^{\lambda/K}_\sigma \tilde{f}_\sigma \|_{L^p(\Upsilon_{\sigma}^K(B_R))} \\
\lesssim \bar{C}(p,\varepsilon) (R/K)^{\alpha(p) + \varepsilon} \big( \sum_{\substack{\nu \in \Theta_{(R/K)^{-1/2}}}}  \|\widetilde{\mathcal{F}}^{\lambda/K}_\sigma (\tilde{f}_\sigma)_\nu\|_{L^{p}(\Upsilon_{\sigma}^K(B_R))}^p \big)^{1/p}.
\end{equation*}
Applying the rescaling argument to both sides of this inequality yields
\begin{equation*}
\|\mathcal{F}^{\lambda}_\sigma f\|_{L^p(B_R)}  \lesssim \bar{C}(p,\varepsilon) (R/K)^{\alpha(p) + \varepsilon} \big( \sum_{\substack{\nu \in \Theta_{R^{-1/2}} \\ \Gamma_\nu^R \subseteq \Gamma_\sigma^K}}  \|\mathcal{F}^{\lambda}_\sigma f_\nu\|_{L^{p}(B_R)}^p \big)^{1/p}
\end{equation*}
and one may sum this estimate over all $K^{-1/2}$-sectors $\Gamma_\sigma^K$ to obtain
\begin{equation}\label{parabolic scaling + induction}
\big( \sum_{\sigma \in \Theta_{K^{-1/2}}}\|\mathcal{F}^{\lambda}_\sigma f \|_{L^p(B_R)}^p \big)^{1/p} \lesssim \bar{C}(p,\varepsilon)(R/K)^{\alpha(p) + \varepsilon} \big( \sum_{\nu \in \Theta_{R^{-1/2}}}  \|\mathcal{F}^{\lambda}_\nu f\|_{L^{p}(B_R)}^p \big)^{1/p}.
\end{equation}
Finally, by combining \eqref{approximation + BD} with \eqref{parabolic scaling + induction}, it follows that
\begin{equation*}
\| \mathcal{F}^{\lambda} f \|_{L^p(B_R)} \lesssim_{\varepsilon} \bar{C}(p,\varepsilon) K^{-\varepsilon/2}
R^{\alpha(p) + \varepsilon} \big( \sum_{\nu \in \Theta_{R^{-1/2}}}  \|\mathcal{F}^{\lambda}_\nu f\|_{L^{p}(B_R)}^p \big)^{1/p}.
\end{equation*}
If $C_{\varepsilon}$ denotes the implicit constant appearing in the above inequality, then the induction can be closed simply by choosing $K$ large enough so that $C_\varepsilon K^{-\varepsilon/2} \leq 1$.




\appendix

\section{Historical background on the local smoothing conjecture}




\subsection{The euclidean wave equation}

The local smoothing conjecture for the euclidean half-wave propagator $e^{i t \sqrt{-\Delta}}$, that is Conjecture \ref{LS conj euclidean}, was formulated by the third author \cite{Sogge1991} in 1991. Moreover, he showed qualitative existence of the local smoothing phenomenon for $n=2$, showing that there is some $\varepsilon_0 > 0$ such that \eqref{LS conj euclidean equation} holds for $0 < \sigma < \varepsilon_0$ if $p=4$. Note that by interpolation with $L^2$ and $L^\infty$, this also shows that there is $\varepsilon(p)>0$ such that \eqref{LS conj euclidean equation} holds for all $0 < \sigma < \varepsilon(p)$ if $2 < p < \infty$. Shortly after, Mockenhoupt, Seeger and the third author \cite{Mockenhaupt1992} obtained a quantitative estimate at the critical Lebesgue exponent $p=4$ through a square function estimate approach. Further refinements at $p=4$ were later obtained by Bourgain \cite{BourgainSF} and Tao and Vargas \cite{TV2}. In particular, the work of Tao and Vargas established a way to transfer bilinear Fourier restriction estimates into estimates for the square function; thus, the best results via their method are obtained through the sharp bilinear restriction estimates for the cone by Wolff \cite{Wolff2001} (see also the endpoint result of Tao \cite{Tao2001}). In higher dimensions, Mockenhoupt, Seeger and the third author \cite{Mockenhaupt1993} also established existence of local smoothing estimates, although in this case their results are concerned with estimates at $p=\frac{2(n+1)}{n-1}$ rather than at the critical Lebesgue exponent $p=\frac{2n}{n-1}$.

All the initial results discussed in the previous paragraph did not imply sharp estimates in terms of the regularity exponent $\sigma$ for any $2 < p < \infty$. A striking advance was made by Wolff \cite{Wolff2000} in 2000 when he introduced the decoupling inequalities discussed in $\S$\ref{sec:Wolff} and established that in the plane $1/p-$ local smoothing holds for all $p > 74$. His result was later extended to higher dimensions by \L aba and Wolff \cite{Laba2002}. Subsequent works of Garrig\'os and Seeger \cite{Garrigos2009} and Garrig\'os, Seeger and Schlag \cite{Garrigos2010} improved the Lebesgue exponent $p$ in the sharp\footnote{Here the word \textit{sharp} refers to the sharp dependence of the constant in terms of the number of pieces featuring in the decoupling inequality; more precisely, the optimal dependence on $\lambda$ in \eqref{FIO decoupling estimate}.} decoupling inequalities, and therefore the Lebesgue exponent for which Conjecture \ref{LS conj euclidean} holds. The sharp decoupling inequalities were finally established by Bourgain and Demeter \cite{Bourgain2015} in 2015, which imply $1/p-$ local smoothing estimates for all $ \frac{2(n+1)}{n-1} \leq p < \infty$.

In parallel progress obtained via decoupling inequalities, Heo, Nazarov and Seeger \cite{Heo2011} introduced in 2011 a different approach to the problem, which in particular yields local smoothing estimates at the endpoint regularity $\sigma=1/p$ if $\frac{2(n-1)}{n-3} < p < \infty$ for $n \geq 4$. 

Finally, some further progress has been obtained for $n=2$. In 2012, S. Lee and Vargas \cite{Lee2012} proved local smoothing estimates for all $\sigma < \bar{s}_p$ if $p=3$ via a sharp square function estimate in $L^3(\R^2)$. This is the first and only time sharp local smoothing estimates (up to the regularity endpoint) have been obtained in the range $2 < p < \frac{2n}{n-1}$. More recently, J. Lee \cite{Lee} has further improved the square function estimate at $p=4$ using the $L^6(\R^2)$ decoupling inequalities of Bourgain--Demeter \cite{Bourgain2015}, showing that \eqref{LS conj euclidean equation} holds for all $ \sigma < 3/16 $ when $p=4$.

The precise numerology and historical progress on the euclidean local smoothing conjecture have been outlined\footnote{For simplicity, the intermediate progress of Bourgain \textcolor{magenta}{\cite{BourgainSF}} and Tao and Vargas \textcolor{magenta}{\cite{Tao1999}} at $p=4$ has not been sketched in Figure \ref{fig:LS n=2}; see Table \ref{table: LS euclidean n=2} for their contribution to the problem.} in Figure \ref{fig:LS n=2} and Table \ref{table: LS euclidean n=2} for $n=2$ and Figure \ref{fig: LS euclidean high n} and Table \ref{table: LS euclidean high n} for $n \geq 3$.




\subsection{Fourier integral operators}

Shortly after the formulation of the local smoothing conjecture for the euclidean wave equation, Mockenhoupt, Seeger and the third author \cite{Mockenhaupt1993} considered the analogous problem for wave equations on compact manifolds and general classes of Fourier integral operators. They established positive partial results at the critical Lebesgue exponent $p=4$ for $n = 2$, and at the subcritical exponent $p=\frac{2(n+1)}{n-1}$ for $n \geq 3$. In 1997, Minicozzi and the third author \cite{Minicozzi1997} provided examples of compact Riemannian manifolds $(M,g)$ for which the operator $e^{it\sqrt{-\Delta_g}}$ does not demonstrate $1/p-$ local smoothing for $p < \bar{p}_{n, +}$ (see also \cite{BHS}). This revealed a difference in the local smoothing phenomenon between the euclidean and variable-coefficient cases.

The next positive results were obtained by Lee and Seeger in \cite{Lee2013}, where they extended the endpoint regularity results in \cite{Heo2011} to general Fourier integral operators; as in the euclidean case, these results only hold for $n \geq 4$. Except for the question of endpoint regularity, the best known results have recently been obtained by the authors in \cite{BHS}, extending to the variable coefficient case the results of Bourgain and Demeter \cite{Bourgain2015}. Moreover, and as discussed in $\S$\ref{sec:LS}, the authors also showed that their results are best possible in odd dimensions in the general context of Conjecture \ref{LS conj FIO}, although one expects to go beyond these exponents in the even dimensional case and in the case of solutions arising from wave equations on compact manifolds.

The precise numerology and historical progress on Conjectures \ref{LS conj manifold} and \ref{LS conj FIO} have been outlined in Figure \ref{fig: LS FIO} and Table \ref{table: LS FIO}.

\clearpage




\subsection{Figure and Table for the euclidean wave equation for $n=2$}


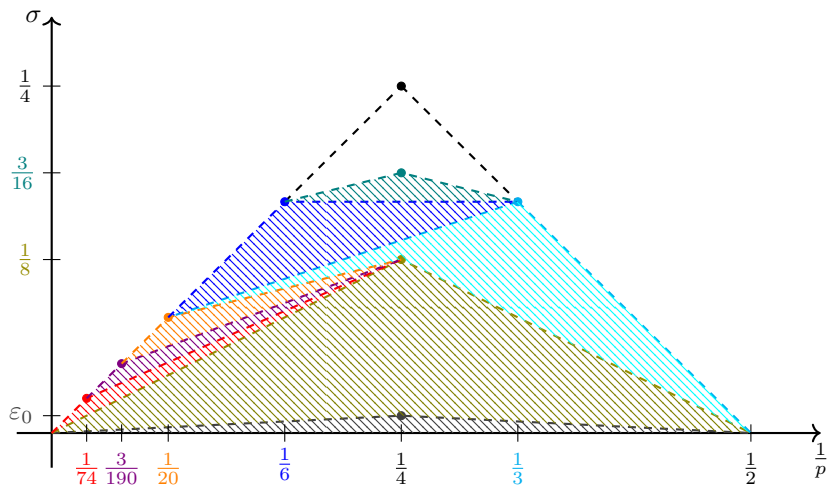
\begin{figure}[b]
\begin{tikzpicture}[scale=2.3] 

\begin{scope}[scale=2]
\draw[thick,->] (-.1,0) -- (2.2,0) node[below] {$ \frac 1 p$};
\draw[thick,->] (0,-.1) -- (0,1.2) node[left] {$ \sigma$};

\draw (.025,1) -- (-.025,1) node[left] {$ \frac{1}{4}$};
\draw (1,-0.025) -- (1,.025) node[below = 0.25cm] {$ \frac{1}{4}$};

\draw (2,-0.025) -- (2,.025) node[below= 0.25cm] {$ \frac{1}{2}$}; 
\node[circle,draw=black, fill=black, inner sep=0pt,minimum size=3pt] (b) at (1,1) {};

\draw[thick, dashed]  (2/3,2/3)  -- (1,1); 

\draw[thick, dashed] (1,1) -- (4/3,2/3);


\draw (.025,1/20) -- (-.025,1/20) node[left] { \textcolor{darkgray}{$ \varepsilon_0$}};
\node[circle,draw=darkgray, fill=darkgray, inner sep=0pt,minimum size=3pt] (b) at (1,1/20) {};

\fill[pattern=north west lines, pattern color=darkgray] (1,1/20) -- (2,0) -- (0,0) -- (1,1/20);

\draw[thick, dashed, color=darkgray] (0,0) -- (1,1/20) -- (2,0) ;


\draw (.025,1/2) -- (-.025,1/2) node[left] { \textcolor{olive}{$ \frac{1}{8}$}};
\node[circle,draw=olive, fill=olive, inner sep=0pt,minimum size=3pt] (b) at (1,1/2) {};

\fill[pattern=north west lines, pattern color=olive] (1,1/2) -- (2,0) -- (1,1/20) -- (0,0) -- (1,1/2);

\draw[thick, dashed, color=olive] (0,0) -- (1,1/2) -- (2,0) ;


\draw (1/10,-0.025) -- (1/10,.025) node[below = 0.25cm] {\textcolor{red}{$ \frac{1}{74}$}};
\node[circle,draw=red, fill=red, inner sep=0pt,minimum size=3pt] (b) at (1/10,1/10) {};

\fill[pattern=north west lines, pattern color=red] (1/10,1/10) -- (1,1/2) -- (0,0) -- (1/10,1/10);

\draw[thick, dashed, color=red] (0,0) -- (1/10,1/10) -- (1,1/2) ;


\draw (1/5,-0.025) -- (1/5,.025) node[below = 0.25cm] {\textcolor{violet}{$ \frac{3}{190}$}};
\node[circle,draw=violet, fill=violet, inner sep=0pt,minimum size=3pt] (b) at (1/5,1/5) {};

\fill[pattern=north west lines, pattern color=violet] (1/5,1/5) -- (1,1/2) -- (1/10,1/10) -- (1/5,1/5);

\draw[thick, dashed, color=violet] (1/10,1/10) -- (1/5,1/5) -- (1,1/2) ;


\draw (1/3,-0.025) -- (1/3,.025) node[below = 0.25cm] {\textcolor{orange}{$ \frac{1}{20}$}};
\node[circle,draw=orange, fill=orange, inner sep=0pt,minimum size=3pt] (b) at (1/3,1/3) {};

\fill[pattern=north west lines, pattern color=orange] (1/3,1/3) -- (1,1/2) -- (1/5,1/5) -- (1/3,1/3);

\draw[thick, dashed, color=orange] (1/5,1/5) -- (1/3,1/3) -- (1,1/2) ;


\draw (4/3, -0.025) -- (4/3, 0.025) node[below=0.25cm] {\textcolor{cyan}{$\frac{1}{3}$}}; 

\node[circle,draw=cyan, fill=cyan, inner sep=0pt,minimum size=3pt] (b) at (4/3,2/3) {};

\draw[thick, dashed, color=cyan] (4/3,4/6) -- (2,0) ;

\draw[thick, dashed, color=cyan] (4/3,4/6) -- (1/3,1/3) ;

\fill[pattern=north west lines, pattern color=cyan] (4/3,2/3) -- (2,0) -- (1,1/2) -- (4/3,2/3);

\fill[pattern=north west lines, pattern color=cyan] (4/3,2/3) -- (1/3,1/3) -- (1,1/2) -- (4/3,2/3);


\draw (2/3,-0.025) -- (2/3,.025);
\draw (2/3, 0) node[below = 0.08cm] {\textcolor{blue}{$ \frac{1}{6}$} };

\node[circle,draw=blue, fill=blue, inner sep=0pt,minimum size=3pt] (b) at (2/3,2/3) {};

\draw[thick, dashed, color=blue] (1/3,1/3) -- (2/3,2/3)  -- (4/3, 2/3) ;

\fill[pattern=north west lines, pattern color=blue] (1/3,1/3) -- (2/3,2/3) -- (4/3,2/3) -- (1/3,1/3);


\draw (-0.025,3/4) -- (.025,3/4);
\draw (0,3/4) node[left = 0.08cm] {\textcolor{teal}{$ \frac{3}{16}$} };

\node[circle,draw=teal, fill=teal, inner sep=0pt,minimum size=3pt] (b) at (1,3/4) {};

\draw[thick, dashed, color=teal] (2/3,2/3) -- (1,3/4)  -- (4/3, 2/3) ;

\fill[pattern=north west lines, pattern color=teal] (2/3,2/3) -- (1, 3/4) -- (4/3,2/3) -- (2/3,2/3);

\end{scope}

\end{tikzpicture}

\caption{Chronological progress on Conjecture \ref{LS conj euclidean equation} for $n=2$. Each new result can be interpolated against the $L^2$ and $L^\infty$ estimates and the previous results in order to yield a new region in the conjectured triangle. The current best results follow from interpolating \textcolor{cyan}{\cite{Lee2012}}, \textcolor{blue}{\cite{Bourgain2015}} and \textcolor{teal}{\cite{Lee}} and the $L^2$ and $L^\infty$ estimates. The white region remains open.}
\label{fig:LS n=2}

\end{figure}


\begin{table}[b]
    \centering \makegapedcells
    \begin{tabular}{|c|c|c|}
    \hline
     & $p$ & $\sigma$  \\ \hline
    \textcolor{darkgray}{S \cite{Sogge}} & $4$ & $\varepsilon_0$ \\ \hline
    {\textcolor{olive}{Mockenhoupt--Seeger--S \cite{Mockenhaupt1992}}} & $4$ & $1/8$ \\ \hline
    {\textcolor{magenta}{Bourgain \cite{BourgainSF}}} & $4$ & $1/8+\varepsilon_0$\\ \hline
    {\textcolor{magenta}{Tao--Vargas \cite{TV2} + Wolff \cite{Wolff2001} }} & $4$ & $1/8+1/88-$ \\ \hline
    {\textcolor{red}{Wolff \cite{Wolff2000}}} & $74+$ & $1/p -$ \\ \hline
    {\textcolor{violet}{Garrig\'os--Seeger \cite{Garrigos2009} }} & $190/3+$ & $1/p-$ \\ \hline
    {\textcolor{orange}{Garrig\'os--Seeger--Schlag \cite{Garrigos2009} }} & $20+$ & $1/p-$ \\ \hline
    {\textcolor{cyan}{S. Lee--Vargas \cite{Lee2012}}} & $3$ & $1/6-$ \\ \hline
    {\textcolor{blue}{Bourgain--Demeter \cite{Bourgain2015}}} & $6$ & $1/6-$ \\ \hline
    {\textcolor{teal}{J. Lee \cite{Lee}}} & $4$ & $3/16-$ \\
    \hline 
    \end{tabular}

\vspace{10pt}

    \caption{Chronological progress on Conjecture \ref{LS conj euclidean} for $n=2$. The notation $p_0+$ means that the estimate \eqref{LS conj euclidean equation} holds for all $p>p_0$, whilst the notation $\sigma_0-$ means that the estimate holds for all $\sigma < \sigma_0$. Otherwise, the equalities for the Lebesgue and regularity exponents are admissible. In the table $\varepsilon_0 > 0$ is a small, unspecified constant. The method of J. Lee can be applied away from the $p = 4$ exponent to give improved estimates in a slightly larger convex region than that given by interpolation; this was pointed out to us by Pavel Zorin--Kranich.}
    \label{table: LS euclidean n=2}
    
    \vspace{1cm}
\end{table}




\newpage

\subsection{Figure and Table for the euclidean wave equation for $n\geq 3$}


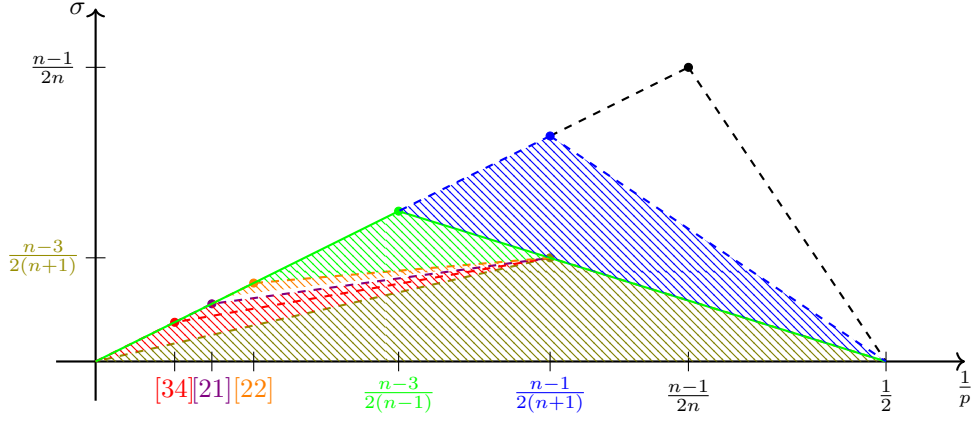
\begin{figure}[b]
\begin{tikzpicture}[scale=2.6] 

\begin{scope}[scale=2]
\draw[thick,->] (-.1,0) -- (2.2,0) node[below] {$ \frac 1 p$};
\draw[thick,->] (0,-.1) -- (0,0.9) node[left] {$ \sigma$};

\draw (.025,3/4) -- (-.025,3/4) node[left] {$ \frac{n-1}{2n}$};
\draw (3/2,-0.025) -- (3/2,.025) node[below = 0.25cm] {$ \frac{n-1}{2n}$};

\draw (2,-0.025) -- (2,.025) node[below= 0.25cm] {$ \frac{1}{2}$}; 
\node[circle,draw=black, fill=black, inner sep=0pt,minimum size=3pt] (b) at (3/2,3/4) {};

\draw[thick, dashed]  (2.3/2,2.3/4)  -- (3/2,3/4); 

\draw[thick, dashed] (3/2,3/4) -- (2,0);


\draw (-0.025, 2.3/8.7) -- (.025, 2.3/8.7);
\draw ( 0, 2.3/8.7) node[left = 0.08cm] {\textcolor{olive}{$ \frac{n-3}{2(n+1)}$} };

\node[circle,draw=olive, fill=olive, inner sep=0pt,minimum size=3pt] (b) at (2.3/2,2.3/8.7) {};

\draw[thick, dashed, color=olive] (2,0) -- (2.3/2,2.3/8.7) -- (0,0) ;

\fill[pattern=north west lines, pattern color=olive] (2.3/2,2.3/8.7) -- (2,0) -- (0,0) -- (2.3/2,2.3/8.7);


\draw (1/5,-0.025) -- (1/5,.025);
\draw (1/5, 0) node[below = 0.08cm] {\textcolor{red}{\cite{Laba2002}} };

\node[circle,draw=red, fill=red, inner sep=0pt,minimum size=3pt] (b) at (1/5,1/10) {};

\draw[thick, dashed, color=red] (0,0) -- (1/5,1/10) -- (2.3/2,2.3/8.7);

\fill[pattern=north west lines, pattern color=red] (0,0) -- (1/5,1/10) -- (2.3/2,2.3/8.7) -- (0,0);


\draw (1/3.4,-0.025) -- (1/3.4,.025);
\draw (1/3.4, 0) node[below = 0.08cm] {\textcolor{violet}{\cite{Garrigos2009}} };

\node[circle,draw=violet, fill=violet, inner sep=0pt,minimum size=3pt] (b) at (1/3.4,1/6.8) {};

\draw[thick, dashed, color=violet] (1/5,1/10) -- (1/3.4,1/6.8) -- (2.3/2,2.3/8.7);

\fill[pattern=north west lines, pattern color=red] (1/5,1/10) -- (1/3.4,1/6.8) -- (2.3/2,2.3/8.7) -- (1/5,1/10);


\draw (1/2.5,-0.025) -- (1/2.5,.025);
\draw (1/2.5, 0) node[below = 0.08cm] {\textcolor{orange}{\cite{Garrigos2010}} };

\node[circle,draw=orange, fill=orange, inner sep=0pt,minimum size=3pt] (b) at (1/2.5,1/5) {};

\draw[thick, dashed, color=orange] (1/3.4,1/6.8) -- (1/2.5,1/5) -- (2.3/2,2.3/8.7);

\fill[pattern=north west lines, pattern color=orange] (1/3.4,1/6.8) -- (1/2.5,1/5) -- (2.3/2,2.3/8.7) -- (1/3,1/6);


\draw (2.3/3,-0.025) -- (2.3/3,.025);
\draw (2.3/3, 0) node[below = 0.08cm] {\textcolor{green}{$ \frac{n-3}{2(n-1)}$} };

\draw[thick, color=green] (2.3/3,2.3/6) -- (0,0) ;

\node[circle,draw=green, fill=green, inner sep=0pt,minimum size=3pt] (b) at (2.3/3,2.3/6) {};

\fill[pattern=north west lines, pattern color=green] (1/2.5,1/5) -- (2.3/3,2.3/6) -- (2.3/2,2.3/8.7) -- (1/2.5,1/5);

\draw[thick, color=green] (2.3/3,2.3/6) -- (2,0) ;


\draw (2.3/2,-0.025) -- (2.3/2,.025);
\draw (2.3/2, 0) node[below = 0.08cm] {\textcolor{blue}{$ \frac{n-1}{2(n+1)}$} };

\fill[pattern=north west lines, pattern color=blue] (2.3/2,2.3/4) -- (2,0) -- (2.3/3,2.3/6) -- (2.3/2,2.3/4);

\draw[thick, dashed, color=blue] (2.3/2,2.3/4) -- (2,0) ;
\draw[thick, dashed, color=blue] (2.3/2,2.3/4) -- (2.3/3,2.3/6)  ;

\node[circle,draw=blue, fill=blue, inner sep=0pt,minimum size=3pt] (b) at (2.3/2,2.3/4) {};

\end{scope}

\end{tikzpicture}

\caption{Chronological progress on Conjecture \ref{LS conj euclidean} for high dimensions ($n \geq 5$). Each new result can be interpolated against the $L^2$ and $L^\infty$ estimates and the previous results in order to yield a new region in the conjectured triangle. The best known results follow from interpolation between \textcolor{blue}{\cite{Bourgain2015}} and the $L^2$ and $L^\infty$ estimates, together with the strengthened results \textcolor{green}{\cite{Heo2011}} at the regularity endpoint $\sigma=1/p$ if $p > \frac{n-3}{2(n-1)}$. The white region remains open.}
\label{fig: LS euclidean high n}
\vspace{1cm}
\end{figure}


\begin{table}[b]
    \centering \makegapedcells
    \begin{tabular}{|l|c|c|c|c|c|c|}
    \hline
     & \multicolumn{2}{|c|}{$n=3$} & \multicolumn{2}{|c|}{$n = 4$} & \multicolumn{2}{|c|}{$n \geq 5$} \\ \hline
     & $p$ & $\sigma$ & $p$ & $\sigma$ & $p$ & $\sigma$ \\ \hline
    {\textcolor{olive}{Mockenhoupt--Seeger--S \cite{Mockenhaupt1992}}} & $4$ & $1/2p-$ & $10/3$ & $1/3p$ & $\frac{2(n+1)}{n-1}$ & $\frac{n-3}{n-1}\frac{1}{p}$ \\ \hline
    {\textcolor{red}{\L aba--Wolff \cite{Laba2002}}} & $18+$ & $1/p-$  & $8.4+$ & $1/p-$ & $ \frac{2(n+1)}{n-3}+$ & $1/p-$  \\ \hline
    {\textcolor{violet}{Garrig\'os--Seeger \cite{Garrigos2009}}} & $15+$ & $1/p-$ & $7.28+$ & $1/p-$ & $ \frac{2(n-1)(n+3)}{(n+1)(n-3)}+$ & $1/p-$ \\ \hline
    {\textcolor{orange}{Garrig\'os--Seeger--Schlag \cite{Garrigos2010}}} & $9+$ & $1/p-$ & $5.6+$ & $1/p-$ &  $ \frac{2n(n+3)}{(n-1)(n-2)}+$ & $1/p-$ \\ \hline
    {\textcolor{green}{Heo--Nazarov--Seeger \cite{Heo2011}}} &  &  & $6$ & $1/6$ & $ \frac{2(n-1)}{n-3}+$ & $1/p$  \\ \hline
    {\textcolor{blue}{Bourgain--Demeter \cite{Bourgain2015}}} & $4$ & $1/4-$ & $10/3$ & $3/10-$ & $\frac{2(n+1)}{n-1}$ & $1/p-$  \\ \hline
\end{tabular}

\vspace{10pt}

    \caption{Chronological progress on Conjecture \ref{LS conj euclidean} for $n \geq 3$. The notation $+$ and $-$ is used in a similar fashion to Table \ref{table: LS euclidean n=2}.}
    \label{table: LS euclidean high n}
    \vspace{3.5cm}
\end{table}


%


\newpage

\subsection{Figure and Table for Fourier integrals}


\begin{figure}[b]
\begin{tikzpicture}[scale=2.6] 

\begin{scope}[scale=2]
\draw[thick,->] (-.1,0) -- (2.2,0) node[below] {$ \frac 1 p$};
\draw[thick,->] (0,-.1) -- (0,0.9) node[left] {$ \sigma$};

\draw (.025,3/4) -- (-.025,3/4) node[left] {$ \frac{n-1}{2n}$};
\draw (3/2,-0.025) -- (3/2,.025) node[below = 0.25cm] {$ \frac{n-1}{2n}$};

\draw (2,-0.025) -- (2,.025) node[below= 0.25cm] {$ \frac{1}{2}$}; 
\node[circle,draw=black, fill=black, inner sep=0pt,minimum size=3pt] (b) at (3/2,3/4) {};

\draw[thick, dashed]  (2.3/2,2.3/4)  -- (3/2,3/4); 

\draw[thick, dashed] (3/2,3/4) -- (2,0);


\draw (-0.025, 2.3/8.7) -- (.025, 2.3/8.7);
\draw ( 0, 2.3/8.7) node[left = 0.08cm] {\textcolor{olive}{$ \frac{n-3}{2(n+1)}$} };

\node[circle,draw=olive, fill=olive, inner sep=0pt,minimum size=3pt] (b) at (2.3/2,2.3/8.7) {};

\draw[thick, dashed, color=olive] (2,0) -- (2.3/2,2.3/8.7) -- (0,0) ;

\fill[pattern=north west lines, pattern color=olive] (2.3/2,2.3/8.7) -- (2,0) -- (0,0) -- (2.3/2,2.3/8.7);


\draw (2.3/3,-0.025) -- (2.3/3,.025);
\draw (2.3/3, 0) node[below = 0.08cm] {\textcolor{green}{$ \frac{n-3}{2(n-1)}$} };

\draw[thick, color=green] (2.3/3,2.3/6) -- (0,0) ;

\node[circle,draw=green, fill=green, inner sep=0pt,minimum size=3pt] (b) at (2.3/3,2.3/6) {};

\fill[pattern=north west lines, pattern color=green] (0,0) -- (2.3/3,2.3/6) -- (2.3/2,2.3/8.7) -- (0,0);

\draw[thick, color=green] (2.3/3,2.3/6) -- (2,0) ;


\draw (2.3/2,-0.025) -- (2.3/2,.025);
\draw (2.3/2, 0) node[below = 0.08cm] {\textcolor{blue}{$ \frac{n-1}{2(n+1)}$} };

\fill[pattern=north west lines, pattern color=blue] (2.3/2,2.3/4) -- (2,0) -- (2.3/3,2.3/6) -- (2.3/2,2.3/4);

\draw[thick, dashed, color=blue] (2.3/2,2.3/4) -- (2,0) ;
\draw[thick, dashed, color=blue] (2.3/2,2.3/4) -- (2.3/3,2.3/6)  ;

\node[circle,draw=blue, fill=blue, inner sep=0pt,minimum size=3pt] (b) at (2.3/2,2.3/4) {};


\draw[thick, dashed, color=red] (2.63/2,2.63/4) -- (2,0) ;

\fill[pattern=north west lines, pattern color=red] (2.63/2,2.63/4) -- (2,0) -- (3/2,3/4) -- (2.63/2,2.63/4);

\draw (2.63/2, -0.025) -- (2.63/2, 0.025) node[below=0.25cm] {\textcolor{red}{$\frac{1}{\bar{p}_{n, +}}$}}; 

\node[circle,draw=red, fill=red, inner sep=0pt,minimum size=3pt] (b) at (2.63/2,2.63/4) {};
\draw (2.63/2, .025) -- (2.63/2,-.025) ;

\end{scope}

\end{tikzpicture}

\caption{Chronological progress on Conjecture \ref{LS conj manifold} in high dimensions $(n \geq 4)$. The \textcolor{red}{red} region is inadmissible. The best known results follow from interpolation between \textcolor{blue}{\cite{BHS}} and the $L^2$ and $L^\infty$ estimates, together with the strengthened results \textcolor{green}{\cite{Lee2013}} at the regularity endpoint $\sigma=1/p$ if $p > \frac{n-3}{2(n-1)}$. The white region remains open. In the case of Conjecture \ref{LS conj FIO} the red region extends to $p=\bar{p}_n$ and there is no white open region in odd dimensions.}
\label{fig: LS FIO}
\vspace{1cm}
\end{figure}
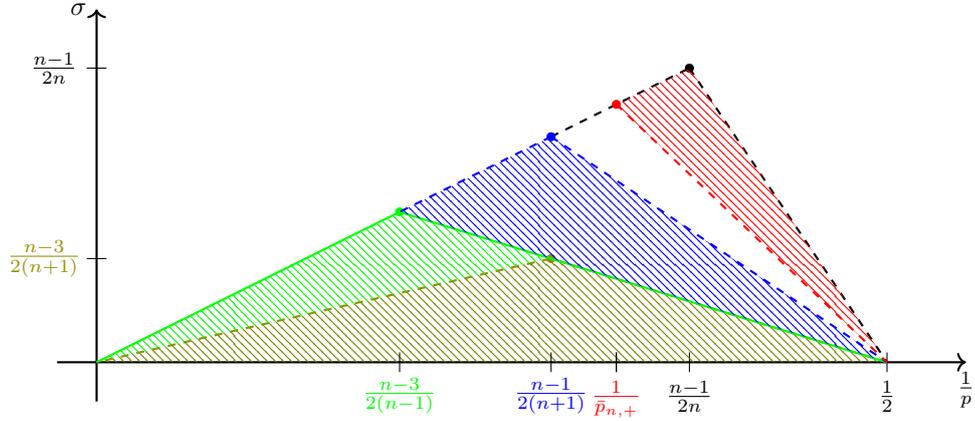


\begin{table}[b]
    \centering \makegapedcells
    \begin{tabular}{|l|c|c|c|c|c|c|}
    \hline
     & \multicolumn{2}{|c|}{$n=2$}& \multicolumn{2}{|c|}{$n=3$} & \multicolumn{2}{|c|}{$n \geq 4$}  \\ \hline
     & $p$ & $\sigma$ & $p$ & $\sigma$ & $p$ & $\sigma$  \\ \hline
    {\textcolor{olive}{Mockenhoupt--Seeger--S \cite{Mockenhaupt1992}}} & $4$ & $1/2p-$ & $4$ & $1/2p-$ &  $\frac{2(n+1)}{n-1}$ & $\frac{n-3}{n-1}\frac{1}{p}$ \\ \hline
    {\textcolor{green}{Lee--Seeger \cite{Lee2013}}} & & &  &  &  $ \frac{2(n-1)}{n-3} +$ & $1/p$  \\ \hline
    {\textcolor{blue}{B--H--S \cite{BHS}}} & $6$ & $1/6-$ & $4$ & $1/4-$ &  $\frac{2(n+1)}{n-1}$ & $1/p-$  \\ \hline
\end{tabular}

\vspace{10pt}

    \caption{Chronological progress on Conjectures \ref{LS conj manifold} and \ref{LS conj FIO} for $n \geq 2$. The notation $+$ and $-$ is used in a similar fashion to Table \ref{table: LS euclidean n=2}.}
    \label{table: LS FIO}
    \vspace{5.5cm}
\end{table}

\newpage




\bibliography{Reference}
\bibliographystyle{amsplain}

\end{document}